\documentclass[a4paper,12pt,twoside]{article}

\usepackage{amsmath, amssymb, amsfonts}
\usepackage[amsmath, thmmarks]{ntheorem}
\usepackage{graphicx}
\usepackage{paper}
\usepackage{mathrsfs}
\usepackage{tikz}
\usepackage{tikz-cd}
\usepackage{float}
\usepackage{shuffle}
\mathtoolsset{showonlyrefs=true}
\newcommand{\MT}{\mathcal{MT}}
\newcommand{\B}{\mathrm{B}}
\newcommand{\dR}{\mathrm{dR}}
\newcommand{\Lm}{\mathbb{L}^{\mathfrak{m}}}
\newcommand{\pia}{\pi_\mathfrak{a}}
\newcommand{\kk}{\mathbf{k}}
\newcommand{\dch}{\mathrm{dch}}
\renewcommand{\Im}{I^{\mathfrak{m}}}
\newcommand{\Il}{I^{\mathfrak{l}}}

\newcommand{\Ia}{I^{\mathfrak{a}}}
\newcommand{\zm}{\zeta^{\mathfrak{m}}}
\newcommand{\za}{\zeta^{\mathfrak{a}}}

\newcommand{\zl}{\zeta^{\mathfrak{l}}}

\newcommand{\sm}{s^{\mathfrak{m}}}
\renewcommand{\sl}{s^{\mathfrak{l}}}
\newcommand{\sa}{s^{\mathfrak{a}}}
\newcommand{\bsm}{\bar{s}^{\mathfrak{m}}}
\newcommand{\bsl}{\bar{s}^{\mathfrak{l}}}
\newcommand{\bsa}{\bar{s}^{\mathfrak{a}}}
\newcommand{\per}{\mathrm{per}}
\renewcommand{\H}{\mathcal{H}}
\renewcommand{\L}{\mathcal{L}}
\newcommand{\A}{\mathcal{A}}
\renewcommand{\P}{\mathcal{P}}
\newcommand{\tikznode}[2]{%
  \tikz[remember picture,baseline=(#1.base)] \node (#1) [inner sep=0pt] {$#2$};%
}

\title{New Hoffman-type generators for alternating multiple zeta values}
\author{Wataru Machida}

\begin{document}
\maketitle

\begin{firstfootnote}
    \ Graduate School of Mathematical Sciences, The University of Tokyo, 8-1 Komaba 3-chome, Meguro-ku, Tokyo, 153-8914, Japan

    \ E-mail address: \texttt{wmachida@ms.u-tokyo.ac.jp}

\end{firstfootnote}

\begin{abstract}
    We introduce multiple $s$ values and prove that every alternating 
    multiple zeta value is a $\Q$-linear combination of multiple $s$ values
    $s(k_1, \dots, k_r)$ with $k_i \in \{1, 2\}$.
    We also study when certain multiple $s$ values can be expressed
    as $\Q$-linear combinations of multiple zeta values.
\end{abstract}

\section{Introduction}
The multiple zeta value (MZV) is defined by the series
\[
\zeta(\kk) = \sum_{0 < n_1 < \dots < n_r} \frac{1}{n_1^{k_1}\cdots n_r^{k_r}},
\]
where $\kk = (k_1, \dots, k_r) \in \Z_{>0}^r$ is an index with $k_r \ge 2$.
For an index $\kk$, we denote its weight $k_1 + \cdots + k_r$ and depth $r$ 
by $\wt\kk$ and $\dep\kk$, respectively.
Let $\mathcal{Z}^1_N$ denote the $\Q$-vector space spanned by MZVs of weight $N$.
Zagier (\cite{MR1341859}) conjectured that
\[
\dim_\Q \mathcal{Z}^1_N = d_N,
\]
where $d_N$ is defined recursively by
\[
d_N = d_{N-2} + d_{N-3}, \qquad d_0 = 1,\ d_1 = 0,\ d_2 = 1.
\]
Terasoma (\cite{MR1918675}) and Deligne--Goncharov (\cite{MR2136480}) independently proved the upper bound
\[
\dim_\Q \mathcal{Z}^1_N \le d_N.
\]
On the other hand, Hoffman (\cite{hoffmanconj}) conjectured that every multiple zeta value is a $\Q$-linear combination of
\[
\{\zeta(k_1, \dots, k_r) \mid k_i \in \{2, 3\}\},
\]
which is known as the Hoffman basis.
The number of such indices of weight $N$ coincides with $d_N$.
This conjecture was proved by Brown (\cite{MR2993755}) using the theory of motivic multiple zeta values.

Alternating multiple zeta values (AMZVs), also known as Euler sums, 
are a natural generalization of MZVs.
For an index $\kk = (k_1, \dots, k_r) \in \Z_{>0}^r$ and signs $\epsilon_1, \dots, \epsilon_r \in \{\pm1\}$ 
satisfying $(k_r, \epsilon_r) \ne (1,1)$, the alternating multiple zeta value is defined by
\[
\zeta\binom{k_1, \dots, k_r}{\epsilon_1, \dots, \epsilon_r}
= \sum_{0<n_1<\cdots<n_r} \frac{\epsilon_1^{n_1}\cdots\epsilon_r^{n_r}}{n_1^{k_1}\cdots n_r^{k_r}}.
\]
When $\epsilon_1 = \cdots = \epsilon_r = 1$, this reduces to a multiple zeta value.
Let $\mathcal{Z}^2_N$ denote the $\Q$-vector space spanned by AMZVs of weight $N$.
Broadhurst (\cite{broadhurst1996enumerationirreduciblekfoldeuler}) conjectured that
\[
\dim_\Q \mathcal{Z}^2_N = F_{N+1},
\]
where $F_N$ is the Fibonacci sequence defined by
\[
F_N = F_{N-1}+F_{N-2}, \qquad F_1 = F_2 = 1.
\]
Deligne (\cite{MR2737978}) proved the upper bound
\[
\dim_\Q \mathcal{Z}^2_N \le F_{N+1}.
\]
Charlton established an analogue of Brown's theorem for AMZVs in \cite{MR4906315}.
More precisely, every alternating multiple zeta value is a $\Q$-linear combination of
\[
\{t^{\shuffle, 0}(k_1, \dots, k_r) \mid k_i \in \{1, 2\}\},
\]
where $t^{\shuffle, 0}$ denotes the shuffle regularization of multiple $t$ values
\footnote{For simplicity, we only state the case of shuffle regularization. 
    Charlton also proved analogous results for certain stuffle regularizations.}.
Multiple $t$ values (M$t$Vs) were introduced by Hoffman (\cite{MR4013729}), and are defined,
for an index $\kk = (k_1, \dots, k_r) \in \Z_{>0}^r$ with $k_r \ge 2$, by
\[
t(\kk)
= \sum_{\substack{0 < n_1 < \cdots < n_r \\ n_i: \text{odd}}} \frac1{n_1^{k_1}\cdots n_r^{k_r}}.
\]
A simple computation shows that
\[
t(\kk) = 2^{-\dep\kk} \sum_{\epsilon_i\in\{\pm1\}} \epsilon_1\cdots\epsilon_r\,
\zeta\binom{k_1, \dots, k_r}{\epsilon_1, \dots, \epsilon_r}.
\]
Hence, every multiple $t$ value is a $\Q$-linear combination of AMZVs.
Such a regularization is necessary, since M$t$Vs diverge whenever the last index is 1.
It is therefore natural to ask whether there exist Hoffman-type generators for AMZVs 
that are defined by convergent series.

To answer this question, we introduce a new family of AMZVs, called multiple $s$ values (M$s$Vs).
For an index $\kk = (k_1, \dots, k_r) \in \Z_{>0}^r$, the multiple $s$ value is defined by
\[
s(\kk) = \sum_{\substack{\epsilon_i \in \{\pm1\}\ (k_i>1)\\ \epsilon_i = -1\ (k_i=1)}}
\zeta\binom{k_1, \dots, k_r}{\epsilon_1, \dots, \epsilon_r}.
\]
Unlike M$t$Vs, the signs corresponding to the indices $k_i = 1$ are fixed to be $-1$.
Consequently, M$s$Vs converge even when the last index is 1, and hence require no regularization.
Our main result is the following analogue of Charlton's theorem for M$s$Vs:
\begin{thm} \label{thm:sbasis}
Every alternating multiple zeta value is a $\Q$-linear combination of
\[
\{s(k_1, \dots, k_r) \mid k_i \in \{1,2\}\}.
\]
\end{thm}
We prove this theorem following the approaches of
\cite[Theorem 1.1]{MR2993755}, \cite[Theorem 2]{MR4377309}, and \cite[Theorem 1.3]{MR4906315}.
First, we define the motivic version of M$s$V, which we denote by $\sm(\kk)$.
It is an element of the graded $\Q$-vector space $\H = \oplus_{N \ge 0} \H_N$ of motivic AMZVs.
Since AMZVs are periods of mixed Tate motives over $\Z[1/2]$,
$\H$ is equipped with the coaction of the motivic Galois group
and the coradical filtration $C_d\H$.
We prove in Proposition~\ref{prop:deriv} that for $k_i \in \{1, 2\}$ 
with $\#\{i \mid k_i = 1\} = d$, we have
\[
\sm(k_1, \dots, k_r) \in C_d\H.
\]
By the theory of mixed Tate motives, we have an isomorphism 
\[
\partial_{N, d} \colon \gr_d\H_N \stackrel{\sim}{\longrightarrow} \bigoplus_{1 \le 2r+1 \le N} \gr_{d-1}\H_{N-(2r+1)},
\]
where $\gr_d\H = C_d\H/C_{d-1}\H$.
Using this isomorphism, we prove by induction on $d$ that 
\[
\{\sm(k_1, \dots, k_r) \bmod C_{d-1}\H_N \mid k_i \in \{1, 2\},\; k_1+\dots+k_r = N,\; \#\{i \mid k_i = 1\} = d\}
\]
is a $\Q$-basis of $\gr_d\H_N$ (Theorem~\ref{thm:grbasis}).
In order to compute the image of $\sm(k_1, \dots, k_r)$ for $k_i \in \{1, 2\}$ 
under the isomorphism $\partial_{N, d}$, we need certain relations for 
motivic M$s$Vs.
We prove these relations first for (non-motivic) M$s$Vs and then lift them to the motivic setting.
Applying the period map $\H \to \C$ to the above result yields Theorem~\ref{thm:sbasis}.

The paper is organized as follows.
In Section 2, we introduce M$s$Vs and their analogues, called multiple $\bar{s}$ values (M$\bar{s}$Vs).
In Section 3, we prove several relations for M$s$Vs and M$\bar{s}$Vs.
In Section 4, we introduce the motivic versions of M$s$Vs and M$\bar{s}$Vs,
and compute the coaction induced on $\gr_d\H$ for $\sm(k_1, \dots, k_r)$ with $k_i \in \{1, 2\}$.
In Section 5, we lift the relations from Section 3 to the motivic setting
and prove Theorem~\ref{thm:sbasis}.
In Section 6, we study unramified M$s$Vs using Glanois's descent criterion.

\medbreak
\noindent{\bfseries Acknowledgment}\quad
The author would like to express his sincere gratitude to Professor Yoichi Mieda
for his invaluable guidance and insightful suggestions.
The author would also like to thank Professor Nobuo Sato for many helpful discussions
and valuable insights.
This work was supported by the WINGS-FMSP Program at the Graduate School of Mathematical Sciences, 
the University of Tokyo.

\section{Multiple $s$ values and multiple $\bar{s}$ values}
In this section, we introduce a new family of alternating multiple zeta values, 
called multiple $s$ values and multiple $\bar{s}$ values.
Recall that, for an index $\kk = (k_1, \dots, k_r) \in \Z_{>0}^r$ and signs $\epsilon_1, \dots, \epsilon_r \in \{\pm1\}$ 
with $(k_r, \epsilon_r) \ne (1,1)$, the alternating multiple zeta value is defined by
\[
\zeta\binom{k_1, \dots, k_r}{\epsilon_1, \dots, \epsilon_r}
= \sum_{0<n_1<\cdots<n_r} \frac{\epsilon_1^{n_1}\cdots\epsilon_r^{n_r}}{n_1^{k_1}\cdots n_r^{k_r}}.
\]
To simplify notation, we often write $k_i$ for $\binom{k_i}{1}$ and
$\overline{k}_i$ for $\binom{k_i}{-1}$.
For example, 
\[
\zeta(2, \overline{1}) = \zeta\binom{2, 1}{1, -1}.
\]
We also write $\{k\}^n$ for the index $\underbrace{k,\dots,k}_{n}$.
\begin{dfn}
    For an index $\kk = (k_1, \dots, k_r) \in \Z_{>0}^r$, the multiple $s$ value $s(\kk)$ is defined by
    \[
    s(\kk) = \sum_{\substack{\epsilon_i \in \{\pm 1\} \ (k_i > 1) \\ \epsilon_i = -1 \ (k_i = 1)}} 
    \zeta\binom{k_1, \dots, k_r}{\epsilon_1, \dots, \epsilon_r}.
    \]
    The multiple $\bar{s}$ value $\bar{s}(\kk)$ is defined by
    \[
    \bar{s}(\kk) = \sum_{\substack{\epsilon_i \in \{\pm 1\} \ (k_i > 1) \\ \epsilon_i = -1 \ (k_i = 1)}} 
    \epsilon_1\cdots\epsilon_r\, \zeta\binom{k_1, \dots, k_r}{\epsilon_1, \dots, \epsilon_r}.
    \]
\end{dfn}
If every $k_i$ is greater than 1, then we have
\begin{align}
    s(\kk) 
    & = \sum_{\epsilon_i \in \{\pm 1\}} 
        \zeta\binom{k_1, \dots, k_r}{\epsilon_1, \dots, \epsilon_r} \\
    & = \sum_{\epsilon_i \in \{\pm 1\}} 
        \sum_{0 < n_1 < \dots < n_r} \frac{\epsilon_1^{n_1}\cdots\epsilon_r^{n_r}}{n_1^{k_1}\cdots n_r^{k_r}} \\
    & = \sum_{0 < n_1 < \dots < n_r} \frac{(1+(-1)^{n_1})\cdots(1+(-1)^{n_r})}{n_1^{k_1}\cdots n_r^{k_r}} \\
    & = 2^{\dep\kk} \sum_{\substack{0 < n_1 < \dots < n_r \\ n_i: \text{even}}} \frac{1}{n_1^{k_1}\cdots n_r^{k_r}} \\
    & = 2^{\dep\kk - \wt\kk}\, \zeta(\kk).
\end{align}
In particular, 
\begin{equation}
    s(\{2\}^r) = 2^{-r}\, \zeta(\{2\}^r) = 2^{-r} \frac{\pi^{2r}}{(2r+1)!}. \label{eq:sz}
\end{equation}
Similarly, under the same assumption, we have
\begin{align}
    \bar{s}(\kk)
    & = \sum_{\epsilon_i \in \{\pm 1\}} \epsilon_1 \cdots \epsilon_r\,
        \zeta\binom{k_1, \dots, k_r}{\epsilon_1, \dots, \epsilon_r} \\
    & = \sum_{\epsilon_i \in \{\pm 1\}} 
        \sum_{0 < n_1 < \dots < n_r} \frac{\epsilon_1^{n_1-1}\cdots\epsilon_r^{n_r-1}}{n_1^{k_1}\cdots n_r^{k_r}} \\
    & = \sum_{0 < n_1 < \dots < n_r} \frac{(1+(-1)^{n_1-1})\cdots(1+(-1)^{n_r-1})}{n_1^{k_1}\cdots n_r^{k_r}} \\
    & = 2^{\dep\kk} \sum_{\substack{0 < n_1 < \dots < n_r \\ n_i: \text{odd}}} \frac{1}{n_1^{k_1}\cdots n_r^{k_r}} \\
    & = 2^{\dep\kk}\, t(\kk).
\end{align}
In particular, 
\begin{equation}
    \bar{s}(\{2\}^r) = 2^r\, t(\{2\}^r) = 2^{-r} \frac{\pi^{2r}}{(2r)!}, \label{eq:bsz}
\end{equation}
where the second equality follows from the identity
\[
t(\{2\}^r) = \frac{\pi^{2r}}{2^{2r}(2r)!}
\]
due to Hoffman (\cite{MR4013729}).

\section{Evaluation of $s(\{2\}^a, 1, \{2\}^b)$ and $\bar{s}(\{2\}^a, 1, \{2\}^b)$}
In this section, we evaluate $s(\{2\}^a, 1, \{2\}^b)$ and $\bar{s}(\{2\}^a, 1, \{2\}^b)$ for $a, b \ge 0$.
These identities play an essential role in the proof of Theorem~\ref{thm:sbasis}.
\begin{thm} \label{thm:sidentity}
    Let $a$ and $b$ be non-negative integers.
    Then we have
    \[
    s(\{2\}^a, 1, \{2\}^b)
        = \sum_{r=0}^{a+b} \frac{(-1)^r}{2^r} \alpha_r^{a, b}\, \zeta(\overline{2r+1})\, s(\{2\}^{a+b-r}),
    \]
    where $\alpha_r^{a, b}$ is the rational number defined by
    \[
    \alpha_r^{a, b} = - \frac{2^{2r}}{2^{2r}-1} \binom{2r}{2a+1} + \binom{2r}{2b}.
    \]
    For $r=0$, the term involving $\frac{2^{2r}}{2^{2r}-1}$ is understood to be zero.
\end{thm}

\begin{thm} \label{thm:bsidentity}
    Let $a$ and $b$ be non-negative integers.
    Then we have
    \[
    \bar{s}(\{2\}^a, 1, \{2\}^b)
        = \sum_{r=0}^{a+b} \frac{(-1)^r}{2^r} \beta_r^{a, b}\, \zeta(\overline{2r+1})\, \bar{s}(\{2\}^{a+b-r}),
    \]
    where $\beta_r^{a, b}$ is the rational number defined by
    \[
    \beta_r^{a, b} =
        \begin{dcases}
            \frac{2^{2r}}{2^{2r}-1} \binom{2r}{2a} + \binom{2r}{2b} & \text{if} \ a \ge 1, \\
            -2^{2r+1} + \delta_{r, b} & \text{if} \ a = 0.\\
        \end{dcases}
    \]
    For $r=0$, the term involving $\frac{2^{2r}}{2^{2r}-1}$ is understood to be zero.
\end{thm}

We prove these theorems following the approach of \cite{MR2993756,MR4377309}.
First, we provide the proof of Theorem~\ref{thm:sidentity}.
Let $F(x, y)$ and $\widehat{F}(x, y)$ be generating functions
\begin{align}
    F(x, y) & = \sum_{a, b \ge 0} (-2)^{a+b}\, s(\{2\}^a, 1, \{2\}^b)\, x^{2a+1} y^{2b}, \\
    \widehat{F}(x, y) & = \sum_{a, b \ge 0} (-2)^{a+b} \sum_{r=0}^{a+b} \frac{(-1)^r}{2^r} 
        \alpha_r^{a, b}\, \zeta(\overline{2r+1})\, s(\{2\}^{a+b-r})\, x^{2a+1}y^{2b}.
\end{align}
Our goal is to show that $F(x, y) = \widehat{F}(x, y)$ 
for complex numbers $x$ and $y$ satisfying $|x| + |y| < 1$.

\begin{prop} \label{prop:F}
    We have
    \[
    F(x, y) = x \frac{\sin \pi y}{\pi y} \sum_{k > 0}
    \frac{(1-x)_{k-1}(1+x)_{k-1}}{(1-y)_k (1+y)_k} \qty(\frac{k}{2} - \frac{k^2 - x^2}{2k+1})
    - x \frac{\sin \pi y}{\pi y},
    \]
    where $(a)_k = a(a+1)\cdots(a+k-1)$ is the ascending Pochhammer symbol.
\end{prop}

\begin{prf}
    By the definition of M$s$Vs, we have
    \begin{align}
        s(\{2\}^a, 1, \{2\}^b) 
        & = \sum_{\epsilon_i, \epsilon'_j \in \{\pm 1\}} 
            \zeta\binom{2, \dots, 2, 1, 2, \dots, 2}{\epsilon_1, \dots, \epsilon_a, -1, \epsilon'_1, \dots, \epsilon'_b} \\
        & = \sum_{0 < n_1 < \dots < n_a < k < m_1 < \dots < m_b} \prod_{i=1}^{a}\frac{1+(-1)^{n_i}}{n_i^2}
            \frac{(-1)^k}{k} \prod_{j=1}^{b}\frac{1+(-1)^{m_j}}{m_j^2} \\
        & = 2^{a+b} \sum_{\substack{0 < n_1 < \dots < n_a < k < m_1 < \dots < m_b \\ n_i, m_j: \text{even}}}
            \frac{1}{n_1^2 \cdots n_a^2} \frac{(-1)^k}{k} \frac{1}{m_1^2 \cdots m_b^2} \\
        & \begin{aligned}
        & = 2^{-(a+b)} \sum_{0 < n_1 < \dots < n_a < k < m_1 < \dots < m_b}
            \frac{1}{n_1^2 \cdots n_a^2} \frac{1}{m_1^2 \cdots m_b^2} \qty(\frac{1}{2k} - \frac{1}{2k-1}) \\
        & \qquad - \delta_{b > 0}\ 2^{-(a+b)} \sum_{0 < n_1 < \dots < n_a < k = m_1 < \dots < m_b}
            \frac{1}{n_1^2 \cdots n_a^2} \frac{1}{m_1^2 \cdots m_b^2} \frac{1}{2k-1}. \label{eq:sseries}
        \end{aligned}
    \end{align}
    Hence, we obtain
    \begin{align}
        F(x, y) 
        & = \sum_{a, b \ge 0} (-2)^{a+b}\, s(\{2\}^a, 1, \{2\}^b)\, x^{2a+1} y^{2b} \\
        & = x \sum_{k > 0} \prod_{0 < n < k} \qty(1-\frac{x^2}{n^2})
            \prod_{m > k} \qty(1-\frac{y^2}{m^2}) \qty(\frac{1}{2k} - \frac{1}{2k-1}\qty(1-\frac{y^2}{k^2})) \label{eq:Fprod}\\
        & = x \frac{\sin \pi y}{\pi y} \sum_{k > 0}
            \frac{(1-x)_{k-1}(1+x)_{k-1}}{(1-y)_k (1+y)_k} \qty(\frac{k}{2} - \frac{k^2 - y^2}{2k-1}).
    \end{align}
    Since for $k > 1$, we have
    \[
    \frac{(1-x)_{k-1}(1+x)_{k-1}}{(1-y)_k (1+y)_k} \frac{k^2 - y^2}{2k-1}
    = \frac{(1-x)_{k-2}(1+x)_{k-2}}{(1-y)_{k-1} (1+y)_{k-1}} \frac{(k-1)^2 - x^2}{2k-1},
    \]
    shifting the index of the second term gives
    \[
    F(x, y) = x \frac{\sin \pi y}{\pi y} \sum_{k > 0}
        \frac{(1-x)_{k-1}(1+x)_{k-1}}{(1-y)_k (1+y)_k} \qty(\frac{k}{2} - \frac{k^2 - x^2}{2k+1})
        - x \frac{\sin \pi y}{\pi y},
    \]
    as claimed.
\end{prf}

\begin{dfn}
    For $|z| < 1$, we define the functions $A(z)$ and $B(z)$ by the power series
    \[
    A(z) = \sum_{r \ge 1} \zeta(2r+1)\, z^{2r}, \quad 
    B(z) = \sum_{r \ge 1} (1-2^{-2r})\, \zeta(2r+1)\, z^{2r} = - \sum_{r \ge 1} \zeta(\overline{2r+1})\, z^{2r}.
    \]
    These series converge only for $|z| < 1$, but they admit analytic continuations to 
    meromorphic functions with simple poles at $\Z \setminus \{0\}$.
    Indeed, they can be also written as
    \[
    A(z) = \sum_{n \ge 1} \frac{z^{2}}{n(n^2 - z^2)}, \quad 
    B(z) = \sum_{n \ge 1} (-1)^{n-1}\frac{z^{2}}{n(n^2 - z^2)}.
    \]
\end{dfn}

\begin{rem} \label{rem:AB}
    These functions can be expressed in terms of the digamma function $\psi(z) = \Gamma'(z)/\Gamma(z)$ as
    \begin{align}
        A(z) & = \psi(1) - \frac{1}{2}\, \qty(\psi(1+z) + \psi(1-z)), \label{eq:digamma} \\
        B(z) & = A(z) - A\qty(\frac{z}{2}). \label{eq:AB}
    \end{align}
\end{rem}

\begin{prop} \label{prop:wiF}
    For $|x| + |y| < 1$, we have
    \[
    \widehat{F}(x, y) = \frac{\sin\pi y}{2\pi}\, (A(x+y) - A(x-y)) - \frac{\sin \pi x}{2\pi}\, (B(x+y) + B(x-y) + 2 \log 2).
    \]
\end{prop}

\begin{prf}
    Applying the identity
    \[
    \frac{2^{2r}}{2^{2r}-1}\, \zeta(\overline{2r+1}) = - \zeta(2r+1) \quad (r \ge 1),
    \]
    we have
    \begin{align}
        \widehat{F}(x, y)
        & = \sum_{a, b \ge 0} (-2)^{a+b} \sum_{r=0}^{a+b} \frac{(-1)^r}{2^r} 
            \alpha_r^{a, b}\, \zeta(\overline{2r+1})\, s(\{2\}^{a+b-r})\, x^{2a+1}y^{2b} \\
        & \begin{aligned}
            & = - \sum_{a, b \ge 0} (-2)^{a+b} \sum_{r=1}^{a+b} \frac{(-1)^r}{2^r} 
                \frac{2^{2r}}{2^{2r}-1}\binom{2r}{2a+1} \zeta(\overline{2r+1})\, s(\{2\}^{a+b-r})\, x^{2a+1}y^{2b} \\
            & \qquad\qquad + \sum_{a, b \ge 0} (-2)^{a+b} \sum_{r=0}^{a+b} \frac{(-1)^r}{2^r}
                \binom{2r}{2b} \zeta(\overline{2r+1})\, s(\{2\}^{a+b-r})\, x^{2a+1}y^{2b}
        \end{aligned} \\
        & \begin{aligned}
            & = \sum_{a, b \ge 0} (-2)^{a+b} \sum_{r=1}^{a+b} \frac{(-1)^r}{2^r} 
                \binom{2r}{2a+1} \zeta(2r+1)\, s(\{2\}^{a+b-r})\, x^{2a+1}y^{2b} \\
            & \qquad\qquad + \sum_{a, b \ge 0} (-2)^{a+b} \sum_{r=0}^{a+b} \frac{(-1)^r}{2^r}
                \binom{2r}{2b} \zeta(\overline{2r+1})\, s(\{2\}^{a+b-r})\, x^{2a+1}y^{2b}.
        \end{aligned}
    \end{align}
    To find the valid range for interchanging the order of summation, 
    let us examine the following positive-term series:
    \[
    \sum_{a, b \ge 0} 2^{a+b} \sum_{r=1}^{a+b} \frac{1}{2^r} 
        \binom{2r}{2a+1} \zeta(2r+1)\, s(\{2\}^{a+b-r})\, |x|^{2a+1}|y|^{2b}.
    \]
    By \eqref{eq:sz}, we have
    \begin{align}
        & \sum_{a, b \ge 0} 2^{a+b} \sum_{r=1}^{a+b} \frac{1}{2^r} 
            \binom{2r}{2a+1} \zeta(2r+1)\, s(\{2\}^{a+b-r})\, |x|^{2a+1}|y|^{2b} \\
        = & \sum_{r \ge 1, s \ge 0} \frac{\pi^{2s}}{(2s+1)!}\, |y|^{2s+1}\, \zeta(2r+1)
            \sum_{a=0}^{r-1} \binom{2r}{2a+1} |x|^{2a+1}|y|^{2r-(2a+1)} \\
        = &\, \frac{\sinh\pi |y|}{2\pi} \sum_{r \ge 1} \zeta(2r+1)\, ((|x|+|y|)^{2r} - (|x|-|y|)^{2r}).
    \end{align}
    Therefore, the series converges for $|x|+|y| < 1$. 
    In this domain, we have
    \begin{align}
        & \sum_{a, b \ge 0} (-2)^{a+b} \sum_{r=1}^{a+b} \frac{(-1)^r}{2^r} 
            \binom{2r}{2a+1} \zeta(2r+1)\, s(\{2\}^{a+b-r})\, x^{2a+1}y^{2b} \\
        = & \sum_{r \ge 1, s \ge 0} \frac{(-1)^s\pi^{2s}}{(2s+1)!}\, y^{2s+1}\, \zeta(2r+1)
            \sum_{a=0}^{r-1} \binom{2r}{2a+1} x^{2a+1}y^{2r-(2a+1)} \\
        = &\, \frac{\sin\pi y}{2\pi} \sum_{r \ge 1} \zeta(2r+1)\, ((x+y)^{2r} - (x-y)^{2r}) \\
        = &\, \frac{\sin\pi y}{2\pi}\, (A(x+y) - A(x-y)).
    \end{align}
    Similarly, for $|x|+|y| < 1$, we have
    \begin{align}
        & \sum_{a, b \ge 0} (-2)^{a+b} \sum_{r=0}^{a+b} \frac{(-1)^r}{2^r}
            \binom{2r}{2b} \zeta(\overline{2r+1})\, s(\{2\}^{a+b-r})\, x^{2a+1}y^{2b} \\
        = & \sum_{r, s \ge 0} \frac{(-1)^s\pi^{2s}}{(2s+1)!}\, x^{2s+1}\, \zeta(\overline{2r+1})
            \sum_{b=0}^{r} \binom{2r}{2b}x^{2r-2b}y^{2b} \\
        = &\, \frac{\sin \pi x}{2\pi} \sum_{r \ge 0} \zeta(\overline{2r+1})\, ((x+y)^{2r} + (x-y)^{2r}) \\
        = & - \frac{\sin \pi x}{2\pi}\, (B(x+y) + B(x-y) + 2 \log 2).
    \end{align}
    In the last equality, we used $\zeta(\overline{1}) = -\log2$.
    Consequently, we obtain
    \[
    \widehat{F}(x, y) = \frac{\sin\pi y}{2\pi}\, (A(x+y) - A(x-y)) - \frac{\sin \pi x}{2\pi}\, (B(x+y) + B(x-y) + 2 \log 2),
    \]
    which completes the proof.
\end{prf}

For convenience, we denote the right-hand side of the equation in Proposition~\ref{prop:wiF} by $\widetilde{F}(x, y)$.
\begin{prop} \label{prop:anF}
    Both $F(x, y)$ and $\widetilde{F}(x, y)$ are entire functions on $\C \times \C$.
    Moreover, for fixed $y \in \C$,
    \[
    F(x,y),\ \widetilde{F}(x,y)=O(e^{\pi |x|}) \quad (|x|\to\infty).
    \]
\end{prop}

\begin{prf}
    From \eqref{eq:sseries}, we have
    \begin{align}
        & \abs{(-2)^{a+b} s(\{2\}^a, 1, \{2\}^b)} \\
        \le & \sum_{0 < n_1 < \dots < n_a < k < m_1 < \dots < m_b} 
            \frac{1}{n_1^2 \cdots n_a^2} \frac{1}{m_1^2 \cdots m_b^2} \qty(\frac{1}{2k-1} - \frac{1}{2k}) \\
        & \qquad + \delta_{b > 0}\ \sum_{0 < n_1 < \dots < n_a < m_1 < \dots < m_b}
            \frac{1}{n_1^2 \cdots n_a^2} \frac{1}{m_1^2 \cdots m_b^2} \frac{1}{2m_1-1} \\
        \le & \, \zeta(\{2\}^{a+b+1}) + \zeta(\{2\}^{a+b}) \\
        = & \, \frac{\pi^{2a+2b+2}}{(2a+2b+3)!} + \frac{\pi^{2a+2b}}{(2a+2b+1)!}.
    \end{align}
    This estimate shows that the defining power series for $F(x, y)$ converges absolutely and uniformly 
    on every compact subset of $\C \times \C$, and hence $F(x, y)$ is entire.
    For a fixed $y \in \C$, a simple calculation shows that
    \[
    \frac{1}{2k} - \frac{1}{2k-1}\qty(1-\frac{y^2}{k^2}) = O\qty(\frac{1}{k^2}) \quad (k \to \infty).
    \]
    Moreover, we have
    \[
    \prod_{m > k} \qty(1-\frac{y^2}{m^2}) = O(1) \quad (k \to \infty),
    \]
    since
    \[
    \abs{\prod_{m > k} \qty(1-\frac{y^2}{m^2})}
    \le \prod_{m > k} \qty(1+\frac{|y|^2}{m^2})
    \le \prod_{m > 0} \qty(1+\frac{|y|^2}{m^2}) 
    = \frac{\sinh\pi|y|}{\pi|y|}.
    \]
    Hence, by \eqref{eq:Fprod}, in order to establish the estimate $F(x,y) = O(e^{\pi |x|})$, 
    it suffices to show that
    \[
    x \sum_{k > 0} \frac{1}{k^2} \abs{\prod_{0 < n < k} \qty(1-\frac{x^2}{n^2})} = O(e^{\pi|x|}).
    \]
    This required bound follows from
    \[
    \abs{\prod_{0 < n < k} \qty(1-\frac{x^2}{n^2})} 
    \le \prod_{0 < n < k} \qty(1+\frac{|x|^2}{n^2})
    \le \prod_{n > 0} \qty(1+\frac{|x|^2}{n^2})
    = \frac{\sinh\pi|x|}{\pi|x|}
    \]
    together with the fact that $\sum_{k > 0} \frac{1}{k^2} = \frac{\pi^2}{6}$.

    For $\widetilde{F}(x, y)$, we only need to show that
    \[
    \sin\pi y\, A(x+y) - \sin\pi x\, B(x+y)
    \]
    is entire and satisfies the estimate $O(e^{\pi|x|})$ for each fixed $y$.
    Indeed, by replacing $y$ with $-y$, the same can be said for $\sin\pi y\, A(x-y) + \sin\pi x\, B(x-y)$,
    and the desired result follows from the definition of $\widetilde{F}(x, y)$.
    The functions $A(z)$ and $B(z)$ are meromorphic on $\C$ with only simple poles at points in $\Z\setminus\{0\}$.
    Their residues at $z = n \in \Z\setminus\{0\}$ are given by
    \[
    \Res_{z = n} A(z) = -\frac{\sgn n}{2}, \qquad \Res_{z = n} B(z) = (-1)^{n}\frac{\sgn n}{2}.
    \]
    When $x+y = n \in \Z$, we have $\sin\pi x = (-1)^{n-1} \sin\pi y$,
    so the potential poles in $\sin\pi y\, A(x+y) - \sin\pi x\, B(x+y)$ are canceled.
    It remains to verify that
    \[
    \sin\pi y\, A(x+y) - \sin\pi x\, B(x+y) = O(e^{\pi|x|}),
    \]
    which follows from the estimates
    \begin{align}
        A(z) & = -\gamma -\frac{\pi}{2}\cot\pi z -\log z + O(z^{-1}), \\
        B(z) & = -\log2 + \frac{\pi}{2}\csc\pi z+O(z^{-1})
    \end{align}
    as $|z| \to \infty$ for $\abs{\arg z} < \pi$.
\end{prf}

\begin{prop} \label{prop:diagF}
    For $x \in \C$, we have
    \[
    F(x, x) = \widetilde{F}(x, x).
    \]
\end{prop}

\begin{prf}
    By \eqref{eq:Fprod}, we have
    \begin{align}
            F(x, x) 
        & = \frac{\sin \pi x}{\pi} \sum_{k > 0} \qty(1-\frac{x^2}{k^2})^{-1}
            \qty(\frac{1}{2k} - \frac{1}{2k-1}\qty(1-\frac{x^2}{k^2})) \\
        & = \frac{\sin \pi x}{2\pi} \sum_{k > 0} \qty(\frac{1}{k\qty(1 - \frac{x^2}{k^2})} - \frac{1}{k-\frac{1}{2}}) \\
        & = \frac{\sin \pi x}{2\pi} \sum_{k > 0} \qty(\frac{x^2}{k(k^2-x^2)} - \qty(\frac{1}{k-\frac{1}{2}} - \frac{1}{k})) \\
        & = \frac{\sin\pi x}{2\pi}\, (A(x) - 2\log2).
    \end{align}
    On the other hand, by the identity \eqref{eq:AB},
    we have
    \begin{align}
        \widetilde{F}(x, x) & = \frac{\sin\pi x}{2\pi}\, (A(2x) - B(2x) - 2\log 2) \\
        & = \frac{\sin\pi x}{2\pi}\, (A(x) - 2\log2).
    \end{align}
    Thus, we obtain $F(x, x) = \widetilde{F}(x, x)$, completing the proof.
\end{prf}

\begin{prop} \label{prop:intF}
    For $n \in \Z$ and $y \in \C$, we have
    \[
    F(n, y) = \widetilde{F}(n, y).
    \]
\end{prop}

\begin{prf}
    We can assume that $n$ is a positive integer.
    By the definition of $\widetilde{F}(x, y)$, we have
    \[
    \widetilde{F}(n, y) = \frac{\sin\pi y}{2\pi}\, (A(y+n) - A(y-n)).
    \]
    By the identity \eqref{eq:digamma}, together with the functional equation $\psi(z+1) - \psi(z) = \frac{1}{z}$, we have
    \[
    A(z+1) - A(z) = - \frac{1}{2} \qty(\frac{1}{z+1} + \frac{1}{z}).
    \]
    Thus, it follows that
    \begin{align}
        A(y+n) - A(y-n)
        & = \sum_{i=-n}^{n-1} (A(y+i+1) - A(y+i)) \\
        & = - \frac{1}{2} \sum_{i=-n}^{n-1} \qty(\frac{1}{y+i+1} + \frac{1}{y+i}) \\
        & = - \sideset{}{^*}\sum_{|i| \le n} \frac{1}{y-i},
    \end{align}
    where the symbol $\sideset{}{^*}\sum_{|i| \le n}$ denotes the weighted sum
    \[
    \sideset{}{^*}\sum_{|i| \le n} a_i \coloneqq \sum_{|i| < n} a_i + \frac{1}{2}\sum_{|i| = n} a_i.
    \]
    Therefore, we obtain
    \[
    \widetilde{F}(n, y) = - \frac{\sin\pi y}{2\pi} \sideset{}{^*}\sum_{|i| \le n} \frac{1}{y-i}.
    \]
    On the other hand, by Proposition~\ref{prop:F} and the partial fraction expansion
    \[
    \frac{1}{(1-y)_k (1+y)_k} = \sum_{i = 1}^{k} \frac{(-1)^i\, i}{(k-i)!(k+i)!} \qty(\frac{1}{y-i} - \frac{1}{y+i}),
    \]
    we have
    \begin{align}
        & \qquad F(x, y) \\
        & \begin{aligned}
            & = x\frac{\sin \pi y}{\pi y} \sum_{k > 0} \sum_{i = 1}^{k} (-1)^{i}
                \frac{(1-x)_{k-1}(1+x)_{k-1}}{(k-i)!(k+i)!}\, i\qty(\frac{1}{y-i} - \frac{1}{y+i})
                \qty(\frac{k}{2} - \frac{k^2 - x^2}{2k+1}) \\
            & \qquad\qquad - x \frac{\sin \pi y}{\pi y}
        \end{aligned} \\
        & \begin{aligned}
            & = x \frac{\sin \pi y}{\pi} \sum_{k > 0} \sum_{i = 1}^{k} (-1)^{i}
                \frac{(1-x)_{k-1}(1+x)_{k-1}}{(k-i)!(k+i)!} \qty(\frac{1}{y-i} + \frac{1}{y+i} - \frac{2}{y})
                \qty(\frac{k}{2} - \frac{k^2 - x^2}{2k+1}) \\
            & \qquad\qquad - x \frac{\sin \pi y}{\pi y}.
        \end{aligned}
    \end{align}
    By setting $x = n \in \Z$, we obtain
    \begin{align}
        & \qquad F(n, y) \\
        & \begin{aligned}
            & = n \frac{\sin \pi y}{\pi} \sum_{k > 0} \sum_{i = 1}^{k} (-1)^{i}
                \frac{(1-n)_{k-1}(1+n)_{k-1}}{(k-i)!(k+i)!} \qty(\frac{1}{y-i} + \frac{1}{y+i} - \frac{2}{y})
                \qty(\frac{k}{2} - \frac{k^2 - n^2}{2k+1}) \\
            & \qquad\qquad - n \frac{\sin \pi y}{\pi y}
        \end{aligned} \\
        & \begin{aligned}
            & = - \frac{\sin \pi y}{\pi} \sum_{k = 1}^{n} \sum_{i = 1}^{k} (-1)^{k-i}
                \frac{(n-k+1)_{2k-1}}{(k-i)!(k+i)!} \qty(\frac{1}{y-i} + \frac{1}{y+i} - \frac{2}{y})
                \qty(\frac{k}{2} + \frac{n^2 - k^2}{2k+1}) \\
            & \qquad\qquad - n \frac{\sin \pi y}{\pi y}
        \end{aligned} \\
        & \begin{aligned}
            & = - \frac{\sin \pi y}{\pi} \sum_{i = 1}^{n} \qty(\frac{1}{y-i} + \frac{1}{y+i} - \frac{2}{y}) 
                \sum_{k = i}^{n} (-1)^{k-i} \frac{(n-k+1)_{2k-1}}{(k-i)!(k+i)!} \qty(\frac{k}{2} + \frac{n^2 - k^2}{2k+1}) \\
            & \qquad\qquad - \frac{\sin \pi y}{\pi}\, \frac{n}{y}.
        \end{aligned}
    \end{align}
    Therefore, it suffices to show that the following identity holds for $1 \le i \le n$:
    \[
    \sum_{k = i}^{n} (-1)^{k-i} \frac{(n-k+1)_{2k-1}}{(k-i)!(k+i)!} \qty(\frac{k}{2} + \frac{n^2 - k^2}{2k+1}) =
    \begin{cases}
        \frac{1}{2} & \text{if} \ 1 \le i < n, \\
        \frac{1}{4} & \text{if} \ i = n.
    \end{cases}
    \]
    We prove this identity by considering the generating function with respect to $n$.
    For a fixed $i \ge 1$, we define $a_{n, k}$ by
    \[
    a_{n, k} \coloneqq \frac{(n-k+1)_{2k-1}}{(k-i)!(k+i)!} \qty(\frac{k}{2} + \frac{n^2 - k^2}{2k+1})
        = \binom{2k}{k-i} \qty(\binom{n+k}{2k+1} + \frac{1}{4} \binom{n+k-1}{2k-1}).
    \]
    We then define the generating function $F_i(x)$ by
    \[
    F_i(x) = \sum_{n = i}^{\infty} \sum_{k = i}^{n} (-1)^{k-i}\, a_{n, k}\, x^n.
    \]
    By the binomial theorem, for $0 \le r < 1$, we have
    \begin{align}
        \sum_{n = i}^{\infty} \sum_{k = i}^{n} a_{n, k}\, r^n
        & = \sum_{k = i}^{\infty} \sum_{n = k}^{\infty} a_{n, k}\, r^n \\
        & = \sum_{k = i}^{\infty} \binom{2k}{k-i}
            \sum_{n = k}^{\infty} \qty(\binom{n+k}{2k+1} + \frac{1}{4} \binom{n+k-1}{2k-1})\, r^n \\
        & = \sum_{k = i}^{\infty} \binom{2k}{k-i}
            \qty(\frac{r^{k+1}}{(1-r)^{2k+2}} + \frac{1}{4} \frac{r^k}{(1-r)^{2k}}) \\
        & = \frac{1}{4} \frac{(1+r)^2}{(1-r)^2} \sum_{k = i}^{\infty}
            \binom{2k}{k-i} \frac{r^k}{(1-r)^{2k}} \\
        & = \frac{r^i}{4} \frac{(1+r)^2}{(1-r)^{2i+2}}
            \sum_{m = 0}^{\infty} \binom{2m+2i}{m} \frac{r^m}{(1-r)^{2m}}.
    \end{align}
    The series 
    \[
    \sum_{m = 0}^{\infty} \binom{2m+p}{m}\, t^m = \frac{1}{\sqrt{1-4t}}\qty(\frac{1-\sqrt{1-4t}}{2t})^{p}
    \]
    has radius of convergence $\frac{1}{4}$.
    In our case, this corresponds to the condition
    \[
    \frac{r}{(1-r)^2} < \frac{1}{4} \iff r < 3-2\sqrt{2} \quad (\text{for} \ 0 \le r < 1).
    \]
    Therefore, the order of summation may be interchanged for $|x| < 3-2\sqrt{2}$.
    In this domain, we have
    \begin{align}
        F_i(x)
        & = \sum_{n = i}^{\infty} \sum_{k = i}^{n} (-1)^{k-i}\, a_{n, k}\, x^n \\
        & = \sum_{k = i}^{\infty} (-1)^{k-i} \sum_{n = k}^{\infty} a_{n, k}\, x^n \\
        & = \frac{x^i}{4} \frac{(1+x)^2}{(1-x)^{2i+2}}
            \sum_{m = 0}^{\infty} (-1)^m \binom{2m+2i}{m} \frac{x^m}{(1-x)^{2m}} \\
        & = \frac{x^i}{4} \frac{1+x}{1-x} \\
        & = \frac{1}{4}\, x^i + \frac{1}{2} \sum_{n > i} x^n,
    \end{align}
    as desired.
\end{prf}

\begin{prf}[of Theorem~\ref{thm:sidentity}]
    Let $f(x) = F(x, y) - \widetilde{F}(x, y)$ for each fixed $y$.
    By Proposition~\ref{prop:anF}, $f$ is an entire function
    satisfying $f(x) = O(e^{\pi|x|})$ as $|x| \to \infty$.
    Moreover, by Proposition~\ref{prop:intF}, $f$ vanishes at all integers.
    Therefore, by \cite[Th\'eor\`eme fondamental]{zbMATH02591652},
    $f(x)$ is a constant multiple of $\sin\pi x$.
    Finally, Proposition~\ref{prop:diagF} yields $f(x) = 0$,
    which implies $F(x, y) = \widehat{F}(x, y)$ for $|x| + |y| < 1$
    by Proposition~\ref{prop:wiF}.
\end{prf}

Next, we turn to the proof of Theorem~\ref{thm:bsidentity}.
Let $G(x, y)$ and $\widehat{G}(x, y)$ be generating functions
\begin{align}
    G(x, y) & = \sum_{a, b \ge 0} (-2)^{a+b}\, \bar{s}(\{2\}^a, 1, \{2\}^b)\, x^{2a} y^{2b}, \\
    \widehat{G}(x, y) & = \sum_{a, b \ge 0} (-2)^{a+b} \sum_{r=0}^{a+b} \frac{(-1)^r}{2^r}
        \beta_r^{a, b}\, \zeta(\overline{2r+1})\, \bar{s}(\{2\}^{a+b-r})\, x^{2a}y^{2b}.
\end{align}

\begin{prop} \label{prop:G}
    We have
    \begin{align}
        G(x, y) 
        & = \cos\pi y \sum_{k > 0} \frac{(\frac{1}{2}-x)_{k-1}(\frac{1}{2}+x)_{k-1}}{(\frac{1}{2}-y)_k (\frac{1}{2}+y)_k}
        \qty(\frac{2k-1}{4} - \frac{(k-\frac{1}{2})^2 - x^2}{2k}) \\
        & = \cos\pi y \sum_{k > 1} \frac{(\frac{1}{2}-x)_{k-1}(\frac{1}{2}+x)_{k-1}}{(\frac{1}{2}-y)_k (\frac{1}{2}+y)_k}
            \qty(\frac{2k-1}{4} - \frac{(k-\frac{1}{2})^2 - y^2}{2k-2})
            + \frac{\cos\pi y}{1-4y^2}.
    \end{align}
\end{prop}

\begin{prf}
    By the definition of M$\bar{s}$Vs, we have
    \begin{align}
        & \quad \bar{s}(\{2\}^a, 1, \{2\}^b) \\
        & = - \sum_{\epsilon_i, \epsilon'_j \in \{\pm 1\}} \epsilon_1\cdots\epsilon_a\, \epsilon'_1\cdots\epsilon'_b\, 
            \zeta\binom{2, \dots, 2, 1, 2, \dots, 2}{\epsilon_1, \dots, \epsilon_a, -1, \epsilon'_1, \dots, \epsilon'_b} \\
        & = - \sum_{0 < n_1 < \dots < n_a < k < m_1 < \dots < m_b} \prod_{i=1}^{a}\frac{1-(-1)^{n_i}}{n_i^2}
            \frac{(-1)^k}{k} \prod_{j=1}^{b}\frac{1-(-1)^{m_j}}{m_j^2} \\
        & = - 2^{a+b} \sum_{\substack{0 < n_1 < \dots < n_a < k < m_1 < \dots < m_b \\ n_i, m_j: \text{odd}}}
            \frac{1}{n_1^2 \cdots n_a^2} \frac{(-1)^k}{k} \frac{1}{m_1^2 \cdots m_b^2} \\
        & \begin{aligned}
        & = 2^{-(a+b)} \sum_{0 < n_1 < \dots < n_a < k < m_1 < \dots < m_b}
            \prod_{i=1}^{a}\frac{1}{(n_i-\frac{1}{2})^2} \prod_{j=1}^{b}\frac{1}{(m_j-\frac{1}{2})^2} \qty(\frac{1}{2k-1} - \frac{1}{2k}) \\
        & \qquad - \delta_{a > 0}\ 2^{-(a+b)} \sum_{0 < n_1 < \dots < n_a = k < m_1 < \dots < m_b}
            \prod_{i=1}^{a}\frac{1}{(n_i-\frac{1}{2})^2} \prod_{j=1}^{b}\frac{1}{(m_j-\frac{1}{2})^2} \frac{1}{2k}. \label{eq:bsseries}
        \end{aligned}
    \end{align}
    Hence,
    \begin{align}
        & \quad G(x, y) \\
        & = \sum_{a, b \ge 0} (-2)^{a+b}\, \bar{s}(\{2\}^a, 1, \{2\}^b)\, x^{2a} y^{2b} \\
        & = \sum_{k > 0} \prod_{0 < n < k} \qty(1-\frac{x^2}{(n-\frac{1}{2})^2})
            \prod_{m > k} \qty(1-\frac{y^2}{(m-\frac{1}{2})^2}) \qty(\frac{1}{2k-1} - \frac{1}{2k}\qty(1-\frac{x^2}{(k-\frac{1}{2})^2})) \label{eq:Gprod} \\
        & = \cos\pi y \sum_{k > 0} \frac{(\frac{1}{2}-x)_{k-1}(\frac{1}{2}+x)_{k-1}}{(\frac{1}{2}-y)_k (\frac{1}{2}+y)_k}
            \qty(\frac{2k-1}{4} - \frac{(k-\frac{1}{2})^2 - x^2}{2k}).
    \end{align}
    Since
    \[
    \frac{(\frac{1}{2}-x)_{k-1}(\frac{1}{2}+x)_{k-1}}{(\frac{1}{2}-y)_k (\frac{1}{2}+y)_k} \frac{(k-\frac{1}{2})^2 - x^2}{2k}
    = \frac{(\frac{1}{2}-x)_{k}(\frac{1}{2}+x)_{k}}{(\frac{1}{2}-y)_{k+1} (\frac{1}{2}+y)_{k+1}} \frac{(k+\frac{1}{2})^2 - y^2}{2k},
    \]
    shifting the index in the second term gives
    \[
    G(x, y) = \cos\pi y \sum_{k > 1} \frac{(\frac{1}{2}-x)_{k-1}(\frac{1}{2}+x)_{k-1}}{(\frac{1}{2}-y)_k (\frac{1}{2}+y)_k}
    \qty(\frac{2k-1}{4} - \frac{(k-\frac{1}{2})^2 - y^2}{2k-2}) + \frac{\cos\pi y}{1-4y^2},
    \]
    as desired.
\end{prf}

\begin{prop} \label{prop:wiG}
    For $|x|, |y| < \frac{1}{2}$, we have
    \begin{equation}
        \begin{aligned}
            \widehat{G}(x, y) & = - \frac{\cos\pi y}{2}\, (A(x+y) + A(x-y) - 2A(y)) \\
            & \qquad - \frac{\cos\pi x}{2}\, (B(x+y) + B(x-y) + 2\log 2) + 2\cos\pi y\, (B(2y) + \log2).
        \end{aligned}
    \end{equation}
\end{prop}

\begin{prf}
    By the definition of $\beta_r^{a, b}$, we have
    \begin{align}
        \widehat{G}(x, y)
        & = \sum_{a, b \ge 0} (-2)^{a+b} \sum_{r=0}^{a+b} \frac{(-1)^r}{2^r} 
            \beta_r^{a, b}\, \zeta(\overline{2r+1})\, \bar{s}(\{2\}^{a+b-r})\, x^{2a}y^{2b} \\
        & \begin{aligned}
            & = \sum_{a \ge 1, b \ge 0} (-2)^{a+b} \sum_{r=1}^{a+b} \frac{(-1)^r}{2^r} 
                \frac{2^{2r}}{2^{2r}-1} \binom{2r}{2a} \zeta(\overline{2r+1})\, \bar{s}(\{2\}^{a+b-r})\, x^{2a}y^{2b} \\
            & \qquad\qquad + \sum_{a, b \ge 0} (-2)^{a+b} \sum_{r=0}^{a+b} \frac{(-1)^r}{2^r}
                \binom{2r}{2b} \zeta(\overline{2r+1})\, \bar{s}(\{2\}^{a+b-r})\, x^{2a}y^{2b} \\
            & \qquad\qquad - \sum_{b \ge 0} (-2)^{b} \sum_{r=0}^{b} \frac{(-1)^r}{2^r}\, 2^{2r+1}\, 
                \zeta(\overline{2r+1})\, \bar{s}(\{2\}^{b-r})\, y^{2b}.
        \end{aligned}
    \end{align}
    Applying the identity
    \[
    \frac{2^{2r}}{2^{2r}-1}\, \zeta(\overline{2r+1}) = - \zeta(2r+1) \quad (r \ge 1),
    \]
    we obtain
    \begin{equation}
        \widehat{G}(x, y)
        \begin{aligned}[t]
            & = - \sum_{a \ge 1, b \ge 0} (-2)^{a+b} \sum_{r=1}^{a+b} \frac{(-1)^r}{2^r} 
                \binom{2r}{2a} \zeta(2r+1)\, \bar{s}(\{2\}^{a+b-r})\, x^{2a}y^{2b} \\
            & \qquad\qquad + \sum_{a, b \ge 0} (-2)^{a+b} \sum_{r=0}^{a+b} \frac{(-1)^r}{2^r}
                \binom{2r}{2b} \zeta(\overline{2r+1})\, \bar{s}(\{2\}^{a+b-r})\, x^{2a}y^{2b} \\
            & \qquad\qquad - \sum_{b \ge 0} (-2)^{b} \sum_{r=0}^{b} \frac{(-1)^r}{2^r}\, 2^{2r+1}\, 
                \zeta(\overline{2r+1})\, \bar{s}(\{2\}^{b-r})\, y^{2b}.
        \end{aligned}
    \end{equation}
    By an argument similar to that for $\widehat{F}(x, y)$, 
    we may interchange the order of summation if $|x|, |y| < \frac{1}{2}$.
    In this domain, it follows from the identity~\eqref{eq:bsz} that
    \begin{align}
        \widehat{G}(x, y)
        & = - \sum_{r \ge 1, s \ge 0} (-2)^s\, \bar{s}(\{2\}^s)\, 
            \zeta(2r+1) \sum_{a=1}^{r} \binom{2r}{2a} x^{2a}y^{2(r+s-a)} \\
        & \qquad\qquad + \sum_{r, s \ge 0} (-2)^s\, \bar{s}(\{2\}^s)\, \zeta(\overline{2r+1})
            \sum_{b=0}^{r} \binom{2r}{2b} x^{2(r+s-b)}y^{2b} \\
        & \qquad\qquad - \sum_{r, s \ge 0} (-2)^{s}\, \bar{s}(\{2\}^s)\, 2^{2r+1}\, \zeta(\overline{2r+1})\, y^{2(r+s)} \\
        & = - \sum_{r \ge 1, s \ge 0} \frac{(-1)^s\pi^{2s}}{(2s)!}\, y^{2s}\, \zeta(2r+1)
            \sum_{a=1}^{r} \binom{2r}{2a} x^{2a}y^{2r-2a} \\
        & \qquad\qquad + \sum_{r, s \ge 0} \frac{(-1)^s\pi^{2s}}{(2s)!}\, x^{2s}\, \zeta(\overline{2r+1})
            \sum_{b=0}^{r} \binom{2r}{2b}x^{2r-2b}y^{2b} \\
        & \qquad\qquad - \sum_{r, s \ge 0} \frac{(-1)^s\pi^{2s}}{(2s)!}\, y^{2s}\, 2^{2r+1}\, \zeta(\overline{2r+1})\, y^{2r} \\
        & = - \frac{\cos\pi y}{2} \sum_{r \ge 1} \zeta(2r+1)\, ((x+y)^{2r} + (x-y)^{2r} - 2y^{2r}) \\
        & \qquad\qquad + \frac{\cos\pi x}{2} \sum_{r \ge 0} \zeta(\overline{2r+1})\, ((x+y)^{2r} + (x-y)^{2r}) \\
        & \qquad\qquad - 2\cos\pi y \sum_{r \ge 0} \zeta(\overline{2r+1})\, (2y)^{2r} \\
        & = - \frac{\cos\pi y}{2}\, (A(x+y) + A(x-y) - 2A(y)) \\
        & \qquad - \frac{\cos\pi x}{2}\, (B(x+y) + B(x-y) + 2\log 2) + 2\cos\pi y\, (B(2y)+\log2).
    \end{align}
    In the last equality, we used $\zeta(\overline{1}) = -\log2$.
    This completes the proof.
\end{prf}

We denote the right-hand side of the equation in Proposition~\ref{prop:wiG} by $\widetilde{G}(x, y)$.

\begin{prop} \label{prop:anG}
    Both $G(x, y)$ and $\widetilde{G}(x, y)$ are entire functions on $\C \times \C$.
    Moreover, for fixed $y \in \C$,
    \[
    G(x,y),\ \widetilde{G}(x,y)=O(e^{\pi |x|}) \quad (|x|\to\infty).
    \]
\end{prop}

\begin{prf}
    From \eqref{eq:bsseries}, we have
    \begin{align}
        & \abs{(-2)^{a+b}\, \bar{s}(\{2\}^a, 1, \{2\}^b)} \\
        \le & \sum_{0 < n_1 < \dots < n_a < k < m_1 < \dots < m_b}
            \prod_{i=1}^{a}\frac{4}{(2n_i-1)^2} \prod_{j=1}^{b}\frac{4}{(2m_j-1)^2} \qty(\frac{1}{2k-1} - \frac{1}{2k}) \\
        & \qquad + \delta_{a > 0}\ \sum_{0 < n_1 < \dots < n_a < m_1 < \dots < m_b}
            \prod_{i=1}^{a}\frac{4}{(2n_i-1)^2} \prod_{j=1}^{b}\frac{4}{(2m_j-1)^2} \frac{1}{2n_a} \\
        \le &\, 4^{a+b}\, \zeta(\{2\}^{a+b+1}) + 4^{a+b}\, \zeta(\{2\}^{a+b}) \\
        = &\, \frac{4^{a+b}\pi^{2a+2b+2}}{(2a+2b+3)!} + \frac{4^{a+b}\pi^{2a+2b}}{(2a+2b+1)!}.
    \end{align}
    This estimate shows that the defining power series for $G(x, y)$ converges absolutely and uniformly 
    on every compact subset of $\C \times \C$, and hence $G(x, y)$ is entire.
    By Proposition~\ref{prop:G}, we have
    \begin{align}
        G(x, y)
        & = \cos\pi y \sum_{k > 1} \frac{(\frac{1}{2}-x)_{k-1}(\frac{1}{2}+x)_{k-1}}{(\frac{1}{2}-y)_k (\frac{1}{2}+y)_k}
            \qty(\frac{2k-1}{4} - \frac{(k-\frac{1}{2})^2 - y^2}{2k-2}) + \frac{\cos\pi y}{1-4y^2} \\
        & \begin{aligned}
            & = \sum_{k > 1} \prod_{0 < n < k} \qty(1-\frac{x^2}{(n-\frac{1}{2})^2}) \prod_{m > k} \qty(1-\frac{y^2}{(m-\frac{1}{2})^2}) \\
            & \qquad\qquad\qquad \times \qty(\frac{1}{2k-1} - \frac{1}{2k-2}\qty(1-\frac{y^2}{(k-\frac{1}{2})^2})) + \frac{\cos\pi y}{1-4y^2}.
        \end{aligned}
    \end{align}
    As in the case of $F(x, y)$, for a fixed $y \in \C$, we have
    \[
    \frac{1}{2k-1} - \frac{1}{2k-2}\qty(1-\frac{y^2}{(k-\frac{1}{2})^2}) = O\qty(\frac{1}{k^2})
    \]
    and
    \[
    \prod_{m > k} \qty(1-\frac{y^2}{(m-\frac{1}{2})^2}) = O(1)
    \]
    as $k \to \infty$.
    Thus, using the estimate
    \[
    \sum_{k > 1} \frac{1}{k^2} \abs{\prod_{0 < n < k} \qty(1-\frac{x^2}{(n-\frac{1}{2})^2})} = O(e^{\pi|x|}),
    \]
    we obtain the desired estimate $G(x,y) = O(e^{\pi |x|})$.

    By an argument similar to that used for $\widetilde{F}(x, y)$, the functions
    \[
    \cos\pi y\, A(x\pm y) + \cos\pi x\, B(x\pm y)
    \]
    are entire and satisfy the estimate $O(e^{\pi|x|})$ for each fixed $y$.
    By the definition of $\widetilde{G}(x, y)$, it suffices to show that 
    \[
    \cos\pi y\, (A(y) + 2B(2y))
    \]
    is entire.
    By the identity \eqref{eq:AB}, this function coincides with $\cos\pi y\, (2A(2y) - A(y))$.
    When $y$ is a half-integer, the poles of $A(2y)$ are canceled by the zeros of $\cos\pi y$.
    When $y$ is an integer, the residues of $2A(2y)$ and $A(y)$ coincide,
    and hence their poles cancel.
    This completes the proof.
\end{prf}

\begin{prop} \label{prop:diagG}
    For $x \in \C$, we have
    \[
    G(x, x) = \widetilde{G}(x, x).
    \]
\end{prop}

\begin{prf}
    By \eqref{eq:Gprod}, we have
    \begin{align}
        G(x, x)
        & = \cos\pi x \sum_{k > 0} \qty(1-\frac{x^2}{(k-\frac{1}{2})^2})^{-1}
            \qty(\frac{1}{2k-1} - \frac{1}{2k}\qty(1-\frac{x^2}{(k-\frac{1}{2})^2})) \\
        & = \frac{\cos\pi x}{2} \sum_{k > 0} \qty(\frac{1}{(k-\frac{1}{2})\qty(1-\frac{x^2}{(k-\frac{1}{2})^2})} - \frac{1}{k}) \\
        & = \frac{\cos\pi x}{2} \sum_{k > 0} \qty(\frac{x^2}{(k-\frac{1}{2})\qty((k-\frac{1}{2})^2 - x^2)} + \frac{1}{k-\frac{1}{2}} - \frac{1}{k}) \\
        & = \frac{\cos\pi x}{2} \qty(2\sum_{k > 0}\frac{(2x)^2}{(2k-1)((2k-1)^2 - (2x)^2)} + 2\log2) \\
        & = \frac{\cos\pi x}{2}\, \qty(A(2x) + B(2x) + 2\log2).
    \end{align}
    By the definition of $\widetilde{G}(x, y)$, we have
    \[
    \widetilde{G}(x, x)
    \begin{aligned}[t]
        & = - \frac{\cos\pi x}{2}\, (A(2x) - 2A(x)) - \frac{\cos\pi x}{2}\, (B(2x) + 2\log2) \\
        & \qquad\qquad  + 2\cos\pi x\, (B(2x) + \log2).
    \end{aligned}
    \]
    By the identity \eqref{eq:AB}, we obtain 
    \[
    G(x, x) = \frac{\cos\pi x}{2}\, \qty(2A(2x) - A(x) + 2\log2) = \widetilde{G}(x, x).
    \] 
    This completes the proof.
\end{prf}

\begin{prop} \label{prop:intG}
    For $n \in \Z$ and $y \in \C$, we have
    \[
    G\qty(n-\frac{1}{2}, y) = \widetilde{G}\qty(n-\frac{1}{2}, y).
    \]
\end{prop}

\begin{prf}
    We can assume that $n$ is a positive integer.
    By Proposition~\ref{prop:G} and the partial fraction expansion
    \[
    \frac{1}{(\frac{1}{2}-y)_k (\frac{1}{2}+y)_k} = \sum_{i=1}^{k} \frac{(-1)^i}{(k-i)!(k+i-1)!} \qty(\frac{1}{y-\qty(i-\frac{1}{2})} - \frac{1}{y+\qty(i-\frac{1}{2})}),
    \]
    we have
    \begin{equation}
        G(x, y)
        \begin{aligned}[t]
            = \cos\pi y \sum_{k > 0} \sum_{i=1}^{k} (-1)^i & \frac{(\frac{1}{2}-x)_{k-1} (\frac{1}{2}+x)_{k-1}}{(k-i)!(k+i-1)!} 
                \qty(\frac{1}{y-\qty(i-\frac{1}{2})} - \frac{1}{y+\qty(i-\frac{1}{2})}) \\
            & \times \qty(\frac{2k-1}{4} - \frac{(k-\frac{1}{2})^2 - x^2}{2k}). 
        \end{aligned}
    \end{equation}
    By setting $x = n-\frac{1}{2}$, we obtain
    \begin{align}
        & \quad G\qty(n-\frac{1}{2}, y) \\
        & \begin{aligned}[t]
            = \cos\pi y \sum_{k > 0} \sum_{i=1}^{k} & (-1)^i \frac{(1-n)_{k-1} (n)_{k-1}}{(k-i)!(k+i-1)!} 
                \qty(\frac{1}{y-\qty(i-\frac{1}{2})} - \frac{1}{y+\qty(i-\frac{1}{2})}) \\
            & \times \qty(\frac{2k-1}{4} - \frac{(k-\frac{1}{2})^2 - (n-\frac{1}{2})^2}{2k})
        \end{aligned} \\
        & \begin{aligned}[t]
            = - \cos\pi y \sum_{k = 1}^{n} \sum_{i=1}^{k} & (-1)^{k-i} \frac{(n-k+1)_{2k-2}}{(k-i)!(k+i-1)!} 
                \qty(\frac{1}{y-\qty(i-\frac{1}{2})} - \frac{1}{y+\qty(i-\frac{1}{2})}) \\
            & \times \qty(\frac{2k-1}{4} + \frac{(n-k)(n+k-1)}{2k})
        \end{aligned} \\
        & \begin{aligned}[t]
            & = - \cos\pi y \sum_{i=1}^{n} \qty(\frac{1}{y-\qty(i-\frac{1}{2})} - \frac{1}{y+\qty(i-\frac{1}{2})}) \\
            & \qquad \times \sum_{k = i}^{n} (-1)^{k-i} \frac{(n-k+1)_{2k-2}}{(k-i)!(k+i-1)!} \qty(\frac{2k-1}{4} + \frac{(n-k)(n+k-1)}{2k})
        \end{aligned} \\
            & \begin{aligned}[t]
            & = - \cos\pi y \sum_{i=1}^{n} \qty(\frac{1}{y-\qty(i-\frac{1}{2})} - \frac{1}{y+\qty(i-\frac{1}{2})}) \\
            & \qquad \times \sum_{k = i}^{n} (-1)^{k-i} \binom{2k-1}{k-i} \qty(\binom{n+k-1}{2k} + \frac{1}{4}\binom{n+k-2}{2k-2}).
        \end{aligned}
    \end{align}
    As in the case of Proposition~\ref{prop:intF}, we have the following identity for $1 \le i \le n$:
    \[
    \sum_{k = i}^{n} (-1)^{k-i} \binom{2k-1}{k-i} \qty(\binom{n+k-1}{2k} + \frac{1}{4}\binom{n+k-2}{2k-2}) =
    \begin{cases}
        \frac{1}{2} & \text{if} \ i < n, \\
        \frac{1}{4} & \text{if} \ i = n.
    \end{cases}
    \]
    Therefore, we obtain
    \[
    G\qty(n-\frac{1}{2}, y) = - \frac{\cos\pi y}{2} \sideset{}{^*}\sum_{\substack{|i| < n \\ i: \text{half-integer}}} \frac{\mathrm{sgn}\, i}{y - i},
    \]
    where the symbol $\sideset{}{^*}\sum_{\substack{|i| < n \\ i: \text{half-integer}}}$ denotes the weighted sum
    \[
    \sideset{}{^*}\sum_{\substack{|i| < n \\ i: \text{half-integer}}} a_i \coloneqq 
        \sum_{\substack{|i| < n-\frac{1}{2} \\ i: \text{half-integer}}} a_i + \frac{1}{2}\sum_{|i| = n-\frac{1}{2}} a_i.
    \]
    On the other hand, by the identity \eqref{eq:AB}, we have
    \begin{align}
        \widetilde{G}\qty(n-\frac{1}{2}, y)
        & \begin{aligned}[t]
            & = - \frac{\cos\pi y}{2} \qty(A\qty(y+n-\frac{1}{2}) + A\qty(y-n+\frac{1}{2})) \\
            & \qquad\qquad + \cos\pi y\, (A(y) + 2B(2y) + 2\log2)
        \end{aligned} \\
        & \begin{aligned}[t]
            & = - \frac{\cos\pi y}{2} \qty(A\qty(y+n-\frac{1}{2}) + A\qty(y-n+\frac{1}{2})) \\
            & \qquad\qquad + \cos\pi y\, (2A(2y) - A(y) + 2\log2).
        \end{aligned}
    \end{align}
    By \eqref{eq:digamma} and the duplication formula of the digamma function
    \[
    \psi(2z) = \frac{1}{2} \qty(\psi(z) + \psi\qty(z+\frac{1}{2})) + \log2,
    \]
    we have
    \[
    2A(2y) - A(y) + 2\log2 = \frac{1}{2}\, \qty(A\qty(y-\frac{1}{2}) + A\qty(y+\frac{1}{2})) - \frac{1}{4}\qty(\frac{1}{y-\frac{1}{2}} - \frac{1}{y+\frac{1}{2}}).
    \]
    Moreover, from the functional equation
    \[
    A(z+1) - A(z) = - \frac{1}{2} \qty(\frac{1}{z+1} + \frac{1}{z}),
    \]
    we obtain
    \begin{align}
        A\qty(y+n-\frac{1}{2}) - A\qty(y+\frac{1}{2})
        & = \sum_{i=1}^{n-1} \qty(A\qty(y+i+\frac{1}{2}) - A\qty(y+i-\frac{1}{2})) \\
        & = - \frac{1}{2} \sum_{i=1}^{n-1} \qty(\frac{1}{y+i+\frac{1}{2}} + \frac{1}{y+i-\frac{1}{2}}),
    \end{align}
    and
    \begin{align}
        A\qty(y-n+\frac{1}{2}) - A\qty(y-\frac{1}{2})
        & = - \sum_{i=1}^{n-1} \qty(A\qty(y-i+\frac{1}{2}) - A\qty(y-i-\frac{1}{2})) \\
        & = \frac{1}{2} \sum_{i=1}^{n-1} \qty(\frac{1}{y-i+\frac{1}{2}} + \frac{1}{y-i-\frac{1}{2}}).
    \end{align}
    Therefore, we have
    \begin{align}
        \widetilde{G}\qty(n-\frac{1}{2}, y)
        & \begin{aligned}[t]
            & = - \frac{\cos\pi y}{2}\, \qty(A\qty(y+n-\frac{1}{2}) - A\qty(y+\frac{1}{2})) \\
            & \qquad\qquad - \frac{\cos\pi y}{2}\, \qty(A\qty(y-n+\frac{1}{2}) - A\qty(y-\frac{1}{2})) \\
            & \qquad\qquad - \frac{\cos\pi y}{4}\, \qty(\frac{1}{y-\frac{1}{2}} - \frac{1}{y+\frac{1}{2}})
        \end{aligned} \\
        & = - \frac{\cos\pi y}{2} \sideset{}{^*}\sum_{\substack{|i| < n \\ i: \text{half-integer}}} \frac{\mathrm{sgn}\, i}{y - i}.
    \end{align}
    This completes the proof.
\end{prf}

\begin{prf}[of Theorem~\ref{thm:bsidentity}]
    Let $g(x) = G(x-\frac{1}{2}, y) - \widetilde{G}(x-\frac{1}{2}, y)$ for each fixed $y$.
    By Proposition~\ref{prop:anG}, $g$ is an entire function
    satisfying $g(x) = O(e^{\pi|x|})$ as $|x| \to \infty$.
    Moreover, by Proposition~\ref{prop:intG}, $g$ vanishes at all integers.
    Therefore, by \cite[Th\'eor\`eme fondamental]{zbMATH02591652},
    $g(x)$ is a constant multiple of $\sin\pi x$.
    Finally, Proposition~\ref{prop:diagG} yields $g(x) = 0$,
    which implies $G(x, y) = \widehat{G}(x, y)$ for $|x|, |y| < \frac{1}{2}$ by Proposition~\ref{prop:wiG}.
\end{prf}

\section{Motivic setup}
In this section, we briefly recall the notion of motivic alternating multiple zeta values,
and introduce motivic multiple $s$ values and motivic multiple $\bar{s}$ values.
\subsection{Mixed Tate motives over $\Z$ or $\Z[1/2]$}
Let $\MT(\Z[1/i])$ be the category of mixed Tate motives over $\Z[1/i]$ for $i = 1, 2$.
The category $\MT(\Z[1/i])$ is a Tannakian category over $\Q$ equipped with two fiber functors, 
the de Rham fiber functor $\omega_{\dR}$ and the Betti fiber functor $\omega_{\B}$.
Let $G_{\dR}^i$ denote the Tannaka group with respect to the de Rham fiber functor $\omega_\dR$.
It admits a semidirect product decomposition
\[
G_\dR^i \cong U_\dR^i \rtimes \Gm,
\]
where $U_\dR^i$ is the pro-unipotent part of $G_\dR^i$.
The function ring $\A^i = \O(U_\dR^i)$ is a Hopf algebra over $\Q$, 
and the coproduct is denoted by $\Delta \colon \A^i \to \A^i \otimes_\Q \A^i$.
The action of $\Gm$ endows $\A^i$ with a grading by non-negative integers, called the weight:
\[
\A^i = \bigoplus_{N \ge 0} \A_N^i.
\]
Let $\P^i$ be the $\Q$-algebra of effective motivic periods of $\MT(\Z[1/i])$.
It is also graded by non-negative integers, called the weight.
The algebra $\P^i$ is an $\A^i$-comodule, with coaction
\[
\Delta \colon \P^i \longrightarrow \A^i \otimes_\Q \P^i,
\]
which preserves the weight.
We have a ring homomorphism
\[
\per \colon \P^i \longrightarrow \C,
\]
called the period map.
Let $\Lm \in \P^i_1$ denote the Lefschetz motivic period, 
which is mapped to $2\pi i$ by the period map.
Then there is a non-canonical isomorphism of $\Q$-algebras (\cite[\S 2.1.2 (2.9)]{glanois:tel-01347989})
\[
\P^i \cong \A^i \otimes_\Q \Q[\Lm].
\]
This isomorphism is also compatible with the $\A^i$-comodule structures,
where $\A^i \otimes \Q[\Lm]$ is regarded as an $\A^i$-comodule with $\Delta(\Lm)=1\otimes\Lm$.
By sending $\Lm$ to zero, we obtain the projection
\begin{equation}
    \pia \colon \P^i \longrightarrow \P^i/\Lm \cong \A^i. \label{eq:proj}
\end{equation}
The algebra $\P^i$ has an involution $F_{\infty} \colon \P^i \to \P^i$ induced by complex conjugation,
which is called the real Frobenius.
The Lefschetz motivic period $\Lm$ is anti-invariant under $F_\infty$, i.e., $F_\infty(\Lm) = - \Lm$.
Let $\H^i$ be the subspace of $\P^i$ consisting of $F_\infty$-invariant elements.
Since the real Frobenius is compatible with the coaction, we obtain the map
\[
\Delta \colon \H^i \longrightarrow \A^i \otimes_\Q \H^i.
\]
Furthermore, we have a non-canonical isomorphism (\cite[\S 2.1.2 (2.10)]{glanois:tel-01347989})
\[
\H^i \cong \A^i \otimes_\Q \Q[(\Lm)^2].
\]

Henceforth, we write $\A$, $\P$, and $\H$ for $\A^2$, $\P^2$, and $\H^2$, respectively.
We retain the superscripts only when it is necessary to distinguish between $i=1$ and $i=2$.
By Borel's computation of the $K$-theory of number fields, we have
\[
\mathrm{Ext}^1_{\MT(\Z[1/2])}(\Q(0), \Q(n)) =
\begin{cases}
    \Q & \text{if}\ n = 2r+1 \ge 1, \\
    0 & \text{otherwise}.
\end{cases}
\]
Consequently, the Hopf algebra $\A$ is (non-canonically) isomorphic to the following graded Hopf algebra.
Let $A$ denote the graded Hopf algebra
\[
A = \Q\langle f_1, f_3, f_5, \dots \rangle.
\]
As a vector space, $A$ is the free associative algebra generated by 
$f_{2r+1}\ (r \ge 0)$ in degree $2r+1$.
The product is the shuffle product.
The coproduct is given by deconcatenation, i.e.,
\[
\Delta(f_{r_1}\cdots f_{r_n}) = \sum_{k = 0}^{n} f_{r_1}\cdots f_{r_k} \otimes f_{r_{k+1}}\cdots f_{r_n}.
\]
Then, there is a non-canonical isomorphism of graded Hopf algebras (\cite[\S 2.3 (2.35)]{glanois:tel-01347989})
\[
\A \cong A.
\]
Let $H$ be the $A$-comodule $H = A \otimes \Q[g_1^2]$, where $g_1$ is a formal symbol in degree 1,
and the coaction $\Delta \colon H \to A \otimes H$ is determined by $\Delta(g_1^2) = 1 \otimes g_1^2$.
Then there is an isomorphism of graded comodules
\begin{equation}
    \H \cong H, \label{eq:isom}
\end{equation}
which maps $(\Lm)^2$ to $g_1^2$.

\subsection{Motivic iterated integral}
For $a_0, \dots, a_{n+1} \in \C$ with $a_0 \ne a_1$ and $a_n \ne a_{n+1}$,
and a piecewise smooth path $\gamma \colon [0, 1] \to \C$ such that
$\gamma(0) = a_0,\; \gamma(1) = a_{n+1},\; \gamma((0, 1)) \subseteq \C \setminus \{a_1, \dots, a_n\}$,
the iterated integral is defined by 
\[
I_{\gamma}(a_0; a_1, \dots, a_n; a_{n+1}) = \int_{0<t_1<\dots<t_n<1} \prod_{i=1}^{n} \frac{d\gamma(t_i)}{\gamma(t_i)-a_i}.
\]
Using the theory of tangential base points, the iterated integral can also be defined when $a_0=a_1$ or $a_n=a_{n+1}$.
The alternating multiple zeta value can be expressed as the following iterated integral (e.g.\ \cite[\S 1, (1.3)]{glanois:tel-01347989}):
\[
\zeta\binom{k_1, \dots, k_r}{\epsilon_1, \dots, \epsilon_r} = (-1)^r\, I_{\dch}(0; \eta_1, \{0\}^{k_1-1}, \dots, \eta_r, \{0\}^{k_r-1}; 1),
\]
where $\dch$ denotes the straight path from 0 to 1, and $\eta_i = \epsilon_i \cdots \epsilon_r$.

When $a_0, \dots, a_{n+1} \in \{0, \pm 1\}$, the iterated integral $I_{\gamma}(a_0; a_1, \dots, a_n; a_{n+1})$
is a period of $\MT(\Z[1/2])$ (\cite[Corollary 1.4]{MR2140264}), and it admits a motivic lift
\[
\Im_{\gamma}(a_0; a_1, \dots, a_n; a_{n+1}) \in \P_n.
\]
Its image under the period map is 
\[
\per(\Im_{\gamma}(a_0; a_1, \dots, a_n; a_{n+1})) = I_{\gamma}(a_0; a_1, \dots, a_n; a_{n+1}).
\]
For an index $(k_1, \dots, k_r) \in \Z_{>0}^r$ and signs $\epsilon_1, \dots, \epsilon_r \in \{\pm1\}$,
the motivic alternating multiple zeta value is defined by the following motivic iterated integral:
\[
\zm\binom{k_1, \dots, k_r}{\epsilon_1, \dots, \epsilon_r} = (-1)^r\, \Im_{\dch}(0; \eta_1, \{0\}^{k_1-1}, \dots, \eta_r, \{0\}^{k_r-1}; 1),
\]
where $\eta_i = \epsilon_i \cdots \epsilon_r$.
Since the path $\dch$ is invariant under complex conjugation,
the above motivic iterated integral is invariant under $F_\infty$
and defines an element of $\H^2$.
When the convergence condition $(k_r, \epsilon_r) \neq (1, 1)$ holds, its image under the period map is
\[
\per\qty(\zm\binom{k_1, \dots, k_r}{\epsilon_1, \dots, \epsilon_r}) = \zeta\binom{k_1, \dots, k_r}{\epsilon_1, \dots, \epsilon_r}.
\]
We often write $k_i$ for $\binom{k_i}{1}$ and
$\overline{k}_i$ for $\binom{k_i}{-1}$ as in the case of alternating multiple zeta values.
When all signs $\epsilon_i = 1$, we write
\[
\zm\binom{k_1, \dots, k_r}{1, \dots, 1} = \zm(k_1, \dots, k_r) \in \H^1,
\]
and refer to it as a motivic multiple zeta value.
The $\Q$-algebra $\H^1$ (resp.\ $\H^2$) is generated by motivic MZVs (resp.\ motivic AMZVs),
as proved by Brown for $\H^1$ (\cite[Corollary 1.2]{MR2993755}), 
and by Deligne for $\H^2$ (\cite[Th\'eor\`eme 7.2]{MR2737978}).
\begin{dfn}
    For an index $\kk = (k_1, \dots, k_r) \in \Z_{>0}^r$, a motivic multiple $s$ value $\sm(\kk)$ is defined by
    \[
    \sm(\kk) = \sum_{\substack{\epsilon_i \in \{\pm 1\} \ (k_i > 1) \\ \epsilon_i = -1 \ (k_i = 1)}} 
    \zm\binom{k_1, \dots, k_r}{\epsilon_1, \dots, \epsilon_r}.
    \]
    A motivic multiple $\bar{s}$ value $\bsm(\kk)$ is defined by
    \[
    \bsm(\kk) = \sum_{\substack{\epsilon_i \in \{\pm 1\} \ (k_i > 1) \\ \epsilon_i = -1 \ (k_i = 1)}} 
    \epsilon_1\cdots\epsilon_r\, \zm\binom{k_1, \dots, k_r}{\epsilon_1, \dots, \epsilon_r}.
    \]
\end{dfn}
By \cite[Proposition 10]{MR4377309}, if each $k_i$ is greater than 1, we have
\[
\sm(\kk) = 2^{\dep\kk - \wt\kk}\, \zm(\kk).
\]
In particular, 
\begin{equation}
    \sm(\{2\}^r) = 2^{-r}\, \zm(\{2\}^r). \label{eq:smzm}
\end{equation}
Note that $\sm(\{2\}^r)$ is a rational multiple of $\zm(2)^r$ by \cite[Lemma 3.4]{MR2993755}.

Let $\Ia(a_0; a_1, \dots, a_n; a_{n+1})$ denote the image of $\Im_{\gamma}(a_0; a_1, \dots, a_n; a_{n+1})$ 
under the projection $\pia \colon \P \to \A$.
Note that $\Ia(a_0; a_1, \dots, a_n; a_{n+1})$ does not depend on the choice of $\gamma$.

Motivic iterated integrals satisfy the following relations (see e.g., \cite[\S 4.2]{MR4906315}):
\begin{enumerate}
    \item \label{enum:unit} $\Ia(a_0; a_1) = 1$,
    \item \label{enum:triv} $\Ia(a_0; a_1, \dots, a_n; a_{n+1}) = 0$ if $a_0 = a_{n+1}$ and $n \ge 1$,
    \item \label{enum:zero} $\Ia(a_0; \{0\}^n; a_{n+1}) = 0$ if $n \ge 1$,
    \item \label{enum:pathcomp} Path composition: for $x \in \{0, \pm1\}$,
          \[
          \Ia(a_0; a_1, \dots, a_n; a_{n+1}) = \sum_{k=0}^{n} \Ia(a_0; a_1, \dots, a_k; x)\, \Ia(x; a_{k+1}, \dots, a_n; a_{n+1}),
          \]
    \item \label{enum:pathrev} Path reversal: 
          $\Ia(a_0; a_1, \dots, a_n; a_{n+1}) = (-1)^n\, \Ia(a_{n+1}; a_n, \dots, a_1; a_0)$,
    \item \label{enum:homothety} Homothety:
          $\Ia(a_0; a_1, \dots, a_n; a_{n+1}) = \Ia(-a_0; -a_1, \dots, -a_n; -a_{n+1})$.
\end{enumerate}

The coaction $\Delta \colon \P \to \A \otimes_\Q \P$ on motivic iterated integrals 
is given by the following explicit formula,
which goes back to Goncharov's computation of the coproduct (\cite[Theorem 1.2]{MR2140264})
and was reformulated by Brown in terms of the coaction (\cite[Theorem 2.4]{MR2993755}):
\begin{equation}
    \begin{aligned}
        & \quad \Delta(\Im_{\gamma}(a_0; a_1, \dots, a_n; a_{n+1})) \\
        & = \sum_{\substack{0 = i_0 < i_1 < \cdots \\ < i_k < i_{k+1} = n+1}} 
            \Bigg(\prod_{p=0}^{k} \Ia(a_{i_p}; a_{i_p+1}, \dots, a_{i_{p+1}-1}; a_{i_{p+1}})\Bigg) 
            \otimes \Im_{\gamma}(a_0; a_{i_1}, \dots, a_{i_k}; a_{n+1}). \label{eq:coaction}
    \end{aligned}
\end{equation}

As in the case of motivic iterated integrals, we denote $\pia(\zm(\kk)), \pia(\sm(\kk))$, and $\pia(\bsm(\kk))$ 
by $\za(\kk), \sa(\kk)$, and $\bsa(\kk)$, respectively.

\begin{prop} \label{prop:chooseisom}
    We can choose an isomorphism~\eqref{eq:isom} satisfying
    \[
    \zm(\overline{2r+1}) \longmapsto f_{2r+1}.
    \]
\end{prop}

\begin{prf}
    By the definition of motivic AMZVs, we have 
    \[
    \zm(\overline{2r+1}) = -\Im_{\dch}(0; -1, \{0\}^{2r}; 1).
    \]
    By the property~\ref{enum:unit}), \ref{enum:triv}), and \ref{enum:zero}), 
    it follows from the formula~\eqref{eq:coaction} that
    \begin{equation}
        \Delta(\zm(\overline{2r+1})) = 1 \otimes \zm(\overline{2r+1}) + \za(\overline{2r+1}) \otimes 1. \label{eq:Delta}
    \end{equation}
    From the structure of $H$, the element $\zm(\overline{2r+1})$ is sent to $c_r f_{2r+1}$ for some $c_r \in \Q$ under the isomorphism.
    Applying the period map yields $c_r \ne 0$, this completes the proof.
\end{prf}

By \cite[Lemma 4.5]{MR3271288}, we have
\[
\zm(2) = - \frac{(\Lm)^2}{24}.
\]
Thus, we have an isomorphism
\[
\phi \colon \H \overset{\sim}{\longrightarrow} H = \Q\langle f_1, f_3, \dots \rangle \otimes_\Q \Q[f_2],
\]
which maps $\zm(\overline{2r+1})$ to $f_{2r+1}$ and $\zm(2)$ to $f_2 = - g_1^2/24$.

\subsection{Infinitesimal coaction}
The explicit formula for the coaction given in $\eqref{eq:coaction}$ contains rich information, 
but it is still complicated to compute.
In \cite{MR2993755}, Brown introduced a more tractable map called the infinitesimal coaction.
We denote the reduced coaction $\Delta - 1 \otimes \id \colon \H \to \A_{>0} \otimes \H$ by $\overline{\Delta}$.
\begin{dfn}
    Let $\L$ be the Lie coalgebra $\A_{>0}/(\A_{>0})^2$.
    For $r \ge 1$, define a map
    \[
    D_r \colon \H \longrightarrow \L_r \otimes_\Q \H
    \]
    to be the composition
    \[
    \H \overset{\overline{\Delta}}{\longrightarrow} \A_{>0} \otimes_\Q \H \longrightarrow \L \otimes_\Q \H \longrightarrow \L_r \otimes_\Q \H.
    \]
    This is a derivation, i.e.,
    \[
    D_r(\xi_1 \xi_2) = (1 \otimes \xi_1)D_r(\xi_2) + (1 \otimes \xi_2)D_r(\xi_1) \quad \text{for all} \ \xi_1, \xi_2 \in \H.
    \]
\end{dfn}
Let $\Il(a_0; a_1, \dots, a_n; a_{n+1})$ be the image of $\Ia(a_0; a_1, \dots, a_n; a_{n+1})$ under the natural projection $\A \to \L$.
Similarly, we write $\zl(\kk), \sl(\kk)$ and $\bsl(\kk)$ for the image of $\za(\kk), \sa(\kk)$ and $\bsa(\kk)$ under the projection, respectively.
Note that $\sl(\{2\}^d) = 0$ for $d \ge 0$, since $\sm(\{2\}^d)$ is a rational multiple of $\zm(2)^d$.
By the formula~$\eqref{eq:coaction}$, we have
\begin{equation}
    \begin{aligned}
        & \quad D_r(\Im_{\gamma}(a_0; a_1, \dots, a_n; a_{n+1})) \\
        & = \sum_{p = 0}^{n-r} \Il(a_p; a_{p+1}, \dots, a_{p+r}; a_{p+r+1})
        \otimes \Im_{\gamma}(a_0; a_1, \dots, a_p, a_{p+r+1}, \dots, a_n; a_{n+1}). \label{eq:infcoaction}
    \end{aligned}
\end{equation}
We express the right-hand side of the above equation as the following diagram.
\begin{center}
    \vspace{0.8cm}
    $\Im_{\gamma}(a_0; a_1, \dots, \tikznode{ap}{a_p}, \underbrace{a_{p+1}, \dots, a_{p+r}}_{\text{block of length}\ r}, \tikznode{apr}{a_{p+r+1}}, \dots, a_n; a_{n+1})$
    \begin{tikzpicture}[remember picture, overlay] 
        \draw ([yshift=3pt]ap.north) |- ++(0,0.4) -| ([yshift=3pt]apr.north);
    \end{tikzpicture}
    \vspace{0.3cm}
\end{center}
By abuse of notation, we also write $D_r$ for the map
\[
H \overset{\overline{\Delta}}{\longrightarrow} A_{>0} \otimes_\Q H \longrightarrow L \otimes_\Q H \longrightarrow L_r \otimes_\Q H,
\]
where $L$ is the Lie coalgebra $A_{>0}/(A_{>0})^2$.

\subsection{Coradical filtration}
\begin{dfn}
    For $d \ge -1$, we define the coradical filtration $C_d\H$ of $\H$ recursively by
    $C_{-1}\H = 0$ and
    \[
    C_d\H = \{a \in \H \mid \overline{\Delta}(a) \in \A \otimes C_{d-1}\H\}.
    \]
    Similarly, we define the coradical filtration $C_d\A$ of $\A$ by
    $C_{-1}\A = 0$ and
    \[
    C_d\A = \{a \in \A \mid \overline{\Delta}(a) \in \A \otimes C_{d-1}\A\}.
    \]
    We denote $C_d\H/C_{d-1}\H$ (resp.\ $C_d\A/C_{d-1}\A$) by $\gr_d\H$ (resp.\ $\gr_d\A$).
    We use the same notation for $H$ and $A$.
\end{dfn}
For $d \ge 1$, the coradical filtration satisfies
\[
\overline{\Delta}(C_d\H) \subseteq \sum_{i=1}^{d} C_i\A \otimes C_{d-i}\H.
\]
Thus, we obtain the map
\[
\widetilde{\Delta} \colon \gr_d\H \overset{\overline{\Delta}}{\longrightarrow} C_1\A \otimes_\Q \gr_{d-1}\H \longrightarrow C_1\L \otimes_\Q \gr_{d-1}\H,
\]
where $C_1\L$ is the image of $C_1\A$ under the projection $\A \to \L$.
\begin{lem} \label{lem:grcoaction}
    The map $\widetilde{\Delta}$ is an isomorphism.
\end{lem}

\begin{prf}
    By the isomorphism $\phi \colon \H \to H$, it suffices to show the same statement for $H$.
    By the definition of the coradical filtration, we have
    \[
    C_d A = \bigoplus_{0 \le d' \le d}\, \bigoplus_{r_1, \dots, r_{d'} \ge 1,\, \text{odd}} \Q\, f_{r_1}\cdots f_{r_{d'}}.
    \]
    Since the coaction acts trivially on $g_1$, it follows that
    \begin{equation}
        \gr_d H = \bigoplus_{k \ge 0} \bigoplus_{r_1, \dots, r_d \ge 1,\, \text{odd}} \Q\, f_{r_1}\cdots f_{r_d}\, g_1^{2k}. \label{eq:grU}
    \end{equation}
    Since the coproduct of $A$ is given by deconcatenation, we have
    \[
    \widetilde{\Delta}(f_{r_1}\cdots f_{r_d}\, g_1^{2k}) = \overline{f}_{\! r_1} \otimes f_{r_2}\cdots f_{r_d}\, g_1^{2k},
    \]
    where $\overline{f}_{\! r} \in L$ denotes the image of $f_r$ under the projection $A \to L$.
    Since the elements $\overline{f}_{\! r} \ (r \ge 1,\, \text{odd})$ are linearly independent, we obtain
    \[
    C_1 L = \bigoplus_{r_1 \ge 1,\, \text{odd}} \Q\, \overline{f}_{\! r_1}.
    \]
    Therefore, the map $\widetilde{\Delta} \colon \gr_d H \to C_1 L \otimes \gr_{d-1}H$ is an isomorphism.
\end{prf}
By the proof of Lemma~\ref{lem:grcoaction}, the structure of $C_1 L$ in each weight is
\[
C_1 L_N =
\begin{cases}
    \Q\, \overline{f}_{\! N} & \text{if} \ N: \text{odd}, \\
    0 & \text{if} \ N: \text{even}.
\end{cases}
\]
Thus, via the isomorphism $\phi$, the structure of $C_1\L$ in each weight is
\[
C_1\L_N =
\begin{cases}
    \Q\, \zl(\overline{N}) & \text{if} \ N: \text{odd}, \\
    0 & \text{if} \ N: \text{even}.
\end{cases}
\]
Therefore, we obtain the isomorphism
\begin{align}
\widetilde{\Delta} \colon \gr_d\H_N & \longrightarrow (C_1\L \otimes \gr_{d-1}\H)_N \\
& = \bigoplus_{1 \le 2r+1 \le N} C_1\L_{2r+1} \otimes \gr_{d-1}\H_{N-(2r+1)} \\
& = \bigoplus_{1 \le 2r+1 \le N} \Q\, \zl(\overline{2r+1}) \otimes \gr_{d-1}\H_{N-(2r+1)}.
\end{align}

\begin{lem} \label{lem:corad}
    Let $d \ge 1$.
    For $\xi \in \H_N$, $\xi \in C_d\H_N$ is equivalent to
    \[
    \forall r \ge 0, \ D_{2r+1}(\xi) \in \L_{2r+1} \otimes C_{d-1}\H.
    \]
\end{lem}

\begin{prf}
    By the isomorphism $\phi$, it suffices to show the same assertion for $H$.
    The ``only~if'' direction follows immediately from the definition of the coradical filtration, 
    and we shall use it in the proof of the converse implication below.
    Let $\xi \in H_N$.
    Assume that $D_{2r+1}(\xi) \in L_{2r+1} \otimes C_{d-1}H$ for all $r \ge 0$.
    The element $\xi \in H_N$ can be uniquely written in the form
    \[
    \xi = \sum_{1 \le 2r+1 \le N} f_{2r+1}\, v_r,
    \]
    where $v_r \in H_{N-(2r+1)}$, and the multiplication $f_{2r+1} v_r$ is the concatenation.
    Then, we have
    \[
    D_1(\xi) = \overline{f}_{\! 1} \otimes v_0,
    \]
    which implies $v_0 \in C_{d-1}H$ by the assumption.
    Thus, $f_1v_0$ lies in $C_d H$.
    Next, we have
    \[
    D_3(\xi - f_1v_0) = \overline{f}_{\! 3} \otimes v_1.
    \]
    Since $f_1v_0 \in C_d H$, we have $D_3(f_1v_0) \in L_3 \otimes C_{d-1}H$.
    Therefore, it follows from the assumption that $D_3(\xi - f_1v_0)$ also lies in $L_3 \otimes C_{d-1}H$, which implies $v_1 \in C_{d-1}H$.
    Iterating this argument, for all $1 \le 2r+1 \le N$, we have $v_r \in C_{d-1}H$, which implies $\xi$ belongs to $C_d H$.
\end{prf}

In what follows, we only consider the motivic iterated integral $\Im_\gamma(-)$ 
for $\gamma=\dch$. 
Thus, we omit $\dch$ and simply write $\Im(-)$.
\begin{prop} \label{prop:deriv}
    Let $d \ge 0$ and $l_0, \dots, l_d \ge 0$.
    Then we have
    \[
    \sm(\{2\}^{l_0}, 1, \{2\}^{l_1}, 1, \dots, \{2\}^{l_{d-1}}, 1, \{2\}^{l_d}) \in C_d\H.
    \]
\end{prop}

\begin{prf}
    We proceed by induction on the degree $d$.
    When $d = 0$, since $\sm(\{2\}^{l_0})$ is a rational multiple of $\zm(2)^{l_0}$,
    we have $\phi(\sm(\{2\}^{l_0})) \in \Q\, f_2^{l_0} = C_0 H_{2l_0}$,
    which implies $\sm(\{2\}^{l_0}) \in C_0\H$.

    Let $d \ge 1$, and assume that the assertion holds for all $0 \le d' < d$.
    Set $\kk = (\{2\}^{l_0}, 1, \dots, 1, \{2\}^{l_d})$.
    By Lemma~\ref{lem:corad}, it suffices to show that 
    \[
    \forall r \ge 0, \ D_{2r+1}(\sm(\kk)) \in \L_{2r+1} \otimes C_{d-1}\H.
    \]
    By the definition of motivic M$s$Vs and motivic AMZVs, we have
    \begin{align}
        (-1)^{\dep\kk}\, \sm(\kk)
        & = \sum_{\substack{\eta_{k, l} \in \{\pm1\} \\ 0 \le k \le d \\ 1 \le l \le l_i}}
            \Im(0; \eta_{0, 1}, 0, \dots, \eta_{0, l_0}, 0, - \eta_{1, 1}, \eta_{1, 1}, 0, \\[-5ex]
        & \qquad\qquad\qquad\qquad\qquad \dots, \eta_{d-1, l_{d-1}}, 0, -\eta_{d, 1}, \eta_{d, 1}, 0, \dots, \eta_{d, l_d}, 0; 1).
    \end{align}
    For convenience, set
    \[
    \eta_{i,l_i+1} \coloneqq -\eta_{i+1,1} \quad \text{for} \ 0 \le i \le d-1, \qquad 
    \eta_{d,l_d+1} \coloneqq 1.
    \]
    We apply the map $D_{2r+1}$ to the right-hand side and compute the resulting expression using the formula \eqref{eq:infcoaction}.
    First, we consider the terms which do not contain the sequence $-\eta_{i, 1}, \eta_{i, 1}$.
    \begin{center}
        \vspace{1.0cm}
        $\Im(0; \dots, \eta_{i, 1}, 0, \dots, \tikznode{etais}{\eta_{i, s}}, \tikznode{zerois}{0}, \dots, \tikznode{etait}{\eta_{i, t}}, \tikznode{zeroit}{0}, \dots, \eta_{i, l_{i+1}}, \dots ; 1)$
        \begin{tikzpicture}[remember picture, overlay] 
            \draw ([yshift=3pt]zerois.north) -- ++(0,0.4) -| node[pos=0, left] {(A)} ([yshift=3pt]zeroit.north);
            
            \draw ([yshift=-3pt]etais.south) -- ++(0,-0.4) -| node[pos=0.9, below right] {(B)} ([yshift=-3pt]etait.south);
        \end{tikzpicture}
        \vspace{1.0cm}
    \end{center}
    Note that no term starting at 0 and ending at $\eta_{i, s}$ appears because of parity.
    The term (A) vanishes by \ref{enum:triv}).
    The term (B) is not zero in general, but after summing over $\eta_{i,s+1}, \ldots, \eta_{i,t-1} \in \{\pm1\}$,
    it vanishes by \cite[Lemma 9]{MR4377309}.

    Next, we consider the terms containing the sequence $-\eta_{i, 1}, \eta_{i, 1}$ and having 0 as one of their endpoints.
    \begin{center}
        \vspace{1.0cm}
        $\Im(0; \dots, \eta_{i, 1}, 0, \dots, \tikznode{etais}{\eta_{i, s}}, \tikznode{zerois}{0}, \dots, \eta_{i, l_i}, 0, -\eta_{i+1, 1}, \eta_{i+1, 1}, 0, \dots, \tikznode{zerojt}{0}, \tikznode{etajt1}{\eta_{j, t+1}}, \dots, \eta_{j, l_j}, \dots ; 1)$
        \begin{tikzpicture}[remember picture, overlay] 
            \draw ([yshift=3pt]zerois.north) -- ++(0,0.4) -| node[pos=0, left] {(C)} ([yshift=3pt]etajt1.north);
            
            \draw ([yshift=-3pt]etais.south) -- ++(0,-0.4) -| node[pos=0.9, below right] {(D)} ([yshift=-3pt]zerojt.south);
        \end{tikzpicture}
        \vspace{1.0cm}
    \end{center}
    For the term (C), by the property \ref{enum:homothety}), we have
    \begin{align}
        & \quad \sum_{\substack{\eta_{k, l} \in \{\pm1\} \\ 0 \le k \le d \\ 1 \le l \le l_i}}
            \Big(\Il(0; \eta_{i, s+1}, 0, \dots, \eta_{i, l_i}, 0, -\eta_{i+1, 1}, \eta_{i+1, 1}, 0, \dots, \eta_{j, t}, 0; \eta_{j, t+1}) \\[-2ex]
        & \qquad\qquad\qquad \otimes \Im(0; \dots, \eta_{i, 1}, 0, \dots, \eta_{i,s}, 0, \eta_{j, t+1}, \dots, \eta_{j, l_j}, 0, \dots; 1)\Big) \\[2ex]
        & = \sum_{\substack{\eta_{k, l} \in \{\pm1\} \\ 0 \le k \le d \\ 1 \le l \le l_i}}
            \Big(\Il(0; \eta_{j, t+1}^{-1}\eta_{i, s+1}, 0, \dots, \eta_{j, t+1}^{-1}\eta_{i, l_i}, 0, \dots, \eta_{j, t+1}^{-1}\eta_{j, t}, 0; 1) \\[-2ex]
        & \qquad\qquad\qquad \otimes \Im(0; \dots, \eta_{i, 1}, 0, \dots, \eta_{i,s}, 0, \eta_{j, t+1}, \dots, \eta_{j, l_j}, 0, \dots; 1)\Big) \\[2ex]
        & = \sum_{\eta_{i, s+1}, \dots, \eta_{j, t} \in \{\pm1\}}
            \Il(0; \eta_{i, s+1}, 0, \dots, \eta_{i, l_i}, 0, -\eta_{i+1, 1}, \eta_{i+1, 1}, 0, \dots, \eta_{j, t}, 0; 1) \\[1ex]
        & \qquad \otimes \sum_{\substack{\eta_{0, 1}, \dots, \eta_{i, s} \\ \eta_{j, t+1}, \dots, \eta_{d, l_d}} \in \{\pm1\}}
            \Im(0; \dots, \eta_{i, 1}, 0, \dots, \eta_{i,s}, 0, \eta_{j, t+1}, \dots, \eta_{j, l_j}, 0, \dots; 1) \\[1ex]
        & = \pm \sum_{\eta_{i, s+1}, \dots, \eta_{j, t} \in \{\pm1\}}
            \Il(0; \eta_{i, s+1}, 0, \dots, \eta_{i, l_i}, 0, -\eta_{i+1, 1}, \eta_{i+1, 1}, 0, \dots, \eta_{j, t}, 0; 1) \\
        & \qquad\qquad\qquad \otimes \sm(\{2\}^{l_0}, \dots, \{2\}^{l_{i-1}}, 1, \{2\}^{s+l_j-t}, 1, \{2\}^{l_{j+1}}, \dots, \{2\}^{l_d}).
    \end{align}
    Similarly, for the term (D), we have
    \begin{align}
        & \quad \sum_{\substack{\eta_{k, l} \in \{\pm1\} \\ 0 \le k \le d \\ 1 \le l \le l_i}}
            \Big(\Il(\eta_{i, s}; 0, \eta_{i, s+1}, \dots, \eta_{i, l_i}, 0, -\eta_{i+1, 1}, \eta_{i+1, 1}, 0, \dots, \eta_{j, t}; 0) \\[-2ex]
        & \qquad\qquad\qquad \otimes \Im(0; \dots, \eta_{i, 1}, 0, \dots, \eta_{i,s}, 0, \eta_{j, t+1}, \dots, \eta_{j, l_j}, 0, \dots; 1)\Big) \\[2ex]
        & = \sum_{\substack{\eta_{k, l} \in \{\pm1\} \\ 0 \le k \le d \\ 1 \le l \le l_i}}
            \Big(\Il(1; 0, \eta_{i, s}^{-1}\eta_{i, s+1}, \dots, \eta_{i, s}^{-1}\eta_{i, l_i}, 0, \dots, \eta_{i, s}^{-1}\eta_{j, t}; 0) \\[-2ex]
        & \qquad\qquad\qquad \otimes \Im(0; \dots, \eta_{i, 1}, 0, \dots, \eta_{i,s}, 0, \eta_{j, t+1}, \dots, \eta_{j, l_j}, 0, \dots; 1)\Big) \\[2ex]
        & = \sum_{\eta_{i, s+1}, \dots, \eta_{j, t} \in \{\pm1\}}
            \Il(1; 0, \eta_{i, s+1}, \dots, \eta_{i, l_i}, 0, -\eta_{i+1, 1}, \eta_{i+1, 1}, 0, \dots, \eta_{j, t}; 0) \\[1ex]
        & \qquad \otimes \sum_{\substack{\eta_{0, 1}, \dots, \eta_{i, s} \\ \eta_{j, t+1}, \dots, \eta_{d, l_d}} \in \{\pm1\}}
            \Im(0; \dots, \eta_{i, 1}, 0, \dots, \eta_{i,s}, 0, \eta_{j, t+1}, \dots, \eta_{j, l_j}, 0, \dots; 1) \\
        & = \pm \sum_{\eta_{i, s+1}, \dots, \eta_{j, t} \in \{\pm1\}}
            \Il(1; 0, \eta_{i, s+1}, \dots, \eta_{i, l_i}, 0, -\eta_{i+1, 1}, \eta_{i+1, 1}, 0, \dots, \eta_{j, t}; 0) \\[1ex]
        & \qquad\qquad\qquad \otimes \sm(\{2\}^{l_0}, \dots, \{2\}^{l_{i-1}}, 1, \{2\}^{s+l_j-t}, 1, \{2\}^{l_{j+1}}, \dots, \{2\}^{l_d}).
    \end{align}
    Since we consider the terms containing the sequence $-\eta_{i, 1}, \eta_{i, 1}$, it follows that $i < j$.
    Therefore, the right tensor components of both terms (C) and (D) are of the form 
    $\sm(\{2\}^{l'_0}, 1, \dots, 1, \{2\}^{l'_{d'}})$ with $0 \le d' < d$.
    Thus, they belong to $C_{d-1}\H$ by the induction hypothesis.

    Finally, we consider the terms containing the sequence $-\eta_{i, 1}, \eta_{i, 1}$ whose endpoints are both nonzero.
    \begin{center}
        \vspace{1.4cm}
        $\Im(0; \dots, \tikznode{-etai1}{-\eta_{i, 1}}, \eta_{i, 1}, 0, \dots, \tikznode{etais}{\eta_{i, s}}, 0, \dots, 0, \tikznode{etajt1}{\eta_{j, t+1}}, \dots, \eta_{j, l_j}, 0, -\eta_{j+1, 1}, \tikznode{etaj11}{\eta_{j+1, 1}}, \dots ; 1)$
        \begin{tikzpicture}[remember picture, overlay] 
            \draw ([xshift=-2pt, yshift=3pt]-etai1.north) -- ++(0,0.8) -| node[pos=0, left] {(H)} ([yshift=3pt]etaj11.north);

            \draw ([xshift=2pt, yshift=3pt]-etai1.north) -- ++(0,0.4) -| node[pos=0.6, right] {(F)} ([yshift=3pt]etajt1.north);
            
            \draw ([xshift=2pt, yshift=-3pt]etais.south) -- ++(0,-0.4) -| node[pos=0.6, right] {(E)} ([yshift=-3pt]etajt1.south);

            \draw ([xshift=-2pt, yshift=-3pt]etais.south) -- ++(0,-0.8) -| node[pos=0.9, below right] {(G)} ([yshift=-3pt]etaj11.south);
        \end{tikzpicture}
        \vspace{1.4cm}
    \end{center}
    Let $n \ge 1$ and $a_0, \dots, a_{n+1}, x \in \{0, \pm1\}$.
    Since the decomposable elements vanish in $\L = \A_{>0}/(\A_{>0})^2$, it follows that 
    \[
    \Il(a_0; a_1, \dots, a_k; x)\, \Il(x; a_{k+1}, \dots, a_n; a_{n+1}) = 0 \quad \text{for}\ 0 < k < n.
    \]
    Therefore, by the path composition property \ref{enum:pathcomp}), we have
    \begin{equation}
        \Il(a_0; a_1, \dots, a_n; a_{n+1}) = \Il(a_0; a_1, \dots, a_n; x) + \Il(x; a_1, \dots, a_n; a_{n+1}). \label{eq:pathcomp}
    \end{equation}
    For the term (E), if $\eta_{i, s} = \eta_{j, t+1}$, then it vanishes by the property~\ref{enum:triv}).
    Assume that $\eta_{i, s} = - \eta_{j, t+1}$.
    By \eqref{eq:pathcomp}, we obtain
    \begin{align}
        & \quad \sum_{\eta_{i, s+1}, \dots, \eta_{j, t} \in \{\pm1\}} \Il(\eta_{i, s}; 0, \eta_{i, s+1}, \dots, -\eta_{i+1, 1}, \eta_{i+1, 1}, \dots, \eta_{j, t}, 0; -\eta_{i, s}) \\[2ex]
        & = \sum_{\eta_{i, s+1}, \dots, \eta_{j, t} \in \{\pm1\}} \Il(\eta_{i, s}; 0, \eta_{i, s+1}, \dots, -\eta_{i+1, 1}, \eta_{i+1, 1}, \dots, \eta_{j, t}, 0; 0) \\
        & \qquad\qquad + \sum_{\eta_{i, s+1}, \dots, \eta_{j, t} \in \{\pm1\}} \Il(0; 0, \eta_{i, s+1}, \dots, -\eta_{i+1, 1}, \eta_{i+1, 1}, \dots, \eta_{j, t}, 0; -\eta_{i, s}) \\[2ex]
        & = \sum_{\eta_{i, s+1}, \dots, \eta_{j, t} \in \{\pm1\}} \Il(\eta_{i, s}; 0, \eta_{i, s+1}, \dots, -\eta_{i+1, 1}, \eta_{i+1, 1}, \dots, \eta_{j, t}, 0; 0) \\
        & \qquad\qquad + \sum_{\eta_{i, s+1}, \dots, \eta_{j, t} \in \{\pm1\}} \Il(0; 0, -\eta_{i, s+1}, \dots, \eta_{i+1, 1}, -\eta_{i+1, 1}, \dots, -\eta_{j, t}, 0; \eta_{i, s}) \\[2ex]
        & = \sum_{\eta_{i, s+1}, \dots, \eta_{j, t} \in \{\pm1\}} \Il(\eta_{i, s}; 0, \eta_{i, s+1}, \dots, -\eta_{i+1, 1}, \eta_{i+1, 1}, \dots, \eta_{j, t}, 0; 0) \\
        & \qquad\qquad + \sum_{\eta_{i, s+1}, \dots, \eta_{j, t} \in \{\pm1\}} \Il(0; 0, \eta_{i, s+1}, \dots, -\eta_{i+1, 1}, \eta_{i+1, 1}, \dots, \eta_{j, t}, 0; \eta_{i, s}) \\
        & = 0,
    \end{align}
    where we used the homothety in the second step.
    Thus, the term (E) vanishes whether $\eta_{i, s} = \eta_{j, t+1}$ or not.
    For the term (F), it follows from the property~\ref{enum:triv}) and \ref{enum:homothety}) that
    \begin{align}
        & \quad \sum_{\substack{\eta_{k, l} \in \{\pm1\} \\ 0 \le k \le d \\ 1 \le l \le l_i}}
            \Big(\Il(-\eta_{i, 1}; \eta_{i, 1}, 0, \dots, \eta_{i, l_i}, 0, -\eta_{i+1, 1}, \eta_{i+1, 1}, 0, \dots, \eta_{j, t}, 0; \eta_{j, t+1}) \\[-2ex]
        & \qquad\qquad\qquad \otimes \Im(0; \dots, \eta_{i-1, l_{i-1}}, 0, -\eta_{i, 1}, \eta_{j, t+1}, \dots, \eta_{j, l_j}, 0, \dots; 1)\Big) \\[2ex]
        & = \sum_{\substack{\eta_{k, l} \in \{\pm1\} \\ 0 \le k \le d \\ 1 \le l \le l_i \\ \eta_{i, 1} = \eta_{j, t+1}}}
            \Big(\Il(-\eta_{i, 1}; \eta_{i, 1}, 0, \dots, \eta_{i, l_i}, 0, -\eta_{i+1, 1}, \eta_{i+1, 1}, 0, \dots, \eta_{j, t}, 0; \eta_{i, 1}) \\[-4ex]
        & \qquad\qquad\qquad \otimes \Im(0; \dots, \eta_{i-1, l_{i-1}}, 0, -\eta_{j, t+1}, \eta_{j, t+1}, \dots, \eta_{j, l_j}, 0, \dots; 1)\Big) \\[2ex]
        & = \sum_{\substack{\eta_{k, l} \in \{\pm1\} \\ 0 \le k \le d \\ 1 \le l \le l_i \\ \eta_{i, 1} = \eta_{j, t+1}}}
            \Big(\Il(-1; 1, 0, \eta_{i, 1}^{-1}\eta_{i, 2}, 0, \dots, \eta_{i, 1}^{-1}\eta_{i, l_i}, 0, \dots, \eta_{i, 1}^{-1}\eta_{j, t}, 0; 1) \\[-4ex]
        & \qquad\qquad\qquad \otimes \Im(0; \dots, \eta_{i-1, l_{i-1}}, 0, -\eta_{j, t+1}, \eta_{j, t+1}, \dots, \eta_{j, l_j}, 0, \dots; 1)\Big) \\[2ex]
        & = \sum_{\eta_{i, 2}, \dots, \eta_{j, t} \in \{\pm1\}}
            \Il(-1; 1, 0, \eta_{i, 2}, 0, \dots, \eta_{i, l_i}, 0, -\eta_{i+1, 1}, \eta_{i+1, 1}, 0, \dots, \eta_{j, t}, 0; 1) \\[1ex]
        & \qquad \otimes \sum_{\substack{\eta_{0, 1}, \dots, \eta_{i-1, l_{i-1}} \\ \eta_{j, t+1}, \dots, \eta_{d, l_d}} \in \{\pm1\}}
            \Im(0; \dots, \eta_{i-1, l_{i-1}}, 0, -\eta_{j,t+1}, \eta_{j, t+1}, 0, \dots, \eta_{j, l_j}, 0, \dots; 1) \\
        & = \pm \sum_{\eta_{i, 2}, \dots, \eta_{j, t} \in \{\pm1\}}
            \Il(-1; 1, 0, \eta_{i, 2}, 0, \dots, \eta_{i, l_i}, 0, -\eta_{i+1, 1}, \eta_{i+1, 1}, 0, \dots, \eta_{j, t}, 0; 1) \\[1ex]
        & \qquad\qquad\qquad \otimes \sm(\{2\}^{l_0}, \dots, \{2\}^{l_{i-1}}, 1, \{2\}^{l_j-t}, 1, \{2\}^{l_{j+1}}, \dots, \{2\}^{l_d}).
    \end{align}
    Similarly, for the term (G), we have
    \begin{align}
        & \quad \sum_{\substack{\eta_{k, l} \in \{\pm1\} \\ 0 \le k \le d \\ 1 \le l \le l_i}}
            \Big(\Il(\eta_{i, s}; 0, \dots, \eta_{i, l_i}, 0, -\eta_{i+1, 1}, \eta_{i+1, 1}, 0, \dots, 0, -\eta_{j+1, 1}, ; \eta_{j+1, 1}) \\[-2ex]
        & \qquad\qquad\qquad \otimes \Im(0; \dots, \eta_{i, 1}, 0, \dots, \eta_{i, s-1}, 0, \eta_{i, s}, \eta_{j+1, 1}, 0, \dots; 1)\Big) \\[2ex]
        & = \sum_{\substack{\eta_{k, l} \in \{\pm1\} \\ 0 \le k \le d \\ 1 \le l \le l_i \\ \eta_{i, s} = -\eta_{j+1, 1}}}
            \Big(\Il(-\eta_{j+1, 1}; 0, \dots, \eta_{i, l_i}, 0, -\eta_{i+1, 1}, \eta_{i+1, 1}, 0, \dots, 0, -\eta_{j+1, 1}, ; \eta_{j+1, 1}) \\[-4ex]
        & \qquad\qquad\qquad \otimes \Im(0; \dots, \eta_{i, 1}, 0, \dots, \eta_{i, s-1}, 0, -\eta_{j+1, 1}, \eta_{j+1, 1}, 0, \dots; 1)\Big) \\[2ex]
        & = \sum_{\substack{\eta_{k, l} \in \{\pm1\} \\ 0 \le k \le d \\ 1 \le l \le l_i \\ \eta_{i, s} = -\eta_{j+1, 1}}}
            \Big(\Il(-1; 0, \eta_{j+1, 1}^{-1}\eta_{i, s+1}, 0, \dots, \eta_{j+1, 1}^{-1}\eta_{i, l_i}, 0, \dots, \eta_{j+1, 1}^{-1}\eta_{j, l_j}, 0, -1; 1) \\[-4ex]
        & \qquad\qquad\qquad \otimes \Im(0; \dots, \eta_{i, 1}, 0, \dots, \eta_{i, s-1}, 0, -\eta_{j+1, 1}, \eta_{j+1, 1}, 0, \dots; 1)\Big) \\[2ex]
        & = \sum_{\eta_{i, s+1}, \dots, \eta_{j, l_j} \in \{\pm1\}}
            \Il(-1; 0, \eta_{i, s+1}, \dots, \eta_{i, l_i}, 0, -\eta_{i+1, 1}, \eta_{i+1, 1}, 0, \dots, \eta_{j, l_j}, 0, -1; 1) \\[1ex]
        & \qquad \otimes \sum_{\substack{\eta_{0, 1}, \dots, \eta_{i, s-1} \\ \eta_{j+1, 1}, \dots, \eta_{d, l_d}} \in \{\pm1\}}
            \Im(0; \dots, \eta_{i, 1}, 0, \dots, \eta_{i, s-1}, 0, -\eta_{j+1, 1}, \eta_{j+1, 1}, 0, \dots; 1) \\
        & = \pm \sum_{\eta_{i, s+1}, \dots, \eta_{j, l_j} \in \{\pm1\}}
            \Il(-1; 0, \eta_{i, s+1}, \dots, \eta_{i, l_i}, 0, -\eta_{i+1, 1}, \eta_{i+1, 1}, 0, \dots, \eta_{j, l_j}, 0, -1; 1) \\[1ex]
        & \qquad\qquad\qquad \otimes \sm(\{2\}^{l_0}, \dots, \{2\}^{l_{i-1}}, 1, \{2\}^{s-1}, 1, \{2\}^{l_{j+1}}, \dots, \{2\}^{l_d}).
    \end{align}
    For the term (H), it follows from the property~\ref{enum:triv}) and \ref{enum:homothety}) that
    \begin{align}
        & \quad \sum_{\substack{\eta_{k, l} \in \{\pm1\} \\ 0 \le k \le d \\ 1 \le l \le l_i}}
            \Big(\Il(-\eta_{i, 1}; \eta_{i, 1}, 0, \dots, \eta_{i, l_i}, 0, -\eta_{i+1, 1}, \eta_{i+1, 1}, 0, \dots, \eta_{j, l_j}, 0, -\eta_{j+1, 1}; \eta_{j+1, 1}) \\[-2ex]
        & \qquad\qquad\qquad \otimes \Im(0; \dots, \eta_{i-1, l_{i-1}}, 0, -\eta_{i, 1}, \eta_{j+1, 1}, 0, \dots, \eta_{j+1, l_{j+1}}, 0, \dots; 1)\Big) \\[2ex]
        & = \sum_{\substack{\eta_{k, l} \in \{\pm1\} \\ 0 \le k \le d \\ 1 \le l \le l_i \\ \eta_{i, 1} = \eta_{j+1, 1}}}
            \Big(\Il(-\eta_{i, 1}; \eta_{i, 1}, 0, \dots, \eta_{i, l_i}, 0, -\eta_{i+1, 1}, \eta_{i+1, 1}, 0, \dots, \eta_{j, l_j}, 0, -\eta_{i, 1}; \eta_{i, 1}) \\[-4ex]
        & \qquad\qquad\qquad \otimes \Im(0; \dots, \eta_{i-1, l_{i-1}}, 0, -\eta_{j+1, 1}, \eta_{j+1, 1}, 0, \dots, \eta_{j+1, l_{j+1}}, 0, \dots; 1)\Big) \\[2ex]
        & = \sum_{\substack{\eta_{k, l} \in \{\pm1\} \\ 0 \le k \le d \\ 1 \le l \le l_i \\ \eta_{i, 1} = \eta_{j+1, 1}}}
            \Big(\Il(-1; 1, 0, \eta_{i, 1}^{-1}\eta_{i, 2}, 0, \dots, \eta_{i, 1}^{-1}\eta_{i, l_i}, 0, \dots, \eta_{i, 1}^{-1}\eta_{j, l_j}, 0, -1; 1) \\[-4ex]
        & \qquad\qquad\qquad \otimes \Im(0; \dots, \eta_{i-1, l_{i-1}}, 0, -\eta_{j+1, 1}, \eta_{j+1, 1}, 0, \dots, \eta_{j+1, l_{j+1}}, 0, \dots; 1)\Big) \\[2ex]
        & = \sum_{\eta_{i, 2}, \dots, \eta_{j, l_j} \in \{\pm1\}}
            \Il(-1; 1, 0, \eta_{i, 2}, 0, \dots, \eta_{i, l_i}, 0, -\eta_{i+1, 1}, \eta_{i+1, 1}, 0, \dots, \eta_{j, l_j}, 0, -1; 1) \\[1ex]
        & \qquad \otimes \sum_{\substack{\eta_{0, 1}, \dots, \eta_{i-1, l_{i-1}} \\ \eta_{j+1, 1}, \dots, \eta_{d, l_d}} \in \{\pm1\}}
            \Im(0; \dots, \eta_{i-1, l_{i-1}}, 0, -\eta_{j+1, 1}, \eta_{j+1, 1}, 0, \dots; 1)\Big) \\
        & = \pm \sum_{\eta_{i, 2}, \dots, \eta_{j, l_j} \in \{\pm1\}}
            \Il(-1; 1, 0, \eta_{i, 2}, 0, \dots, \eta_{i, l_i}, 0, -\eta_{i+1, 1}, \eta_{i+1, 1}, 0, \dots, \eta_{j, l_j}, 0, -1; 1) \\[1ex]
        & \qquad\qquad\qquad \otimes \sm(\{2\}^{l_0}, \dots, \{2\}^{l_{i-1}}, 1, \{2\}^{l_{j+1}}, \dots, \{2\}^{l_d}).
    \end{align}
    Since we consider the terms containing the sequence $-\eta_{i, 1}, \eta_{i, 1}$, it follows that $i \le j$.
    Moreover, for the terms (F) and (G), we can assume $i < j$ because of parity.
    Therefore, the right tensor components of the terms (F), (G), and (H) are of the form 
    $\sm(\{2\}^{l'_0}, 1, \dots, 1, \{2\}^{l'_{d'}})$ with $0 \le d' < d$.
    Thus, they belong to $C_{d-1}\H$ by the induction hypothesis.
    This completes the proof.
\end{prf}
Thus, we can consider the coradical graded motivic M$s$Vs
\[
\sm(\{2\}^{l_0}, 1, \dots, 1, \{2\}^{l_d}) \bmod C_{d-1}\H \in \gr_d\H.
\]

\begin{prop} \label{prop:Delta}
    For $d \ge 0$ and $l_0, \dots, l_d \ge 0$, let $\kk = (\{2\}^{l_0}, 1, \dots, 1, \{2\}^{l_d})$.
    Then the image of $\sm(\kk)$ under the isomorphism $\widetilde{\Delta} \colon \gr_d\H \to C_1\L \otimes \gr_{d-1}\H$
    can be written explicitly as follows:
    \begin{align}
        \widetilde{\Delta}(\sm(\kk))
        & = \sum_{i = 0}^{d-1} \sum_{\substack{0 \le a \le l_i-1+\delta_{i, 0} \\ 0 \le b \le l_{i+1}}} \Big(\sl(\{2\}^a, 1, \{2\}^b) \\
        & \qquad\qquad\quad \otimes \sm(\{2\}^{l_0}, \dots, \{2\}^{l_{i-1}}, 1, \{2\}^{l_i+l_{i+1}-(a+b)}, 1, \{2\}^{l_{i+2}}, \dots, \{2\}^{l_d})\Big) \\[-2ex]
        & \quad - \sum_{i = 0}^{d-1} \sum_{\substack{0 \le a \le l_i \\ 0 \le b \le l_{i+1}-1}} \Big(\sl(\{2\}^b, 1, \{2\}^a) \\
        & \qquad\qquad\quad \otimes \sm(\{2\}^{l_0}, \dots, \{2\}^{l_{i-1}}, 1, \{2\}^{l_i+l_{i+1}-(a+b)}, 1, \{2\}^{l_{i+2}}, \dots, \{2\}^{l_d})\Big) \\[-2ex]
        & \quad + \sum_{i = 1}^{d-1} \sum_{1 \le b \le l_{i+1}} \Big(\bsl(\{2\}^{l_i}, 1, \{2\}^b) \\
        & \qquad\qquad\qquad \otimes \sm(\{2\}^{l_0}, \dots, \{2\}^{l_{i-1}}, 1, \{2\}^{l_{i+1}-b}, 1, \{2\}^{l_{i+2}}, \dots, \{2\}^{l_d})\Big) \\
        & \quad - \sum_{i = 0}^{d-2} \sum_{1 \le a \le l_i} \Big(\bsl(\{2\}^{l_{i+1}}, 1, \{2\}^a) \\
        & \qquad\qquad\qquad \otimes \sm(\{2\}^{l_0}, \dots, \{2\}^{l_{i-1}}, 1, \{2\}^{l_i-a}, 1, \{2\}^{l_{i+2}}, \dots, \{2\}^{l_d})\Big).
    \end{align}
\end{prop}

\begin{prf}
    The derivation
    \[
    D_{2r+1} \colon C_d\H \longrightarrow \L_{2r+1} \otimes C_{d-1}\H
    \]
    induces a map
    \[
    \gr_d D_{2r+1} \colon \gr_d\H \longrightarrow \L_{2r+1} \otimes \gr_{d-1}\H.
    \]
    This map coincides with the composition
    \[
    \gr_d\H  \overset{\widetilde{\Delta}}{\longrightarrow} C_1\L \otimes \gr_{d-1}\H \longrightarrow C_1\L_{2r+1} \otimes \gr_{d-1}\H \subseteq \L_{2r+1} \otimes \gr_{d-1}\H.
    \]
    Thus, we have $\widetilde{\Delta} = \bigoplus_{r \ge 0} \gr_d D_{2r+1}$.
    We computed $D_{2r+1}((-1)^{\dep\kk}\, \sm(\kk))$ in the proof of Proposition~\ref{prop:deriv}.
    It consists of the five types of terms (C), (D), (F), (G), and (H).
    Moreover, the right tensor component of each term is of the form
    \[
    \sm((\{2\}^{l'_0}, 1, \dots, 1, \{2\}^{l'_{d'}})) \quad \text{with}\ 0 \le d' < d.
    \]
    Therefore, to compute $\gr_d D_{2r+1}(\sm(\kk))$, it suffices to consider only those terms with $d' = d-1$.
    Consequently, in the terms (C) and (D), only the case $j = i+1$ contributes.
    \begin{center}
        \vspace{1.0cm}
        $\Im(0; \dots, \tikznode{etais}{\eta_{i, s}}, \tikznode{zerois}{0}, \eta_{i, s+1}, \dots, \eta_{i, l_i}, 0, -\eta_{i+1, 1}, \eta_{i+1, 1}, 0, \dots, \eta_{i+1, t}, \tikznode{zeroi1t}{0}, \tikznode{etai1t1}{\eta_{i+1, t+1}}, \dots ; 1)$
        \begin{tikzpicture}[remember picture, overlay] 
            \draw ([yshift=3pt]zerois.north) -- ++(0,0.4) -| node[pos=0, left] {(C)} ([yshift=3pt]etai1t1.north);
            
            \draw ([yshift=-3pt]etais.south) -- ++(0,-0.4) -| node[pos=0.9, below right] {(D)} ([yshift=-3pt]zeroi1t.south);
        \end{tikzpicture}
        \vspace{1.0cm}
    \end{center}
    For the term (C), the ranges of the summation over $i$, $s$, and $t$ are $0 \le i \le d-1$,
    $1 - \delta_{i, 0} \le s \le l_i$, and $0 \le t \le l_{i+1}$, respectively.
    Note that $s=0$ is allowed when $i = 0$.
    We have
    \begin{align}
        & \quad \sum_{\eta_{i, s+1}, \dots, \eta_{i+1, t} \in \{\pm1\}}
            \Il(0; \eta_{i, s+1}, 0, \dots, \eta_{i, l_i}, 0, -\eta_{i+1, 1}, \eta_{i+1, 1}, 0, \dots, \eta_{i+1, t}, 0; 1) \\[1ex]
        & \qquad \otimes \sum_{\substack{\eta_{0, 1}, \dots, \eta_{i, s} \\ \eta_{i+1, t+1}, \dots, \eta_{d, l_d}} \in \{\pm1\}}
            \Im(0; \dots, \eta_{i, 1}, 0, \dots, \eta_{i,s}, 0, \eta_{i+1, t+1}, \dots, \eta_{i+1, l_{i+1}}, 0, \dots; 1) \\[1ex]
        & = (-1)^{\dep\kk}\, \sl(\{2\}^{l_i-s}, 1, \{2\}^{t}) \\
        & \qquad\qquad\qquad \otimes \sm(\{2\}^{l_0}, \dots, \{2\}^{l_{i-1}}, 1, \{2\}^{s+l_{i+1}-t}, 1, \{2\}^{l_{i+2}}, \dots, \{2\}^{l_d}).
    \end{align}
    Thus, by the computation in the proof of Proposition~\ref{prop:deriv}, we obtain
    \begin{align}
        \text{(C)} 
        & = (-1)^{\dep\kk} \sum_{i = 0}^{d-1} \sum_{\substack{1 - \delta_{i, 0} \le s \le l_i \\ 0 \le t \le l_{i+1}}} \sl(\{2\}^{l_i-s}, 1, \{2\}^{t}) \\
        & \qquad\qquad\qquad \otimes \sm(\{2\}^{l_0}, \dots, \{2\}^{l_{i-1}}, 1, \{2\}^{s+l_{i+1}-t}, 1, \{2\}^{l_{i+2}}, \dots, \{2\}^{l_d}) \\
        & = (-1)^{\dep\kk} \sum_{i = 0}^{d-1} \sum_{\substack{0 \le a \le l_i-1+\delta_{i, 0} \\ 0 \le b \le l_{i+1}}} \sl(\{2\}^a, 1, \{2\}^b) \\
        & \qquad\qquad\qquad \otimes \sm(\{2\}^{l_0}, \dots, \{2\}^{l_{i-1}}, 1, \{2\}^{l_i+l_{i+1}-(a+b)}, 1, \{2\}^{l_{i+2}}, \dots, \{2\}^{l_d}).
    \end{align}
    For the term (D), the ranges of the summation over $i$, $s$, and $t$ are $0 \le i \le d-1$, 
    $1 \le s \le l_i+1$, and $1 \le t \le l_{i+1}$, respectively.
    By the property \ref{enum:pathrev}), we have
    \begin{align}
        & \quad \sum_{\eta_{i, s+1}, \dots, \eta_{i+1, t} \in \{\pm1\}}
            \Il(1; 0, \eta_{i, s+1}, \dots, \eta_{i, l_i}, 0, -\eta_{i+1, 1}, \eta_{i+1, 1}, 0, \dots, \eta_{i+1, t}; 0) \\[1ex]
        & \qquad \otimes \sum_{\substack{\eta_{0, 1}, \dots, \eta_{i, s} \\ \eta_{j, t+1}, \dots, \eta_{d, l_d}} \in \{\pm1\}}
            \Im(0; \dots, \eta_{i, 1}, 0, \dots, \eta_{i,s}, 0, \eta_{i+1, t+1}, \dots, \eta_{i+1, l_{i+1}}, 0, \dots; 1) \\
        & = - \sum_{\eta_{i, s+1}, \dots, \eta_{i+1, t} \in \{\pm1\}}
            \Il(0; \eta_{i+1, t}, 0, \dots, \eta_{i+1, 1}, -\eta_{i+1, 1}, 0, \eta_{i, l_i}, 0, \dots, \eta_{i, s+1}, 0; 1) \\[1ex]
        & \qquad \otimes \sum_{\substack{\eta_{0, 1}, \dots, \eta_{i, s} \\ \eta_{j, t+1}, \dots, \eta_{d, l_d}} \in \{\pm1\}}
            \Im(0; \dots, \eta_{i, 1}, 0, \dots, \eta_{i,s}, 0, \eta_{i+1, t+1}, \dots, \eta_{i+1, l_{i+1}}, 0, \dots; 1) \\
        & = - (-1)^{\dep\kk}\, \sl(\{2\}^{t-1}, 1, \{2\}^{l_i-s+1}) \\
        & \qquad\qquad\qquad \otimes \sm(\{2\}^{l_0}, \dots, \{2\}^{l_{i-1}}, 1, \{2\}^{s+l_{i+1}-t}, 1, \{2\}^{l_{i+2}}, \dots, \{2\}^{l_d}).
    \end{align}
    Therefore, by the computation in the proof of Proposition~\ref{prop:deriv}, we obtain
    \begin{align}
        \text{(D)} 
        & = - (-1)^{\dep\kk} \sum_{i = 0}^{d-1} \sum_{\substack{1 \le s \le l_i+1 \\ 1 \le t \le l_{i+1}}} \sl(\{2\}^{t-1}, 1, \{2\}^{l_i-s+1}) \\
        & \qquad\qquad\qquad \otimes \sm(\{2\}^{l_0}, \dots, \{2\}^{l_{i-1}}, 1, \{2\}^{s+l_{i+1}-t}, 1, \{2\}^{l_{i+2}}, \dots, \{2\}^{l_d}) \\
        & = - (-1)^{\dep\kk} \sum_{i = 0}^{d-1} \sum_{\substack{0 \le a \le l_i \\ 0 \le b \le l_{i+1}-1}} \sl(\{2\}^b, 1, \{2\}^a) \\
        & \qquad\qquad\qquad \otimes \sm(\{2\}^{l_0}, \dots, \{2\}^{l_{i-1}}, 1, \{2\}^{l_i+l_{i+1}-(a+b)}, 1, \{2\}^{l_{i+2}}, \dots, \{2\}^{l_d}).
    \end{align}

    Next, in the terms (F) and (G), only the case $j = i+1$ contributes.
    \begin{center}
        \vspace{1.0cm}
        $\Im(0; \dots, \tikznode{-etai1}{-\eta_{i, 1}}, \eta_{i, 1}, \dots, \tikznode{etais}{\eta_{i, s}}, \dots, -\eta_{i+1, 1}, \eta_{i+1, 1}, \dots, \tikznode{etai1t1}{\eta_{i+1, t+1}}, \dots, -\eta_{i+2, 1}, \tikznode{etai21}{\eta_{i+2, 1}}, \dots ; 1)$
        \begin{tikzpicture}[remember picture, overlay] 
            \draw ([xshift=2pt, yshift=3pt]-etai1.north) -- ++(0,0.4) -| node[pos=0.6, right] {(F)} ([yshift=3pt]etai1t1.north);
            
            \draw ([xshift=-2pt, yshift=-3pt]etais.south) -- ++(0,-0.4) -| node[pos=0, left] {(G)} ([yshift=-3pt]etai21.south);
        \end{tikzpicture}
        \vspace{1.0cm}
    \end{center}
    For the term (F), the ranges of the summation over $i$ and $t$ are 
    $1 \le i \le d-1$ and $1 \le t \le l_{i+1}$, respectively.
    By \eqref{eq:pathcomp} and the property \ref{enum:homothety}), it follows that
    \begin{align}
        & \quad \sum_{\eta_{i, 2}, \dots, \eta_{i+1, t} \in \{\pm1\}}
            \Il(-1; 1, 0, \eta_{i, 2}, 0, \dots, -\eta_{i+1, 1}, \eta_{i+1, 1}, 0, \dots, \eta_{i+1, t}, 0; 1) \\
        & = - \sum_{\eta_{i, 2}, \dots, \eta_{i+1, t} \in \{\pm1\}}
            \Il(0; 1, 0, \eta_{i, 2}, 0, \dots, -\eta_{i+1, 1}, \eta_{i+1, 1}, 0, \dots, \eta_{i+1, t}, 0; -1) \\
        & \qquad + \sum_{\eta_{i, 2}, \dots, \eta_{i+1, t} \in \{\pm1\}}
            \Il(0; 1, 0, \eta_{i, 2}, 0, \dots, -\eta_{i+1, 1}, \eta_{i+1, 1}, 0, \dots, \eta_{i+1, t}, 0; 1) \\
        & = - \sum_{\eta_{i, 2}, \dots, \eta_{i+1, t} \in \{\pm1\}}
            \Il(0; -1, 0, -\eta_{i, 2}, 0, \dots, \eta_{i+1, 1}, -\eta_{i+1, 1}, 0, \dots, -\eta_{i+1, t}, 0; 1) \\
        & \qquad + \sum_{\eta_{i, 2}, \dots, \eta_{i+1, t} \in \{\pm1\}}
            \Il(0; 1, 0, \eta_{i, 2}, 0, \dots, -\eta_{i+1, 1}, \eta_{i+1, 1}, 0, \dots, \eta_{i+1, t}, 0; 1) \\
        & = \sum_{\eta_{i, 1}, \dots, \eta_{i+1, t} \in \{\pm1\}}
            \eta_{i, 1}\, \Il(0; \eta_{i, 1}, 0, \eta_{i, 2}, 0, \dots, -\eta_{i+1, 1}, \eta_{i+1, 1}, 0, \dots, \eta_{i+1, t}, 0; 1) \\
        & = (-1)^{l_i+1+t}\, \bsl(\{2\}^{l_i}, 1, \{2\}^t).
    \end{align}
    By the computation in the proof of Proposition~\ref{prop:deriv}, we have
    \begin{align}
        \text{(F)}
        & = (-1)^{\dep\kk} \sum_{i = 1}^{d-1} \sum_{1 \le t \le l_{i+1}} \bsl(\{2\}^{l_i}, 1, \{2\}^t) \\
        & \qquad\qquad\qquad \otimes \sm(\{2\}^{l_0}, \dots, \{2\}^{l_{i-1}}, 1, \{2\}^{l_{i+1}-t}, 1, \{2\}^{l_{i+2}}, \dots, \{2\}^{l_d}).
    \end{align}
    For the term (G), the ranges of the summation over $i$ and $s$ are 
    $0 \le i \le d-2$ and $1 \le s \le l_i$, respectively.
    By the properties \ref{enum:pathrev}), \ref{enum:homothety}) and \eqref{eq:pathcomp}, it follows that
    \begin{align}
        & \quad \sum_{\eta_{i, s+1}, \dots, \eta_{i+1, l_{i+1}} \in \{\pm1\}}
            \Il(-1; 0, \eta_{i, s+1}, \dots, 0, -\eta_{i+1, 1}, \eta_{i+1, 1}, 0, \dots, \eta_{i+1, l_{i+1}}, 0, -1; 1) \\
        & = - \sum_{\eta_{i, s+1}, \dots, \eta_{i+1, l_{i+1}} \in \{\pm1\}}
            \Il(1; -1, 0, \eta_{i+1, l_{i+1}}, \dots, \eta_{i+1, 1}, -\eta_{i+1, 1}, 0, \dots, \eta_{i, s+1}, 0; -1) \\
        & = \sum_{\eta_{i, s+1}, \dots, \eta_{i+1, l_{i+1}} \in \{\pm1\}}
            \Il(0; -1, 0, \eta_{i+1, l_{i+1}}, \dots, \eta_{i+1, 1}, -\eta_{i+1, 1}, 0, \dots, \eta_{i, s+1}, 0; 1) \\
        & \quad - \sum_{\eta_{i, s+1}, \dots, \eta_{i+1, l_{i+1}} \in \{\pm1\}}
            \Il(0; -1, 0, \eta_{i+1, l_{i+1}}, \dots, \eta_{i+1, 1}, -\eta_{i+1, 1}, 0, \dots, \eta_{i, s+1}, 0; -1) \\
        & = \sum_{\eta_{i, s+1}, \dots, \eta_{i+1, l_{i+1}} \in \{\pm1\}}
            \Il(0; -1, 0, \eta_{i+1, l_{i+1}}, \dots, \eta_{i+1, 1}, -\eta_{i+1, 1}, 0, \dots, \eta_{i, s+1}, 0; 1) \\
        & \quad - \sum_{\eta_{i, s+1}, \dots, \eta_{i+1, l_{i+1}} \in \{\pm1\}}
            \Il(0; 1, 0, -\eta_{i+1, l_{i+1}}, \dots, -\eta_{i+1, 1}, \eta_{i+1, 1}, 0, \dots, -\eta_{i, s+1}, 0; 1) \\
        & = - \sum_{\eta_{i, s+1}, \dots, \eta_{i+2, 1} \in \{\pm1\}}
            \eta_{i+2, 1}\, \Il(0; \eta_{i+2, 1}, \dots, \eta_{i+1, 1}, -\eta_{i+1, 1}, 0, \dots, \eta_{i, s+1}, 0; 1) \\
        & = - (-1)^{l_{i+1}+l_i-s+2}\, \bsl(\{2\}^{l_{i+1}}, 1, \{2\}^{l_i-s+1}).
    \end{align}
    By the computation in the proof of Proposition~\ref{prop:deriv}, we have
    \begin{align}
        \text{(G)}
        & = -(-1)^{\dep\kk} \sum_{i = 0}^{d-2} \sum_{1 \le s \le l_i} \bsl(\{2\}^{l_{i+1}}, 1, \{2\}^{l_i-s+1}) \\
        & \qquad\qquad\qquad \otimes \sm(\{2\}^{l_0}, \dots, \{2\}^{l_{i-1}}, 1, \{2\}^{s-1}, 1, \{2\}^{l_{i+2}}, \dots, \{2\}^{l_d}) \\
        & = -(-1)^{\dep\kk} \sum_{i = 0}^{d-2} \sum_{1 \le a \le l_i} \bsl(\{2\}^{l_{i+1}}, 1, \{2\}^a) \\
        & \qquad\qquad\qquad \otimes \sm(\{2\}^{l_0}, \dots, \{2\}^{l_{i-1}}, 1, \{2\}^{l_i-a}, 1, \{2\}^{l_{i+2}}, \dots, \{2\}^{l_d}).
    \end{align}

    Finally, in the term (H), only the case $j = i$ contributes.
    \begin{center}
        \vspace{1.0cm}
        $\Im(0; \dots, \tikznode{-etai1}{-\eta_{i, 1}}, \eta_{i, 1}, 0, \dots, \eta_{i, l_i}, 0, -\eta_{i+1, 1}, \tikznode{etai11}{\eta_{i+1, 1}}, \dots ; 1)$
        \begin{tikzpicture}[remember picture, overlay] 
            \draw ([xshift=-2pt, yshift=3pt]-etai1.north) -- ++(0,0.4) -| node[pos=0, left] {(H)} ([yshift=3pt]etai11.north);
        \end{tikzpicture}
        \vspace{0.5cm}
    \end{center}
    By \eqref{eq:pathcomp}, it follows that
    \begin{align}
        & \quad \sum_{\eta_{i, 2}, \dots, \eta_{i, l_i} \in \{\pm1\}}
            \Il(-1; 1, 0, \eta_{i, 2}, 0, \dots, \eta_{i, l_i}, 0, -1; 1) \\
        & = \sum_{\eta_{i, 2}, \dots, \eta_{i, l_i} \in \{\pm1\}}
            \Il(-1; 1, 0, \eta_{i, 2}, 0, \dots, \eta_{i, l_i}, 0, -1; 0) \\
        & \qquad\qquad + \sum_{\eta_{i, 2}, \dots, \eta_{i, l_i} \in \{\pm1\}}
            \Il(0; 1, 0, \eta_{i, 2}, 0, \dots, \eta_{i, l_i}, 0, -1; 1).
    \end{align}
    By the properties~\ref{enum:pathrev}) and \ref{enum:homothety}), one can rewrite the first term as
    \begin{align}
        & \quad \sum_{\eta_{i, 2}, \dots, \eta_{i, l_i} \in \{\pm1\}}
            \Il(-1; 1, 0, \eta_{i, 2}, 0, \dots, \eta_{i, l_i}, 0, -1; 0) \\
        & = \sum_{\eta_{i, 2}, \dots, \eta_{i, l_i} \in \{\pm1\}}
            \Il(1; -1, 0, -\eta_{i, 2}, 0, \dots, -\eta_{i, l_i}, 0, 1; 0) \\
        & = - \sum_{\eta_{i, 2}, \dots, \eta_{i, l_i} \in \{\pm1\}}
            \Il(0; 1, 0, -\eta_{i, l_i}, 0, \dots, -\eta_{i, 2}, 0, -1; 1).
    \end{align}
    Therefore, we obtain
    \[
    \sum_{\eta_{i, 2}, \dots, \eta_{i, l_i} \in \{\pm1\}} \Il(-1; 1, 0, \eta_{i, 2}, 0, \dots, \eta_{i, l_i}, 0, -1; 1) = 0,
    \]
    which implies the term (H) vanishes by the computation in the proof of Proposition~\ref{prop:deriv}.
    Combining the above computations, we obtain the desired result.
\end{prf}

\section{Proof of the main theorem}
In this section, we prove the following motivic analogue of Theorem~\ref{thm:sbasis}.
\begin{thm} \label{thm:msbasis}
    The set
    \[
    \{\sm(k_1, \dots, k_r) \mid k_i \in \{1, 2\}\}
    \]
    forms a basis of the $\Q$-vector space of motivic alternating multiple zeta values.
\end{thm}
Theorem~\ref{thm:sbasis} is obtained by applying the period map to Theorem~\ref{thm:msbasis}.
To prove Theorem~\ref{thm:msbasis}, it suffices to show the following theorem.
\begin{thm} \label{thm:grbasis}
    For $N \ge d \ge 0$, the set
    \[
    \{\sm(k_1, \dots, k_r) \bmod C_{d-1}\H_N \mid k_i \in \{1, 2\},\; k_1+\dots+k_r = N,\; \#\{i \mid k_i = 1\} = d\}
    \]
    forms a basis of $\gr_d\H_N$.
\end{thm}

First, we lift the relations for M$s$Vs in Theorem~\ref{thm:sidentity} and M$\bar{s}$Vs in Theorem~\ref{thm:bsidentity}
to the motivic setting.
\begin{thm} \label{thm:msidentity}
    Let $a$ and $b$ be non-negative integers.
    Then we have
    \begin{equation}
    \sm(\{2\}^a, 1, \{2\}^b)
        = \sum_{r=0}^{a+b} \frac{(-1)^r}{2^r} \alpha_r^{a, b}\, \zm(\overline{2r+1})\, \sm(\{2\}^{a+b-r}), \label{eq:msidentity}
    \end{equation}
    where $\alpha_r^{a, b}$ is the rational number defined by
    \[
    \alpha_r^{a, b} = - \frac{2^{2r}}{2^{2r}-1} \binom{2r}{2a+1} + \binom{2r}{2b}.
    \]
    For $r=0$, the term involving $\frac{2^{2r}}{2^{2r}-1}$ is understood to be zero.
\end{thm}

\begin{thm} \label{thm:mbsidentity}
    Let $a$ and $b$ be non-negative integers.
    Then we have
    \begin{equation}
    \bsm(\{2\}^a, 1, \{2\}^b)
        = \sum_{r=0}^{a+b} \frac{(-1)^r}{2^r} \beta_r^{a, b}\, \zm(\overline{2r+1})\, \bsm(\{2\}^{a+b-r}), \label{eq:mbsidentity}
    \end{equation}
    where $\beta_r^{a, b}$ is the rational number defined by
    \[
    \beta_r^{a, b} =
        \begin{dcases}
            \frac{2^{2r}}{2^{2r}-1} \binom{2r}{2a} + \binom{2r}{2b} & \text{if} \ a \ge 1, \\
            -2^{2r+1} + \delta_{r, b} & \text{if} \ a = 0.\\
        \end{dcases}
    \]
    For $r=0$, the term involving $\frac{2^{2r}}{2^{2r}-1}$ is understood to be zero.
\end{thm}

Brown lifted Zagier's 2-3-2 formula to its motivic version via the following theorem.
For $N \ge 1$, we write $D_{<N}$ for $\bigoplus_{1 \le 2r+1 < N} D_{2r+1}$.
\begin{thm}[{{\cite[Theorem 3.3]{MR2993755}}}]
    Let $N \ge 2$.
    The kernel of $D_{<N}$ on $\H^1$ is one-dimensional in weight $N$,
    and is spanned by $\zm(N)$\textup{:}
    \[
    \Ker D_{<N} \cap \H_N^1 = \zm(N)\, \Q.
    \]
\end{thm}
To prove Theorems~\ref{thm:msidentity} and~\ref{thm:mbsidentity}, we use the following alternating analogue
due to Glanois.
\begin{thm}[{{\cite[Corollary 2.4.5]{glanois:tel-01347989}}}] \label{thm:KerD}
    Let $N \ge 1$.
    The kernel of $D_{<N}$ on $\H^2$ is one-dimensional in weight $N$,
    and is spanned by $\zm(\overline{N})$\textup{:}
    \[
    \Ker D_{<N} \cap \H_N^2 = \zm(\overline{N})\, \Q.
    \]
\end{thm}

\begin{prop} \label{prop:ds1}
    Let $a$ and $b$ be non-negative integers.
    Then for $0 \le r \le a+b$, we have
    \begin{align}
        & \quad D_{2r+1}(\sm(\{2\}^a, 1, \{2\}^b)) \\
        = & \qty(\sum_{\substack{0 \le a' \le a \\ 0 \le b' \le b \\ a' + b' = r}} \sl(\{2\}^{a'}, 1, \{2\}^{b'})
        - \sum_{\substack{0 \le a' \le a \\ 0 \le b' \le b-1 \\ a' + b' = r}} \sl(\{2\}^{b'}, 1, \{2\}^{a'}))
        \otimes \sm(\{2\}^{a+b-r}).
    \end{align}    
\end{prop}

\begin{prf}
    We have the following commutative diagram:
    \[
    \begin{tikzcd}[
        column sep=1.5em,
        row sep=2em
        ]
        {C_1\H} &&&& {C_1\L_{2r+1} \otimes C_0\H} \\
        \\
        {\gr_1\H} && {C_1\L \otimes \gr_0\H} && {C_1\L_{2r+1} \otimes \gr_0\H.}
        \arrow["{D_{2r+1}}", from=1-1, to=1-5]
        \arrow[two heads, from=1-1, to=3-1]
        \arrow[equals, from=1-5, to=3-5]
        \arrow["{\widetilde{\Delta}}", from=3-1, to=3-3]
        \arrow[two heads, from=3-3, to=3-5]
    \end{tikzcd}
    \]
    Therefore, by the case $d=1$ of Proposition~\ref{prop:Delta}, we obtain
    \begin{align}
        D_{2r+1}(\sm(\{2\}^a, 1, \{2\}^b))
        & = \sum_{\substack{0 \le a' \le a \\ 0 \le b' \le b \\ a' + b' = r}} 
            \sl(\{2\}^{a'}, 1, \{2\}^{b'}) \otimes \sm(\{2\}^{a+b-r}) \\
        & \qquad - \sum_{\substack{0 \le a' \le a \\ 0 \le b' \le b-1 \\ a' + b' = r}} 
            \sl(\{2\}^{b'}, 1, \{2\}^{a'}) \otimes \sm(\{2\}^{a+b-r}),
    \end{align}
    as desired.
\end{prf}

\begin{lem} \label{lem:alpha}
    Let $a$ and $b$ be non-negative integers.
    Then for $0 \le r \le a+b$, we have
    \[
    \sum_{\substack{0 \le a' \le a \\ 0 \le b' \le b \\ a' + b' = r}} \alpha_r^{a', b'}
    - \sum_{\substack{0 \le a' \le a \\ 0 \le b' \le b-1 \\ a' + b' = r}} \alpha_r^{b', a'}
    = \alpha_r^{a, b}.
    \]
\end{lem}

\begin{prf}
    By the definition of $\alpha_r^{a, b}$, it suffices to show the following two identities:
    \[
    \sum_{\substack{0 \le a' \le a \\ 0 \le b' \le b \\ a' + b' = r}} \binom{2r}{2a'+1}
        - \sum_{\substack{0 \le a' \le a \\ 0 \le b' \le b-1 \\ a' + b' = r}} \binom{2r}{2b'+1}
    = \binom{2r}{2a+1},
    \]
    and
    \[
    \sum_{\substack{0 \le a' \le a \\ 0 \le b' \le b \\ a' + b' = r}} \binom{2r}{2b'}
        - \sum_{\substack{0 \le a' \le a \\ 0 \le b' \le b-1 \\ a' + b' = r}} \binom{2r}{2a'}
    = \binom{2r}{2b}.
    \]
    The former identity follows from
    \[
    \sum_{\substack{0 \le a' \le a \\ 0 \le b' \le b-1 \\ a' + b' = r}} \binom{2r}{2b'+1}
    = \sum_{\substack{0 \le a' \le a \\ 0 \le b' \le b-1 \\ a' + b' = r}} \binom{2r}{2a'-1}
    = \sum_{\substack{0 \le a' \le a-1 \\ 1 \le b' \le b \\ a' + b' = r}} \binom{2r}{2a'+1},
    \]
    where we used $2a'-1 = 2(a'-1)+1$ and shifted the indices in the last step.
    The latter identity follows from
    \[
    \sum_{\substack{0 \le a' \le a \\ 0 \le b' \le b-1 \\ a' + b' = r}} \binom{2r}{2a'}
    = \sum_{\substack{0 \le a' \le a \\ 0 \le b' \le b-1 \\ a' + b' = r}} \binom{2r}{2b'},
    \]
    which completes the proof.
\end{prf}

\begin{prf}[of Theorem~\ref{thm:msidentity}]
    We proceed by induction on $a+b$. 
    The case $a=b=0$ is trivial, so we assume $a+b > 0$.
    Let $R^{a, b}$ and $L^{a, b}$ denote the right-hand side and the left-hand side of \eqref{eq:msidentity}, respectively.
    By the induction hypothesis, for $a'+b' < a+b$, we have
    \[
    \sl(\{2\}^{a'}, 1, \{2\}^{b'}) = \frac{(-1)^{a'+b'}}{2^{a'+b'}} \alpha_{a'+b'}^{a', b'}\, \zl(\overline{2a'+2b'+1}).
    \]
    Thus, for $0 \le r < a+b$, Proposition~\ref{prop:ds1} implies that
    \[
    D_{2r+1}(L^{a, b}) = \frac{(-1)^r}{2^r} \qty(\sum_{\substack{0 \le a' \le a \\ 0 \le b' \le b \\ a' + b' = r}} \alpha_r^{a', b'}
        - \sum_{\substack{0 \le a' \le a \\ 0 \le b' \le b-1 \\ a' + b' = r}} \alpha_r^{b', a'})
        \zl(\overline{2r+1}) \otimes \sm(\{2\}^{a+b-r}).
    \]
    By Lemma~\ref{lem:alpha}, we obtain
    \[
    D_{2r+1}(L^{a, b}) = \frac{(-1)^r}{2^r} \alpha_r^{a, b}\, \zl(\overline{2r+1}) \otimes \sm(\{2\}^{a+b-r}).
    \]
    On the other hand, since $D_{2r+1}$ is a derivation, we have
    \begin{equation}
        \begin{aligned}[t]
            D_{2r+1}(R^{a, b}) = \sum_{s = 0}^{a+b} \frac{(-1)^s}{2^s} \alpha_s^{a, b}\, 
                \Big((1 \otimes & \zm(\overline{2s+1})) D_{2r+1}(\sm(\{2\}^{a+b-s})) \\
            & + (1 \otimes \sm(\{2\}^{a+b-r})) D_{2r+1}(\zm(\overline{2s+1}))\Big).
        \end{aligned}
    \end{equation}
    By equation~\eqref{eq:Delta}, we have
    \[
    D_{2r+1}(\zm(\overline{2s+1})) = \delta_{r, s}\, \zl(\overline{2r+1}) \otimes 1.
    \]
    By the computation of the terms of type (A) and (B) in the proof of 
    Proposition~\ref{prop:deriv}, it follows that
    \[
    D_{2r+1}(\sm(\{2\}^d)) = 0\quad \text{for all}\ d \ge 0.
    \]
    Thus, we obtain
    \[
    D_{2r+1}(R^{a, b}) = \frac{(-1)^r}{2^r} \alpha_r^{a, b}\, \zl(\overline{2r+1}) \otimes \sm(\{2\}^{a+b-r}) = D_{2r+1}(L^{a, b}).
    \]
    Hence,
    \[
    D_{<2a+2b+1}(L^{a, b} - R^{a, b})  = 0.
    \]
    By Theorem~\ref{thm:KerD}, this implies that there exists a constant $c \in \Q$ such that
    $L^{a, b} - R^{a, b} = c\, \zm(\overline{2a+2b+1})$.
    Applying the period map yields $c = 0$, which completes the proof.
\end{prf}

The proof of Theorem~\ref{thm:mbsidentity} is similar.
\begin{prop} \label{prop:dbs}
    Let $a$ and $b$ be non-negative integers.
    Then for $0 \le r \le a+b$, we have
    \begin{equation}
        \begin{aligned}
            D_{2r+1}(\bsm(\{2\}^a, 1, \{2\}^b)) & = \frac{(-1)^r}{2^r}(\alpha_r^{a, b} - \delta_{r \ge a}\, \alpha_r^{a, r-a})\,
                \zl(\overline{2r+1}) \otimes \bsm(\{2\}^{a+b-r}) \\
            & \qquad + \delta_{r \ge a}\, \bsl(\{2\}^{a}, 1, \{2\}^{r-a}) \otimes \bsm(\{2\}^{a+b-r}).
        \end{aligned}
    \end{equation}
\end{prop}

\begin{prf}
    By the definition of the motivic M$\bar{s}$Vs, we have
    \[
    (-1)^{a+b+1}\, \bsm(\{2\}^a, 1, \{2\}^b) = \sum_{\eta_i, \eta'_j \in \{\pm1\}} \eta_1\Im(0; \eta_1, 0, \dots, \eta_a, 0, -\eta'_1, \eta'_1, 0, \dots, \eta'_b, 0; 1).
    \]
    When $a=0$, we understand $\eta_1$ to be $-\eta'_1$.
    By applying $D_{2r+1}$, we get three types of terms as shown in the following diagram.
    \begin{center}
        \vspace{1.4cm}
        $\Im(\tikznode{zero0}{0}; \eta_1, 0, \dots, \tikznode{etas}{\eta_s}, \tikznode{zero1}{0}, \eta_{s+1}, 0, \dots, \tikznode{-eta'1}{-\eta'_1}, \tikznode{eta'1}{\eta'_1}, 0, \dots, \eta'_t, \tikznode{zero2}{0}, \tikznode{eta'tplus1}{\eta'_{t+1}}, 0, \dots, \eta'_b, 0; 1)$
        \begin{tikzpicture}[remember picture, overlay] 
            \draw ([yshift=3pt]zero0.north) -- ++(0,0.8) -| node[pos=0, left] {(C)} ([xshift=2pt, yshift=3pt]eta'tplus1.north);

            \draw ([yshift=3pt]zero1.north) -- ++(0,0.4) -| node[pos=0, left] {(A)} ([xshift=-2pt, yshift=3pt]eta'tplus1.north);
            
            \draw ([yshift=-3pt]etas.south) -- ++(0,-0.4) -| node[pos=0.9, below right] {(B)} ([yshift=-3pt]zero2.south);
        \end{tikzpicture}
        \vspace{1.0cm}
    \end{center}
    For the type (A), by the property~\ref{enum:homothety}), we have
    \begin{align}
        & \sum_{\substack{1 \le s \le a \\ 0 \le t \le b \\ a-s+t=r}} \sum_{\eta_i, \eta'_j \in \{\pm1\}}
            \Big(\Il(0; \eta_{s+1}, 0, \dots, -\eta'_1, \eta'_1, 0, \dots, \eta'_t, 0; \eta'_{t+1}) \\[-2ex]
        & \qquad\qquad\qquad\qquad\qquad \otimes \eta_1\Im(0; \eta_1, 0, \dots, \eta_s, 0, \eta'_{t+1}, 0, \dots, \eta'_b, 0; 1)\Big) \\[2ex]
        = & \sum_{\substack{1 \le s \le a \\ 0 \le t \le b \\ a-s+t=r}} \sum_{\eta_i, \eta'_j \in \{\pm1\}}
            \Big(\Il(0, \eta'^{\,-1}_{t+1}\,\eta_{s+1}, 0, \dots, -\eta'^{\,-1}_{t+1}\,\eta'_1, \eta'^{\,-1}_{t+1}\,\eta'_1, 0, \dots, \eta'^{\,-1}_{t+1}\,\eta'_t, 0; 1) \\[-2ex]
        & \qquad\qquad\qquad\qquad\qquad \otimes \eta_1\Im(0; \eta_1, 0, \dots, \eta_s, 0, \eta'_{t+1}, 0, \dots, \eta'_b, 0; 1)\Big) \\[2ex]
        = & \sum_{\substack{1 \le s \le a \\ 0 \le t \le b \\ a-s+t=r}} \sum_{\substack{\eta_{s+1}, \dots, \eta_a \\ \eta'_1, \dots, \eta'_t} \in \{\pm1\}}
            \Il(0; \eta_{s+1}, 0, \dots, -\eta'_1, \eta'_1,0, \dots, \eta'_t, 0; 1) \\
        & \qquad\qquad\qquad \otimes \sum_{\substack{\eta_1, \dots, \eta_s \\ \eta'_{t+1}, \dots, \eta'_b} \in \{\pm1\}}
            \eta_1\Im(0; \eta_1, 0, \dots, \eta_s, 0, \eta'_{t+1}, 0, \dots, \eta'_b, 0; 1) \\[1ex]
        = &\, (-1)^{a+b+1} \sum_{\substack{1 \le s \le a \\ 0 \le t \le b \\ a-s+t=r}}
            \sl(\{2\}^{a-s}, 1, \{2\}^{t}) \otimes \bsm(\{2\}^{a+b-r}) \\
        = &\, (-1)^{a+b+1} \sum_{\substack{0 \le a' \le a-1 \\ 0 \le b' \le b \\ a' + b' = r}}
            \sl(\{2\}^{a'}, 1, \{2\}^{b'}) \otimes \bsm(\{2\}^{a+b-r}).
    \end{align}
    For the type (B), by the properties~\ref{enum:pathrev}) and \ref{enum:homothety}), we have
    \begin{align}
        & \sum_{\substack{1 \le s \le a+1 \\ 1 \le t \le b \\ a-s+t=r}} \sum_{\eta_i, \eta'_j \in \{\pm1\}} 
            \Big(\Il(\eta_s; 0, \eta_{s+1}, 0, \dots, -\eta'_1, \eta'_1, 0, \dots, \eta'_t; 0) \\[-2ex]
        & \qquad\qquad\qquad\qquad\qquad \otimes \eta_1\Im(0; \eta_1, 0, \dots, \eta_s, 0, \eta'_{t+1}, 0, \dots, \eta'_b, 0; 1)\Big) \\[2ex]
        = & - \sum_{\substack{1 \le s \le a+1 \\ 1 \le t \le b \\ a-s+t=r}} \sum_{\eta_i, \eta'_j \in \{\pm1\}} 
            \Big(\Il(0; \eta_s^{-1}\,\eta'_t, 0, \dots, \eta_s^{-1}\,\eta'_1, -\eta_s^{-1}\,\eta'_1, 0, \dots, \eta_s^{-1}\,\eta_{s+1}, 0; 1) \\[-2ex]
        & \qquad\qquad\qquad\qquad\qquad \otimes \eta_1\Im(0; \eta_1, 0, \dots, \eta_s, 0, \eta'_{t+1}, 0, \dots, \eta'_b, 0; 1)\Big) \\[2ex]
        = & - \sum_{\substack{1 \le s \le a+1 \\ 1 \le t \le b \\ a-s+t=r}} \sum_{\substack{\eta_{s+1}, \dots, \eta_a \\ \eta'_1, \dots, \eta'_t} \in \{\pm1\}}
            \Il(0; \eta'_t, 0, \dots, \eta'_2, 0, \eta'_1, -\eta'_1,0, \eta_a, 0, \dots, \eta_{s+1}, 0; 1) \\
        & \qquad\qquad\qquad \otimes \sum_{\substack{\eta_1, \dots, \eta_s \\ \eta'_{t+1}, \dots, \eta'_b} \in \{\pm1\}}
            \eta_1\Im(0; \eta_1, 0, \dots, \eta_s, 0, \eta'_{t+1}, 0, \dots, \eta'_b, 0; 1) \\[1ex]
        = & -(-1)^{a+b+1} \sum_{\substack{1 \le s \le a+1 \\ 1 \le t \le b \\ a-s+t=r}}
            \sl(\{2\}^{t-1}, 1, \{2\}^{a+1-s}) \otimes \bsm(\{2\}^{a+b-r}) \\
        = & -(-1)^{a+b+1} \sum_{\substack{0 \le a' \le a \\ 0 \le b' \le b-1 \\ a' + b' = r}}
            \sl(\{2\}^{b'}, 1, \{2\}^{a'}) \otimes \bsm(\{2\}^{a+b-r}).
    \end{align}
    For the type (C), by the property~\ref{enum:homothety}), we have
    \begin{align}
        & \delta_{r \ge a} \sum_{\eta_i, \eta'_j \in \{\pm1\}} 
            \Big(\eta_1\Il(0; \eta_1, 0, \dots, \eta_a, 0, -\eta'_1, \eta'_1, 0, \dots, \eta'_{r-a}, 0; \eta'_{r-a+1}) \\[-2ex]
        & \qquad\qquad\qquad\qquad \otimes \Im(0; \eta'_{r-a+1}, 0, \dots, \eta'_{b}, 0; 1)\Big) \\[2ex]
        = &\, \delta_{r \ge a} \sum_{\eta_i, \eta'_j \in \{\pm1\}} 
            \Big(\eta_1\Il(0; \eta'^{\,-1}_{r-a+1}\,\eta_1, 0, -\eta'^{\,-1}_{r-a+1}\,\eta'_1, \eta'^{\,-1}_{r-a+1}\,\eta'_1, 0, \dots, \eta'^{\,-1}_{r-a+1}\,\eta'_{r-a}, 0; 1) \\[-2ex]
        & \qquad\qquad\qquad\qquad \otimes \Im(0; \eta'_{r-a+1}, 0, \dots, \eta'_{b}, 0; 1)\Big) \\[2ex]
        = &\, \delta_{r \ge a} \sum_{\substack{\eta_1, \dots, \eta_a \\ \eta'_1, \dots, \eta'_{r-a}} \in \{\pm1\}} 
            \eta_1\Il(0; \eta_1, 0, \dots, \eta_a, 0, -\eta'_1, \eta'_1, 0, \dots, \eta'_{r-a}, 0; 1) \\
        & \qquad\qquad\qquad \otimes \sum_{\eta'_{r-a+1}, \dots, \eta'_b \in \{\pm1\}}
            \eta'_{r-a+1}\Im(0; \eta'_{r-a+1}, 0, \dots, \eta'_{b}, 0; 1) \\[1ex]
        = &\, \delta_{r \ge a}\, (-1)^{a+b+1}\, \bsl(\{2\}^a, 1, \{2\}^{r-a}) \otimes \bsm(\{2\}^{a+b-r}).
    \end{align}
    Therefore, by Theorem~\ref{thm:msidentity} and Lemma~\ref{lem:alpha}, we obtain
    \begin{align}
        & \qquad D_{2r+1}(\bsm(\{2\}^a, 1, \{2\}^b)) \\
        & = \qty(\sum_{\substack{0 \le a' \le a-1 \\ 0 \le b' \le b \\ a' + b' = r}} \sl(\{2\}^{a'}, 1, \{2\}^{b'})
            - \sum_{\substack{0 \le a' \le a \\ 0 \le b' \le b-1 \\ a' + b' = r}} \sl(\{2\}^{b'}, 1, \{2\}^{a'}))
            \otimes \bsm(\{2\}^{a+b-r}) \\
        & \qquad\qquad + \delta_{r \ge a}\, \bsl(\{2\}^a, 1, \{2\}^{r-a}) \otimes \bsm(\{2\}^{a+b-r}) \\
        & = \frac{(-1)^r}{2^r} \qty(\sum_{\substack{0 \le a' \le a-1 \\ 0 \le b' \le b \\ a' + b' = r}} \alpha_r^{a', b'}
            - \sum_{\substack{0 \le a' \le a \\ 0 \le b' \le b-1 \\ a' + b' = r}} \alpha_r^{b', a'})
            \zl(\overline{2r+1}) \otimes \bsm(\{2\}^{a+b-r}) \\
        & \qquad\qquad + \delta_{r \ge a}\, \bsl(\{2\}^a, 1, \{2\}^{r-a}) \otimes \bsm(\{2\}^{a+b-r}) \\
        & = \frac{(-1)^r}{2^r}(\alpha_r^{a, b} - \delta_{r \ge a}\, \alpha_r^{a, r-a})\,
            \zl(\overline{2r+1}) \otimes \bsm(\{2\}^{a+b-r}) \\
        & \qquad\qquad + \delta_{r \ge a}\, \bsl(\{2\}^a, 1, \{2\}^{r-a}) \otimes \bsm(\{2\}^{a+b-r}).
    \end{align}
    This completes the proof.
\end{prf}

\begin{lem} \label{lem:beta}
    Let $a$ and $b$ be non-negative integers.
    Then for $0 \le r \le a+b$, we have
    \[
    \alpha_r^{a, b} - \delta_{r \ge a}\, \alpha_r^{a, r-a} = \beta_r^{a, b} - \delta_{r \ge a}\, \beta_r^{a, r-a}.
    \]
\end{lem}

\begin{prf}
    By the definition of $\alpha_r^{a, b}$ and $\beta_r^{a, b}$, it suffices to show the following identity:
    \[
    - \binom{2r}{2a+1} + \delta_{r \ge a}\, \binom{2r}{2a+1} =
    \begin{dcases}
        \binom{2r}{2a} - \delta_{r \ge a}\, \binom{2r}{2a} & \text{if} \ a \ge 1, \\
        0 & \text{if} \ a = 0.
    \end{dcases}
    \]
    This follows immediately from the fact that $\binom{2r}{2a+1} = \binom{2r}{2a} = 0$ for $r < a$.
\end{prf}

\begin{prf}[of Theorem~\ref{thm:mbsidentity}]
    We proceed by induction on $a+b$. 
    The case $a=b=0$ is trivial, so we assume $a+b > 0$.
    Let $\overline{R}^{a, b}$ and $\overline{L}^{a, b}$ denote the right-hand side and the left-hand side of \eqref{eq:mbsidentity}, respectively.
    By the induction hypothesis, for $a'+b' < a+b$, we have
    \[
    \bsl(\{2\}^{a'}, 1, \{2\}^{b'}) = \frac{(-1)^{a'+b'}}{2^{a'+b'}} \beta_{a'+b'}^{a', b'}\, \zl(\overline{2a'+2b'+1}).
    \]
    Thus, for $0 \le r < a+b$, Proposition~\ref{prop:dbs} implies that
    \[
    D_{2r+1}(\overline{L}^{a, b}) = \frac{(-1)^r}{2^r} \qty(\alpha_r^{a, b} - \delta_{r \ge a}\, \alpha_r^{a, r-a} + \delta_{r \ge a}\, \beta_r^{a, r-a})
        \zl(\overline{2r+1}) \otimes \bsm(\{2\}^{a+b-r}).
    \]
    By Lemma~\ref{lem:beta}, we obtain
    \[
    D_{2r+1}(\overline{L}^{a, b}) = \frac{(-1)^r}{2^r} \beta_r^{a, b}\, \zl(\overline{2r+1}) \otimes \bsm(\{2\}^{a+b-r}).
    \]
    On the other hand, by the same argument as in the proof of 
    Theorem~\ref{thm:msidentity}, we have
    \[
    D_{2r+1}(\overline{R}^{a, b}) = \frac{(-1)^r}{2^r} \beta_r^{a, b}\, \zl(\overline{2r+1}) \otimes \bsm(\{2\}^{a+b-r}) = D_{2r+1}(\overline{L}^{a, b}).
    \]
    Hence,
    \[
    D_{<2a+2b+1}(\overline{L}^{a, b} - \overline{R}^{a, b})  = 0.
    \]
    By Theorem~\ref{thm:KerD}, this implies that there exists a constant $c \in \Q$ such that
    $\overline{L}^{a, b} - \overline{R}^{a, b} = c\, \zm(\overline{2a+2b+1})$.
    Applying the period map yields $c = 0$, which completes the proof.
\end{prf}

From now on, we simply write $\alpha^{a, b}$ and $\beta^{a, b}$ for $\alpha_{a+b}^{a, b}$ and $\beta_{a+b}^{a, b}$, respectively.
\begin{cor} \label{cor:coeff}
    Let $a$ and $b$ be non-negative integers.
    Then we have
    \begin{align}
        \sl(\{2\}^{a}, 1, \{2\}^{b}) & = \frac{(-1)^{a+b}}{2^{a+b}} \alpha^{a, b}\, \zl(\overline{2a+2b+1}), \\
        \bsl(\{2\}^{a}, 1, \{2\}^{b}) & = \frac{(-1)^{a+b}}{2^{a+b}} \beta^{a, b}\, \zl(\overline{2a+2b+1}).
    \end{align}
\end{cor}

\begin{prf}
    This follows immediately from Theorems~\ref{thm:msidentity} and \ref{thm:mbsidentity}.
\end{prf}

\begin{dfn}
    For each $N, d \ge 1$, let $\partial_{N, d}$ denote the isomorphism
    \[
    \partial_{N, d} \colon \gr_d\H_N \stackrel{\sim}{\longrightarrow} \bigoplus_{1 \le 2r+1 \le N} \gr_{d-1}\H_{N-(2r+1)} = \gr_{d-1}\H_{<N}
    \]
    induced by 
    \[
    \widetilde{\Delta} \colon \gr_d\H_N \stackrel{\sim}{\longrightarrow} \bigoplus_{1 \le 2r+1 \le N} \Q\, \zl(\overline{2r+1}) \otimes \gr_{d-1}\H_{N-(2r+1)}
    \]
    followed by sending $\zl(\overline{2r+1})$ to $(-2)^r$.
\end{dfn}

By Proposition~\ref{prop:Delta} and Corollary~\ref{cor:coeff}, we have
\begin{align}
    & \quad \partial_{N, d}(\sm(\{2\}^{l_0}, 1, \dots, 1, \{2\}^{l_d})) \\
    & = \sum_{i = 0}^{d-1} \sum_{\substack{0 \le a \le l_i-1+\delta_{i, 0} \\ 0 \le b \le l_{i+1}}}
        \alpha^{a, b}\, \sm(\{2\}^{l_0}, \dots, \{2\}^{l_{i-1}}, 1, \{2\}^{l_i+l_{i+1}-(a+b)}, 1, \{2\}^{l_{i+2}}, \dots, \{2\}^{l_d}) \\
    & \qquad - \sum_{i = 0}^{d-1} \sum_{\substack{0 \le a \le l_i \\ 0 \le b \le l_{i+1}-1}} 
        \alpha^{b, a}\, \sm(\{2\}^{l_0}, \dots, \{2\}^{l_{i-1}}, 1, \{2\}^{l_i+l_{i+1}-(a+b)}, 1, \{2\}^{l_{i+2}}, \dots, \{2\}^{l_d}) \\[-3ex]
    & \qquad + \sum_{i = 1}^{d-1} \sum_{1 \le b \le l_{i+1}} 
        \beta^{l_i, b}\, \sm(\{2\}^{l_0}, \dots, \{2\}^{l_{i-1}}, 1, \{2\}^{l_{i+1}-b}, 1, \{2\}^{l_{i+2}}, \dots, \{2\}^{l_d}) \\
    & \qquad - \sum_{i = 0}^{d-2} \sum_{1 \le a \le l_i} 
        \beta^{l_{i+1}, a}\, \sm(\{2\}^{l_0}, \dots, \{2\}^{l_{i-1}}, 1, \{2\}^{l_i-a}, 1, \{2\}^{l_{i+2}}, \dots, \{2\}^{l_d}). \label{eq:partial}
\end{align}

Let $\{1, 2\}^{\times}$ denote the set of finite words over the alphabet $\{1 ,2\}$.
For each $i \in \{1, 2\}$ and each word $w \in \{1, 2\}^{\times}$, 
we write $\deg_i w$ for the number of occurrences of $i$ in $w$.
We define the weight of $w \in \{1, 2\}^{\times}$ by 
$|w| = \deg_1 w + 2\deg_2 w$.

\begin{dfn}
    For $N \ge d \ge 1$ with $N \equiv d \pmod 2$,
    we define the sets $B_{N, d}$ and $B'_{N, d}$ by
    \begin{align}
        B_{N, d} & = \{w \in \{1, 2\}^{\times} \mid |w| = N,\; \deg_1w = d\}, \\
        B'_{N, d} & = \{w \in \{1, 2\}^{\times} \mid |w| < N,\; \deg_1w = d-1\}.
    \end{align}
    In the case $d = 1$, the set $B'_{N, d}$ also includes the empty word.
    We equip both sets with the lexicographic order induced by $1<2$.
\end{dfn}
If we write $N = 2m+d$, then clearly
\[
\# B_{N, d} = \binom{m+d}{d} = \sum_{0 \le m' \le m} \binom{m'+d-1}{d-1} = \# B'_{N, d}.
\]
\begin{exa}
    For $N = 6, d = 2$, we have
    \begin{align}
        B_{6, 2} & = \{1122, 1212, 1221, 2112, 2121, 2211\}, \\
        B'_{6, 2} & = \{1, 12, 122, 21, 212, 221\}.
    \end{align}
\end{exa}

Let $N \ge d \ge 1$ satisfy $N \equiv d \pmod 2$.
By \eqref{eq:partial}, $\partial_{N, d}(\sm(w))\ (w \in B_{N, d})$ can be written in the form
\[
\partial_{N, d}(\sm(w)) = \sum_{w' \in B'_{N, d}} f_{w'}^{w}\, \sm(w'),
\]
where $f_{w'}^{w}$ is a $\Z$-linear combination of $\alpha^{a, b}$ and $\beta^{a, b}$.

\begin{dfn}
    For $N \ge d \ge 1$ with $N \equiv d \pmod 2$,
    let $M_{N, d}$ be the matrix
    \[
    M_{N, d} = (f_{w'}^{w})_{\,w \in B_{N, d},\, w' \in B'_{N, d}}.
    \]
\end{dfn}

\begin{exa}
    For $N = 6, d = 2$, the matrix $M_{6, 2}$ is as follows.
    All blank entries are zero.
    \begin{table}[H]
        \renewcommand{\arraystretch}{1.2}
        \centering
        \begin{tabular}{|c|c|c|c|c|c|c|}
            \hline
            $ $ & 1 & 12 & 122 & 21 & 212 & 221 \\ 
            \hline 
            1122 & $\beta^{0, 2}$                 & $-(\alpha^{1, 0} - \beta^{0, 1})$ &                 &                                 &                 &                 \\
            1212 & $\beta^{1, 1}$                 & $\alpha^{0, 1}$                   &                 &                                 &                 &                 \\
            1221 & $\alpha^{0, 2}$                & $\alpha^{1, 0}$                   & $\alpha^{0, 0}$ & $\alpha^{0, 1} - \alpha^{1, 0}$ &                 &                 \\
            2112 &                                & $\alpha^{1, 0} - \beta^{0, 1}$    &                 & $\beta^{0, 1}$                  &                 &                 \\
            2121 & $\alpha^{1, 1} - \beta^{1, 1}$ &                                   &                 & $\alpha^{1, 0}$                 & $\alpha^{0, 0}$ &                 \\
            2211 & $\alpha^{2, 0} - \beta^{0, 2}$ &                                   &                 & $\alpha^{1, 0} - \beta^{0, 1}$  &                 & $\alpha^{0, 0}$ \\
            \hline 
        \end{tabular}
    \end{table}
\end{exa}

By the definition of $\alpha_r^{a, b}$ and $\beta_r^{a, b}$, 
it follows that $\alpha_r^{a, b}, \beta_r^{a, b} \in \Z_{(2)} = \{a/b \mid b: \text{odd}\}$.
\begin{lem} \label{lem:twoprop}
The elements $\alpha^{a, b}, \beta^{a, b} \in \Z_{(2)}$ satisfy the following properties\textup{:}
    \begin{enumerate}
        \item $\alpha^{a, b} -  \alpha^{b, a} \in 2\Z_{(2)},$
        \item $\alpha^{a, b} - \beta^{a, b} \in 2\Z_{(2)},$
        \item $\beta^{a, b} -  \beta^{b, a} \in 2\Z_{(2)},$
        \item $\alpha^{a, b} - \beta^{b, a} \in 2\Z_{(2)}.$
    \end{enumerate}
\end{lem}

\begin{prf}
   It suffices to prove i) and iv).
    By a direct calculation, we have
    \[
    \alpha^{a, b} -  \alpha^{b, a} = - \frac{2^{2a+2b}}{2^{2a+2b}-1} \qty(\binom{2a+2b}{2a+1} - \binom{2a+2b}{2b+1}),
    \]
    and
    \[
    \alpha^{a, b} - \beta^{b, a} = 
    \begin{dcases}
        - \frac{2^{2a+2b}}{2^{2a+2b}-1} \qty(\binom{2a+2b}{2a+1} + \binom{2a+2b}{2a}) & \text{if} \ b \ge 1, \\
        2^{2a+1} & \text{if} \ b = 0.
    \end{dcases}
    \]
    Since the right-hand sides clearly belong to $2\Z_{(2)}$, this completes the proof.
\end{prf}

We also write $c_{2^a12^b}$ for $\alpha^{a, b}$.
\begin{prop} \label{prop:decon}
    Let $w \in B_{N, d}$.
    Then the coefficient of $\sm(w') \ (w' \in B'_{N, d})$ in 
    \[
    \partial_{N, d}(\sm(w)) - \sum_{\substack{w = uv \\ \deg_1 v = 1}} c_v\, \sm(u)
    \]
    lies in $2\Z_{(2)}$.
\end{prop}

\begin{prf}
    Let $w = 2^{l_0}1\cdots12^{l_d} \in B_{N, d}$.
    We rewrite \eqref{eq:partial} as follows:
    \begin{align}
        \raisebox{6ex}[0pt][0pt]{$ \partial_{N, d}(\sm(w)) $}
        & \begin{aligned}
            & = \sum_{i = 0}^{d-1} \sum_{\substack{0 \le a \le l_i-1+\delta_{i, 0} \\ 0 \le b \le l_{i+1}}}
                \alpha^{a, b}\, \sm(2^{l_0}\cdots2^{l_{i-1}}12^{l_i+l_{i+1}-(a+b)}12^{l_{i+2}}\cdots12^{l_d}) \\
            & \quad - \sum_{i = 0}^{d-1} \sum_{\substack{0 \le a \le l_i \\ 0 \le b \le l_{i+1}-1}} 
                \alpha^{b, a}\, \sm(2^{l_0}\cdots2^{l_{i-1}}12^{l_i+l_{i+1}-(a+b)}12^{l_{i+2}}\cdots12^{l_d}) 
        \end{aligned} \tag{I}\label{eq:I} \\
        & \quad + \sum_{i = 1}^{d-1} \sum_{1 \le b \le l_{i+1}} \beta^{l_i, b}\, 
            \sm(2^{l_0}\cdots2^{l_{i-1}}12^{l_{i+1}-b}12^{l_{i+2}}\cdots12^{l_d}) \tag{II}\label{eq:II} \\
        & \quad - \sum_{i = 0}^{d-2} \sum_{1 \le a \le l_i} \beta^{l_{i+1}, a}\, 
            \sm(2^{l_0}\cdots2^{l_{i-1}}12^{l_i-a}12^{l_{i+2}}\cdots12^{l_d}). \tag{III}\label{eq:III}
    \end{align}
    \begin{align}
        \eqref{eq:I}
        & = \sum_{i = 0}^{d-1} \sum_{\substack{0 \le a \le l_i-1+\delta_{i, 0} \\ 0 \le b \le l_{i+1}-1}} (\alpha^{a, b} - \alpha^{b, a})\, 
            \sm(2^{l_0}\cdots2^{l_{i-1}}12^{l_i+l_{i+1}-(a+b)}12^{l_{i+2}}\cdots12^{l_d}) \tag{Ia}\label{eq:Ia} \\
        & \qquad - \sum_{i = 1}^{d-1} \sum_{\substack{0 \le b \le l_{i+1}-1}} \alpha^{b, l_i}\, 
            \sm(2^{l_0}\cdots2^{l_{i-1}}12^{l_{i+1}-b}12^{l_{i+2}}\cdots12^{l_d}) \tag{Ib}\label{eq:Ib} \\
        & \qquad + \sum_{i = 0}^{d-1} \sum_{\substack{0 \le a \le l_i-1+\delta_{i, 0}}} \alpha^{a, l_{i+1}}\, 
            \sm(2^{l_0}\cdots2^{l_{i-1}}12^{l_i-a}12^{l_{i+2}}\cdots12^{l_d}). \tag{Ic}\label{eq:Ic}
        \end{align}
    \begin{align}
        \eqref{eq:Ib} + \eqref{eq:II}
        & = - \sum_{i = 1}^{d-1} \sum_{\substack{0 \le b \le l_{i+1}-1}} \alpha^{b, l_i}\, 
            \sm(2^{l_0}\cdots2^{l_{i-1}}12^{l_{i+1}-b}12^{l_{i+2}}\cdots12^{l_d}) \\
        & \quad + \sum_{i = 1}^{d-1} \sum_{1 \le b \le l_{i+1}} \beta^{l_i, b}\, 
            \sm(2^{l_0}\cdots2^{l_{i-1}}12^{l_{i+1}-b}12^{l_{i+2}}\cdots12^{l_d}) \\
        & = - \sum_{i = 1}^{d-1} \sum_{\substack{1 \le r \le l_{i+1}-1}} (\alpha^{b, l_i} - \beta^{l_i, b}) \\[-1ex]
        & \qquad\qquad\qquad\qquad\qquad \times \sm(2^{l_0}\cdots2^{l_{i-1}}12^{l_{i+1}-b}12^{l_{i+2}}\cdots12^{l_d}) \tag{IIa}\label{eq:IIa} \\
        & \qquad - \sum_{i = 1}^{d-1} \alpha^{0, l_i}\, \sm(2^{l_0}1\cdots12^{l_{i-1}}12^{l_{i+1}}1\cdots12^{l_d}) \tag{IIb}\label{eq:IIb}\\
        & \qquad + \sum_{i = 1}^{d-1} \beta^{l_i, l_{i+1}}\, \sm(2^{l_0}1\cdots12^{l_{i-1}}112^{l_{i+2}}1\cdots12^{l_d}). \tag{IIc}\label{eq:IIc}
    \end{align}
    \begin{align}
        \eqref{eq:Ic} + \eqref{eq:III}
        & = \sum_{i = 0}^{d-1} \sum_{\substack{0 \le a \le l_i-1+\delta_{i, 0}}} \alpha^{a, l_{i+1}}\, 
            \sm(2^{l_0}\cdots2^{l_{i-1}}12^{l_i-a}12^{l_{i+2}}\cdots12^{l_d}) \\
        & \qquad - \sum_{i = 0}^{d-2} \sum_{1 \le a \le l_i} \beta^{l_{i+1}, a}\, 
            \sm(2^{l_0}\cdots2^{l_{i-1}}12^{l_i-a}12^{l_{i+2}}\cdots12^{l_d}) \\
        & = \sum_{i = 0}^{d-2} \sum_{1 \le a \le l_i-1+\delta_{i, 0}} (\alpha^{a, l_{i+1}} - \beta^{l_{i+1}, a}) \\[-1ex]
        & \qquad\qquad\qquad\qquad\qquad \times \sm(2^{l_0}\cdots2^{l_{i-1}}12^{l_i-a}12^{l_{i+2}}\cdots12^{l_d}) \tag{IIIa}\label{eq:IIIa} \\
        & \qquad + \sum_{i = 0}^{d-2} \alpha^{0, l_{i+1}}\, \sm(2^{l_0}1\dots12^{l_i}12^{l_{i+2}}1\dots12^{l_d}) \tag{IIIb}\label{eq:IIIb} \\
        & \qquad - \sum_{i = 1}^{d-2} \beta^{l_{i+1}, l_i}\, \sm(2^{l_0}1\dots12^{l_{i-1}}112^{l_{i+2}}1\dots12^{l_d}) \tag{IIIc}\label{eq:IIIc} \\
        & \qquad + \sum_{0 \le a \le l_{d-1}-1+\delta_{d-1, 0}} \alpha^{a, l_d}\, \sm(2^{l_0}1\dots12^{l_{d-2}}12^{l_{d-1}-a}). \tag{IIId}\label{eq:IIId}
    \end{align}
    Since the terms \eqref{eq:IIb} and \eqref{eq:IIIb} cancel each other, we obtain
    \begin{align}
        \partial_{N, d}(\sm(w))
        & = \eqref{eq:I} + \eqref{eq:II} + \eqref{eq:III} \\
        & = \eqref{eq:Ia} + \eqref{eq:Ib} + \eqref{eq:Ic} + \eqref{eq:II} + \eqref{eq:III} \\
        & = \eqref{eq:Ia} + \eqref{eq:IIa} + \eqref{eq:IIIa} + \eqref{eq:IIc} + \eqref{eq:IIIc} + \eqref{eq:IIId}.
    \end{align}
    By Lemma~\ref{lem:twoprop}, for $w' \in B'_{N, d}$, the coefficient of $\sm(w')$ in $\eqref{eq:Ia} + \eqref{eq:IIa} + \eqref{eq:IIIa}$
    lies in $2\Z_{(2)}$.
    Thus, it suffices to show that the coefficient of $\sm(w')$ in 
    \[
    \eqref{eq:IIc} + \eqref{eq:IIIc} + \eqref{eq:IIId} - \sum_{\substack{w = uv \\ \deg_1 v = 1}} c_v\, \sm(u)
    \]
    lies in $2\Z_{(2)}$.
    If $d = 1$, then $\eqref{eq:IIc} = \eqref{eq:IIIc} = 0$, so we have
    \begin{align}
    \eqref{eq:IIc} + \eqref{eq:IIIc} + \eqref{eq:IIId}
    & = \sum_{0 \le a \le l_{d-1}} \alpha^{a, l_d}\, \sm(2^{l_0}1\dots12^{l_{d-2}}12^{l_{d-1}-a}) \\
    & = \sum_{0 \le a \le l_{d-1}} c_{2^a12^{l_d}}\, \sm(2^{l_0}1\dots12^{l_{d-2}}12^{l_{d-1}-a}) \\
    & = \sum_{\substack{w = uv \\ \deg_1 v = 1}} c_v\, \sm(u).
    \end{align}
    If $d \ge 2$, we have
    \begin{align}
        \eqref{eq:IIc} + \eqref{eq:IIIc} + \eqref{eq:IIId}
        & = \sum_{i = 1}^{d-1} \beta^{l_i, l_{i+1}}\, 
            \sm(2^{l_0}1\cdots12^{l_{i-1}}112^{l_{i+2}}1\cdots12^{l_d}) \\
        & \qquad\qquad - \sum_{i = 1}^{d-2} \beta^{l_{i+1}, l_i}\, 
            \sm(2^{l_0}1\dots12^{l_{i-1}}112^{l_{i+2}}1\dots12^{l_d}) \\
        & \qquad\qquad + \sum_{0 \le a \le l_{d-1}-1} \alpha^{a, l_d}\, 
            \sm(2^{l_0}1\dots12^{l_{d-2}}12^{l_{d-1}-a}) \\
        & = \sum_{i = 1}^{d-2}(\beta^{l_i, l_{i+1}} - \beta^{l_{i+1}, l_i})\,
            \sm(2^{l_0}1\dots12^{l_{i-1}}112^{l_{i+2}}1\dots12^{l_d}) \\
        & \qquad\qquad + \beta^{l_{d-1}, l_d}\, \sm(2^{l_0}1\cdots12^{l_{d-2}}1) \\
        & \qquad\qquad + \sum_{0 \le a \le l_{d-1}-1} \alpha^{a, l_d}\, 
            \sm(2^{l_0}1\dots12^{l_{d-2}}12^{l_{d-1}-a}) \\
        & = \sum_{i = 1}^{d-2}(\beta^{l_i, l_{i+1}} - \beta^{l_{i+1}, l_i})\,
            \sm(2^{l_0}1\dots12^{l_{i-1}}112^{l_{i+2}}1\dots12^{l_d}) \\
        & \qquad\qquad - (\alpha^{l_{d-1}, l_d} - \beta^{l_{d-1}, l_d})\, \sm(2^{l_0}1\cdots12^{l_{d-2}}1) \\
        & \qquad\qquad + \sum_{0 \le a \le l_{d-1}} \alpha^{a, l_d}\, 
            \sm(2^{l_0}1\dots12^{l_{d-2}}12^{l_{d-1}-a}) \\
        & = \sum_{i = 1}^{d-2}(\beta^{l_i, l_{i+1}} - \beta^{l_{i+1}, l_i})\,
            \sm(2^{l_0}1\dots12^{l_{i-1}}112^{l_{i+2}}1\dots12^{l_d}) \\
        & \qquad\qquad - (\alpha^{l_{d-1}, l_d} - \beta^{l_{d-1}, l_d})\, \sm(2^{l_0}1\cdots12^{l_{d-2}}1) \\
        & \qquad\qquad + \sum_{\substack{w = uv \\ \deg_1 v = 1}} c_v\, \sm(u).
    \end{align}    
    Thus, the proposition follows from Lemma~\ref{lem:twoprop}.
\end{prf}

\begin{thm} \label{thm:matrix}
    For $N \ge d \ge 1$ with $N \equiv d \pmod 2$, $M_{N, d}$ is invertible.
\end{thm}

\begin{prf}
    Since $M_{N, d}$ is a matrix over $\Z_{(2)}$, it suffices to show that
    the matrix $\overline{M}_{N, d} \coloneqq M_{N, d} \bmod 2\Z_{(2)}$ is invertible.
    Consider the map
    \begin{align}
        B'_{N, d} & \longrightarrow B_{N, d} \\
        u\         & \longmapsto u12^{r},
    \end{align}
    where $r$ is the unique integer such that $|u12^{r}| = N$.
    This map is an order-preserving bijection.
    By Proposition~\ref{prop:decon}, all non-zero entries of $\overline{M}_{N, d}$
    in the column indexed by $u$ lie in the rows indexed by $u2^a12^{r-a}\ (0 \le a \le r)$.
    Since $u2^a12^{r-a} \ge u12^r$, $\overline{M}_{N, d}$ is lower-triangular.
    Moreover, its diagonal entries are of the form $c_{12^r} = \alpha^{0, r}$,
    which belongs to $1+2\Z_{(2)}$.
    Therefore, $\overline{M}_{N, d}$ is lower-triangular with all diagonal entries equal to 1,
    and hence invertible.
\end{prf}

\begin{prf}[of Theorem~\ref{thm:grbasis}]
    If $N \equiv d+1 \pmod 2$, the set
    \[
    \{\sm(k_1, \dots, k_r) \in \gr_d\H_N \mid k_i \in \{1, 2\},\; k_1+\dots+k_r = N,\; \#\{i \mid k_i = 1\} = d\}
    \]
    is empty.
    On the other hand, by the structure of $\gr_d H$ in \eqref{eq:grU}, we have $\gr_d H_N = 0$.
    Hence $\gr_d\H_N = 0$ via the isomorphism $\phi \colon \H \to H$, and there is nothing to prove.
    Thus, we may assume that $N \equiv d \pmod 2$.
    We proceed by induction on $d$.
    If $d = 0$, then $N$ is even, and we have $\gr_0\H_N = \Q\, \zm(2)^{N/2}$.
    Therefore, the assertion follows since $\sm(\{2\}^{N/2})$ is a constant multiple of $\zm(2)^{N/2}$.
    Now suppose that $d \ge 1$.
    By the induction hypothesis, the set $\{\sm(w') \mid w' \in B'_{N, d}\}$ is a basis of $\gr_{d-1}\H_{<N}$.
    Hence, Theorem~\ref{thm:matrix} implies that 
    $\{\partial_{N, d}(\sm(w)) \mid w \in B_{N, d}\}$ is a basis of $\gr_{d-1}\H_{<N}$.
    Since $\partial_{N, d}$ is an isomorphism, we conclude that 
    $\{\sm(w) \mid w \in B_{N, d}\}$ is a basis of $\gr_d\H_N$, as desired.
\end{prf}

\section{Unramified multiple $s$ values}
In this section, we study unramified multiple $s$ values.
A motivic alternating multiple zeta value is called unramified
if it is a $\Q$-linear combination of motivic multiple zeta values.
As we have seen, if all $k_i$ are greater than 1, 
\[
\sm(\kk) = 2^{\dep\kk - \wt\kk}\, \zm(\kk) \in \H^1.
\]
Glanois established a criterion for motivic alternating multiple zeta values 
to be unramified in terms of the derivations $D_r$.
\begin{thm}[{{\cite[Corollary 5.1.3]{glanois:tel-01347989}}}] \label{thm:descent}
    Let $\xi \in \H^2$ be a motivic alternating multiple zeta value.
    Then, $\xi \in \H^1$, i.e., $\xi$ is a $\Q$-linear combination
    of motivic multiple zeta values if and only if
    \[
    D_1(\xi) = 0 \quad \text{and} \quad D_{2r+1}(\xi) \in \L_{2r+1} \otimes \H^1 \ \text{for all}\ r > 0.
    \]
\end{thm}
Applying this theorem, we determine which multiple $s$ values of the form 
\[
\sm(\{2\}^{l_0}, 1, \dots, 1, \{2\}^{l_d})
\]
are unramified for $d = 1, 2$.

\begin{thm} \label{thm:unram1}
    Let $a, b \ge 0$.
    Then we have
    \[
    \sm(\{2\}^a, 1, \{2\}^b) \in \H^1 \iff b > 0.
    \]
    Moreover, for an odd integer $N = 2m+1 \ge 1$, the set
    \[
    \{\sm(\{2\}^a, 1, \{2\}^b) \mid a+b = m,\; b>0\}
    \]
    is a basis of $C_1\H^2_N\, \cap \H^1 = C_1\H^1_N$.
\end{thm}

\begin{prf}
    From the identity~\eqref{eq:smzm}, $\sm(\{2\}^d)$ belongs to $\H^1$ for all $d \ge 0$.
    Thus, by Proposition~\ref{prop:ds1}, we have
    \[
    D_{2r+1}(\sm(\{2\}^a, 1, \{2\}^b)) \in \L_{2r+1} \otimes \H^1
    \]
    for all $r \ge 0$.
    Moreover,
    \[
    D_1(\sm(\{2\}^a, 1, \{2\}^b)) = (1 - \delta_{b > 0})\, \sl(1) \otimes \sm(\{2\}^{a+b}).
    \]
    Since $\sl(1) = \zl(\overline{1})$ is sent to $\overline{f}_{\! 1}$ under the isomorphism $\L \to L$,
    we have $\sl(1) \ne 0$.
    Therefore, the first statement follows from Theorem~\ref{thm:descent}.
    By Theorem~\ref{thm:msbasis}, the set 
    \[
    \{\sm(\{2\}^a, 1, \{2\}^b) \mid a+b = m,\; b>0\}
    \]
    is linearly independent.
    Thus, to prove the second statement, it suffices to show that
    \[
    \dim C_1\H^1_N = \#\{\sm(\{2\}^a, 1, \{2\}^b) \mid a+b = m,\; b>0\} = m.
    \]
    By \cite[Theorem 1.1]{MR2993755}, we have
    \[
    \dim C_1\H^1_N = \# \{\zm(w) \mid w \in \{2, 3\}^{\times},\; |w| = N,\; \deg_3 w \le 1\} = m,
    \]
    which completes the proof.
\end{prf}

\begin{exa}
    For weight three and five, we have
    \begin{align}
        s(1, 2) & = -\frac{5}{8} \zeta(3), \\
        s(2, 1, 2) & = \frac{1}{16} \zeta(2, 3) - \frac{1}{16} \zeta(5), \\
        s(1, 2, 2) & = -\frac{1}{6} \zeta(2, 3) - \frac{29}{192} \zeta(5).
    \end{align}
\end{exa}

\begin{prop} \label{prop:ds2}
    Let $a$, $b$ and $c$ be non-negative integers.
    Then for $0 \le r \le a+b+c$, we have  
    \begin{align}
        & \quad D_{2r+1}(\sm(\{2\}^a, 1, \{2\}^b, 1, \{2\}^c)) \\
        & = \qty(\sum_{\substack{0 \le a' \le a \\ 0 \le b' \le b \\ a' + b' = r}} \sl(\{2\}^{a'}, 1, \{2\}^{b'})
                - \sum_{\substack{0 \le a' \le a \\ 0 \le b' \le b-1 \\ a' + b' = r}} \sl(\{2\}^{b'}, 1, \{2\}^{a'})
                - \sum_{\substack{1 \le a' \le a \\ b + a' = r}} \bsl(\{2\}^b, 1, \{2\}^{a'})) \\
        & \qquad\qquad\qquad \otimes \sm(\{2\}^{a+b-r}, 1, \{2\}^c) \\
        & + \qty(\sum_{\substack{0 \le b' \le b-1 \\ 0 \le c' \le c \\ b' + c' = r}} \sl(\{2\}^{b'}, 1, \{2\}^{c'})
                        - \sum_{\substack{0 \le b' \le b \\ 0 \le c' \le c-1 \\ b' + c' = r}} \sl(\{2\}^{c'}, 1, \{2\}^{b'})
                        + \sum_{\substack{1 \le c' \le c \\ b + c' = r}} \bsl(\{2\}^b, 1, \{2\}^{c'})) \\
        & \qquad\qquad\qquad \otimes \sm(\{2\}^a, 1, \{2\}^{b+c-r}).
    \end{align}
\end{prop}

\begin{prf}
    Let $\kk = (\{2\}^a, 1, \{2\}^b, 1, \{2\}^c)$ for $a, b, c \ge 0$.
    By the computation in the proof of Proposition~\ref{prop:deriv},
    $D_{2r+1}(\sm(\kk))$ consists of the five types of terms (C), (D), (F), (G), and (H).
    For the terms (C) and (D), the ranges of $i$ and $j$ are $0 \le i < j \le 2$.
    However, the case $0 = i < j = 2$ does not occur for parity reasons.
    Thus, it suffices to consider the case $j = i+1$.
    For the terms (F), (G), and (H), the ranges of $i$ and $j$ are $1 \le i < j \le 2$,
    $0 \le i < j \le 1$, and $1 \le i \le j \le 1$, respectively.
    These are precisely the cases treated in the proof of Proposition~\ref{prop:Delta}.
    Therefore, we obtain the desired result by extracting the terms 
    whose left tensor components are of weight $2r+1$.
\end{prf}

\begin{thm} \label{thm:unram2}
    Let $a, b, b' \ge 0,\; c, c' > 0$ with $b + c = b' + c'$.
    Then we have
    \[
    \beta^{b', c'}\, \sm(\{2\}^a, 1, \{2\}^b, 1, \{2\}^c) - \beta^{b, c}\, \sm(\{2\}^a, 1, \{2\}^{b'}, 1, \{2\}^{c'}) \in \H^1.
    \]
    Moreover, for an even integer $N = 2m+2 \ge 2$, the set
    \[
    \begin{aligned}
        & S = \{\sm(\{2\}^{m+1})\}\; \cup \\
        & \left\{
        \beta^{b, c}\, \sm(\{2\}^a, 1, 1, \{2\}^{m-a}) - \beta^{0, m-a}\, \sm(\{2\}^a, 1, \{2\}^b, 1, \{2\}^c)
        \ \middle\vert\;
        \begin{aligned}
            & a+b+c = m, \\
            & a \ge 0,\; b, c > 0
        \end{aligned}
        \right\}
    \end{aligned}
    \]
    is a basis of $C_2\H^2_N\, \cap\, \H^1 = C_2\H^1_N$.
\end{thm}

\begin{prf}
    Let $a, b, b' \ge 0,\; c, c' > 0$ with $b + c = b' + c'$, and set
    \[
    \xi = \beta^{b', c'}\, \sm(\{2\}^a, 1, \{2\}^b, 1, \{2\}^c) - \beta^{b, c}\, \sm(\{2\}^a, 1, \{2\}^{b'}, 1, \{2\}^{c'}).
    \]
    To prove the first statement, it suffices to show that 
    \[
    D_1(\xi) = 0 \quad \text{and} \quad D_{2r+1}(\xi) \in \L_{2r+1} \otimes \H^1 \ \text{for all}\ r > 0.
    \]
    Let $\kk = (\{2\}^a, 1, \{2\}^b, 1, \{2\}^c)$.
    By Proposition~\ref{prop:ds2}, we have
    \[
    D_1(\sm(\kk)) = (1 - \delta_{b>0})\, \sl(1) \otimes (\sm(\{2\}^{a+b}, 1, \{2\}^c) - \sm(\{2\}^a, 1, \{2\}^{b+c})).
    \]
    This vanishes regardless of whether $b > 0$.
    Hence,
    \[
    D_1(\xi) = 0.
    \]
    By Proposition~\ref{prop:ds2} and Theorem~\ref{thm:unram1},
    it follows that
    \[
    D_{2r+1}(\sm(\kk)) -
    \delta_{r, b+c}\, \bsl(\{2\}^b, 1, \{2\}^c) \otimes \sm(\{2\}^a, 1)
    \in \L_{2r+1} \otimes \H^1.
    \]
    By Corollary~\ref{cor:coeff}, we have
    \[
    \beta^{b', c'}\, \bsl(\{2\}^b, 1, \{2\}^c) - \beta^{b, c}\, \bsl(\{2\}^{b'}, 1, \{2\}^{c'}) = 0,
    \]
    which implies
    \[
    D_{2r+1}(\xi) \in \L_{2r+1} \otimes \H^1 \ \text{for all}\ r > 0.
    \]
    Since $\beta^{0, m-a} \ne 0$ for all $a$, Theorem~\ref{thm:msbasis}
    implies that the set $S$ is linearly independent.
    Thus, to prove the second statement, it suffices to show that
    \[
    \dim C_2\H^1_N = \# S = 1 + \binom{m}{2}.
    \]
    By \cite[Theorem 1.1]{MR2993755}, we have
    \[
    \dim C_2\H^1_N = \# \{\zm(w) \mid w \in \{2, 3\}^{\times},\; |w| = N,\; \deg_3 w \le 2\} = 1 + \binom{m}{2},
    \]
    which completes the proof.
\end{prf}

\begin{cor}
    Let $a \ge 0,\; b, c > 0$.
    Then 
    \[
    s(\{2\}^a, 1, \{2\}^b, 1, \{2\}^c) - s(\{2\}^a, 1, \{2\}^c, 1, \{2\}^b)
    \]
    is a $\Q$-linear combination of multiple zeta values.
\end{cor}

\begin{prf}
    This follows immediately from Theorem~\ref{thm:unram2}, and $\beta^{b, c} = \beta^{c, b} \ne 0$ for $b, c > 0$.
\end{prf}

\bibliographystyle{my_amsalpha}
\bibliography{reference}

\end{document}